\documentclass{siamonline171218}

\usepackage{xspace}
\usepackage{amsmath}
\usepackage{mathrsfs}
\usepackage{algorithmic}
\usepackage{enumerate}
\usepackage{multirow}
\usepackage{diagbox}
\usepackage{subfig}
\usepackage{acro}
\usepackage{booktabs}
\usepackage{cleveref}
\usepackage{newtxmath}

\usepackage{lipsum}
\usepackage{amsfonts}
\usepackage{epstopdf}
\ifpdf
  \DeclareGraphicsExtensions{.eps,.pdf,.png,.jpg}
\else
  \DeclareGraphicsExtensions{.eps}
\fi

\usepackage{enumitem}
\setlist[enumerate]{leftmargin=.5in}
\setlist[itemize]{leftmargin=.5in}

\newsiamremark{remark}{Remark}
\newsiamremark{hypothesis}{Hypothesis}
\crefname{hypothesis}{Hypothesis}{Hypotheses}
\newsiamthm{claim}{Claim}

\usepackage{amsopn}

\newtheorem{assumption}{Assumption}

\DeclareMathOperator*{\diag}{\texttt{diag}}
\DeclareMathOperator{\Id}{Id} % Identity operator

\newcommand{\domain}{\Omega} 
\newcommand{\Cdot}{\,\cdot\,}
\newcommand{\bx}{\boldsymbol{x}}
\newcommand{\by}{\boldsymbol{y}}
\newcommand{\bz}{\boldsymbol{z}}
\newcommand{\bdf}{\boldsymbol{f}}
\newcommand{\bdg}{\boldsymbol{g}}
\newcommand{\bda}{\boldsymbol{a}}
\newcommand{\bdb}{\boldsymbol{b}}
\newcommand{\bdu}{\boldsymbol{u}}
\newcommand{\bdv}{\boldsymbol{v}}
\newcommand{\bds}{\boldsymbol{s}}
\newcommand{\bdw}{\boldsymbol{w}}
\newcommand{\bdc}{\boldsymbol{c}}
\newcommand{\DistFunc}{\mathcal{D}}
\newcommand{\RegFunc}{\mathcal{R}}
\newcommand{\bdbeta}{\boldsymbol{\beta}}
\newcommand{\bdomega}{\boldsymbol{\omega}}
\newcommand{\proj}{\texttt{Proj}}
\newcommand{\prox}{\mathrm{\texttt{Prox}}}
\newcommand{\bdK}{\boldsymbol{K}}
\newcommand{\whv}{\widehat{\boldsymbol{v}}}
\newcommand{\whu}{\widehat{\boldsymbol{u}}}
\newcommand{\whf}{\widehat{\boldsymbol{f}}}
\newcommand{\wtH}{\widetilde{H}}
\newcommand{\tra}{\boldsymbol{\mathsf{T}}}
\newcommand{\ctra}{\boldsymbol{\mathsf{H}}}
\newcommand{\Real}{\mathbb{R}}
\newcommand{\one}{\vmathbb{1}}
\newcommand{\two}{\vmathbb{2}}

\graphicspath{{figures/}}
\ifpdf
  \DeclareGraphicsExtensions{.eps,.pdf,.png,.jpg}
\else
  \DeclareGraphicsExtensions{.eps}
\fi

\crefname{line}{step}{steps}
\Crefname{line}{Step}{Steps}

\numberwithin{theorem}{section}

\DeclareAcronym{TV}{
  short = TV,
  long = total variation}
\DeclareAcronym{KL}{
  short = KL,
  long = Kullback--Leibler}
\DeclareAcronym{CT}{
  short = CT,
  long = computed tomography}
  \DeclareAcronym{ART}{
  short = ART,
  long = algebraic reconstruction technique}
  \DeclareAcronym{EART}{
  short = EART,
  long = extended algebraic reconstruction technique}
   \DeclareAcronym{DBT}{
  short = DBT,
  long = digital breast tomosynthesis}
   \DeclareAcronym{PDHGM}{
  short = PDHGM,
  long = primal-dual hybrid gradient method}
  \DeclareAcronym{l.s.c}{
  short = l.s.c.,
  long = lower semicontinuous}
  \DeclareAcronym{ADMM}{
  short = ADMM,
  long = alternating direction method of multipliers}

\newcommand{\TheTitle}{An Extended Primal-Dual Algorithm Framework for Nonconvex Problems with Application to Nonlinear Imaging} 
\newcommand{\TheAuthors}{Yu Gao, Xiaochuan Pan, Chong Chen}

\headers{Extended Primal-Dual Algorithm Framework}{Y. Gao, X. Pan, and C. Chen}

\title{{\TheTitle}
}

\author{
  Yu Gao\thanks{LSEC, ICMSEC, Academy of Mathematics and Systems Science, Chinese Academy of Sciences, Beijing 100190, China. School of Mathematical Sciences, University of Chinese Academy of Sciences, Beijing 100049, China.}
  \and
  Xiaochuan Pan\thanks{Department of Radiology, The University of Chicago, Chicago, IL 60637, USA.}
  \and
  Chong Chen\thanks{LSEC, ICMSEC, Academy of Mathematics and Systems Science, Chinese Academy of Sciences, Beijing 100190, China. School of Mathematical Sciences, University of Chinese Academy of Sciences, Beijing 100049, China.}
}

\ifpdf
\hypersetup{
  pdftitle={\TheTitle},
  pdfauthor={\TheAuthors}
}
\fi

\begin{document}

\maketitle

% REQUIRED
\begin{abstract}
We propose an extended primal-dual algorithm framework for solving a general nonconvex optimization model. This work is motivated by image reconstruction problems in a class of nonlinear imaging, where the forward operator can be formulated as a nonlinear convex function with respect to the reconstructed image. Using the proposed framework, we put forward six specific iterative schemes, and present their detailed mathematical explanation. We also establish the relationship to existing algorithms. Moreover, under proper assumptions, we analyze the convergence of the schemes for the general model when the optimal dual variable regarding the nonlinear operator is non-vanishing. As a representative, the image reconstruction for spectral computed tomography is used to demonstrate the effectiveness of the proposed algorithm framework. By special properties of the concrete problem, we further prove the convergence of these customized schemes when the optimal dual variable regarding the nonlinear operator is vanishing. Finally, the numerical experiments show that the proposed algorithm has good performance on image reconstruction for various data with non-standard scanning configuration. 
%Contribution: The main novel scientific contribution: (1) We propose an extended primal-dual algorithm framework for solving a general nonconvex optimization model, which is motivated by image reconstruction problems in a class of nonlinear imaging where the forward operator can be formulated as a nonlinear convex function regarding the reconstructed image. (2) Using the proposed framework, we present six specific iterative schemes, and give their complete mathematical explanation, and also establish the relationship to existing algorithms. (3) Under proper assumptions, we analyze the convergence of the schemes for the general model when the optimal dual variable regarding the nonlinear operator is non-vanishing. (4) Without loss of generality, the image reconstruction for various data with non-standard scanning configuration in spectral computed tomography is used to demonstrate the effectiveness of the proposed algorithm framework. By special properties of the concrete problem, we further prove the convergence of these customized schemes when the optimal dual variable regarding the nonlinear operator is vanishing.
\end{abstract}


% REQUIRED

\section{Introduction}\label{sec:intro}

We consider image reconstruction problems given as nonlinear operator equations    
\begin{equation}\label{eq:nl_op_eq}
\bdg  = \bdK(\bdf) + \bdg_{\text{noise}},
\end{equation}
where $\bdK : \mathcal{X} \rightarrow \mathcal{Y}_K$ denotes a Fr\'echet differentiable nonlinear operator,  
$\bdf \in \domain \subset \mathcal{X}$ is the image to be reconstructed, $\bdg \in \mathcal{Y}_K$ is  
measured data, and $\bdg_{\text{noise}} \in \mathcal{Y}_K$ is uncertain noise.  
Here $\mathcal{X}$ and $\mathcal{Y}_K$ are two finite-dimensional real vector spaces equipped with an inner 
product $\langle\cdot,\cdot\rangle$ and norm $\|\cdot\| = \sqrt{\langle\cdot,\cdot\rangle}$, 
and $\domain$ is a closed convex subset of $\mathcal{X}$. 

Specifically, in spectral \ac{CT}, for instance, the component of forward projection operator can be 
formulated as the form of log-sum exponential functional via basis material decomposition \cite{alma76,maal76},
\begin{equation}\label{eq:log_sum_exp}
K_j(\bdf)  := \ln\sum_{m=1}^{M} s_{jm} \exp\Biggl(-\sum_{d=1}^{D} b_{dm}\bda_j^{\tra} \bdf_{\!\!d} \Biggr),  
\end{equation}
where $\bdK := [K_1, \ldots, K_J]^{\tra}$, $s_{jm}$ is the discrete 
normalized X-ray spectrum at energy bin $m$ for ray $j$ with $\sum_{m=1}^Ms_{jm} = 1$, $b_{dm}$ is the decomposition coefficient 
at energy bin $m$ of the $d$-th basis material, $\bdf_{\!\!d}$ is the unknown density (image) of the $d$-th basis material,  
$\bda_j$ is the discrete X-ray transform determined by ray $j$, and $\tra$ denotes transpose. For \ac{DBT}, the functional of 
form \cref{eq:log_sum_exp} is applied as the component of forward projection operator as well \cite{lpna17}. 
Moreover, for the discrete X-ray transform with nonlinear partial volume effect \cite{chpan20}, the component of $\bdK$ 
can be represented as 
\begin{equation}\label{eq:NPVE}
    K_j(\bdf) := \ln \frac{1}{N_j} \sum_{l_j=1}^{N_j} \exp\bigl( -\bda_{l_j}^{\tra}\bdf \bigr),
\end{equation}
where $N_j$ denotes the considered total number of rays between detector bin $j$ and focal spot. 
On the other hand, the component of $\bdK$ in phase retrieval can be written as 
\begin{equation}\label{eq:PhaseR}
    K_j(\bdf) := \vert\langle \bda_j, \bdf \rangle\vert^2 =  \bdf^{\ctra}\bda_j\bda_j^{\ctra}\bdf, 
\end{equation}
where $\bdf$ and $\bda_j$ are complex vectors, and $\ctra$ denotes conjugate transpose \cite{WaXu19}.  

The image reconstruction problem of interest is to determine unknown $\bdf$ from measured 
data $\bdg$ as in \cref{eq:nl_op_eq}, where the data is often sparse sampled and/or contains uncertain noise. 
%It is well-known that problems of the form \cref{eq:nl_op_eq} are generally nonlinear ill-posed 
%problems \cite{kanesch08}.  
%where $\bdg = \left[g_1, \ldots, g_{J}\right]^{\tra} \in \Real^{J}$ and $\bdK(\bdf) = \left[K_1(\bdf), \ldots, K_{J}(\bdf)\right]^{\tra}$. 
Generally, solving nonlinear system of the form   
\begin{equation}\label{eq:ill_posed_prob}
\bdK(\bdf) = \bdg 
\end{equation}
is equivalent to solve an ill-posed inverse problem since the solutions of \cref{eq:ill_posed_prob} do not 
depend continuously on the data \cite{kanesch08}. An \ac{EART} was proposed to solve \cref{eq:ill_posed_prob} for image reconstruction 
in dual spectral CT (DECT) in \cite{zhaozz14}, which is actually a generalization of \ac{ART} or Kaczmarz method for 
solving linear system to the nonlinear case in \cref{eq:ill_posed_prob}. Similarly, a stochastic gradient descent 
method was employed and analyzed to solve \cref{eq:ill_posed_prob} in \cite{sgd_for_ill_posed}. The authers in \cite{lpna17} 
reformulated the image reconstruction problem in \cref{eq:ill_posed_prob} for \ac{DBT} as a nonlinear 
least square problem, and proposed a limited memory BFGS method. 
Moreover, there are several iterative regularization methods that were developed to solve \cref{eq:ill_posed_prob}, 
such as nonlinear Landweber iteration, and Newton type methods (see \cite{FaYu05,kanesch08} 
and the references therein). However, few of these literatures considered a regularization term, 
let alone a nonsmooth one. 

The image reconstruction problem in \cref{eq:nl_op_eq} can be reformulated as the following constrained 
optimization problem 
\begin{equation}\label{eq:opt}
 \mathop{\min}_{\bdf \in \domain} \Bigl\{\DistFunc\bigl(\bdK(\bdf), \bdg\bigr) + \RegFunc(\bdf)\Bigr\}. 
\end{equation}
Here $\DistFunc(\cdot, \cdot)$ denotes the data fidelity term which is often designed as the form 
of $\ell_p$-norm or \ac{KL} divergence,
and $\RegFunc(\cdot)$ stands for the regularization term which is 
frequently schemed as $\ell_1$-norm of the sparse representation of the image under certain basis 
functions \cite{El08}.  
%the functional of the gradient magnitude of
%the unknown image, such as Tikhonov functional, \ac{TV} functional. 
%For the case with sparse sampled data, the regularization term is usually presented as the $\ell_1$-norm 
%of the spare representation of the image under certain basis functions. 
Generally, it can be written as  
\begin{equation}\label{eq:sparse_rep}
 \RegFunc(\bdf) := \lambda \phi(A \bdf),  
 %= \lambda  \sum_{d=1}^D\|W_d \bdf_{\!\!d}\|_1,
\end{equation}
where 
%$W = \diag\bigl(W_1, \ldots, W_D\bigr)$, and 
$A : \mathcal{X} \rightarrow \mathcal{Y}_A$ is a continuous linear operator and denotes the associated discrete 
form of the sparse representation, $\mathcal{Y}_A$ is also a finite-dimensional 
real vector space, and $\lambda$ is a nonnegative regularization parameter. Hence, a general nonconvex 
optimization model can be represented as minimizing a sum of three terms of the form
\begin{equation} \label{eq:opt_primal}
\min_{\bdf \in \mathcal{X}} \Bigl\{F\bigl(\bdK(\bdf)\bigr) + E(A\bdf) + G(\bdf)\Bigr\},  
\end{equation} 
where each component of $\bdK$ and $A$ are nonlinear and linear with regard to $\bdf$ respectively. 
Note that we do not merge the first two terms in \cref{eq:opt_primal} into a single one as the model 
studied in \cite{va14}, which can help us to focus on the specific term regarding the nonlinear forward operator. 
For example in \cref{eq:opt}, we can let $F(\Cdot) := \DistFunc(\Cdot,\bdg)$, $E(\Cdot) := \lambda \phi(\Cdot)$ 
by \cref{eq:sparse_rep}, and 
\begin{align*}
G(\bdf) := \delta_{\domain}(\bdf) = \begin{cases} 
     0 & \text{if}~\bdf\in \domain,  \\
    +\infty &  \text{otherwise}. 
   \end{cases}
\end{align*}

Model \cref{eq:opt_primal} can be also seen as a generalization of the following well-studied convex  
optimization problem (see \cite{eszhch10,chpo11,heyu2012})
\begin{equation} \label{eq:convex_opt_primal}
\min_{\bdf \in \mathcal{X}} \Bigl\{E(A\bdf) + G(\bdf)\Bigr\}. 
\end{equation} 
If $\bdK$ is linear, model \cref{eq:opt_primal} boils down to model \cref{eq:convex_opt_primal}. 
Even though both $F$ and $\bdK$ are nonlinear convex, their composition $F \circ \bdK$ is 
not definitely convex. It is difficult to solve the nonconvex problem as in \cref{eq:opt_primal}. 
%The former generally implies $F\bigl(\bdK(\bdf)\bigr)$ nonconvex. 
%\cref{eq:opt_primal} also appears in many other applications, there are several methods to solve the convex case of \cref{eq:opt_primal}, for instance, the block coordinate descent(BCD) \cite{wright2015coordinate}, and alternating direction method of multipliers (ADMM) \cite{ADMM14}, and the Chambolle--Pock method \cite{chpo11}. Following \cite{chpo11}, by Lengendre transform (see \cref{def:Lengendretransform}), the primal problem \cref{eq:opt_primal} can be adapted into the saddle-point problem
%\begin{align}\label{eq:mini_max_problem}
%    \min_{\bdf}\max_{\bdu,\bdv} G(\bdf) + \langle K(\bdf), \bdu\rangle +\langle A\bdf,\bdv \rangle - F^{\ast}(\bdu)-E^{\ast}(\bdv), %\tag{mini-max},
%\end{align}
%where $\bdu$ and $\bdv$ denote the dual variables. However, $\langle K(\cdot), \bdu\rangle$ still not be convex for fixed dual variable $\bdu$ which caused by uncertain sign of components $u_j$. For the case where $K$ is linear, a primal-dual method \cite{chpo11} was put forward to solve \cref{eq:mini_max_problem}, which has been proven convergence concurrently. \cite{heyu2012} and \cite{valkonen2020testing} also give convergence analysis from different perspective. 
However, there are already some studies on how to solve such kind of problems. 
For instance, an exact/linearized nonlinear \ac{PDHGM} was proposed and analyzed 
in \cite{va14,accle_Val19}, which is a natural extension of the modified \ac{PDHGM} for 
solving \cref{eq:convex_opt_primal} the linear case in \cite{chpo11}. Particularly, the authors in \cite{chpan20,chpan21} 
proposed an interesting nonconvex prinal-dual algorithm for image reconstruction in 
spectral \ac{CT} and discrete X-ray transform with nonlinear partial volume effect. A nonconvex \ac{ADMM} 
was proposed in \cite{BaSi21} to solve the nonconvex and possibly nonsmooth optimization 
problem with application to \ac{CT} imaging (see \cite{WaYiZe19}). 
But very few of existing literatures excavate the properties of forward operator $\bdK$.

In this work, we focus on solving a class of nonconvex problems in \cref{eq:opt_primal} with operator 
$\bdK$ being nonlinear convex with respect to $\bdf$, which is motivated by several practical 
nonlinear imaging problems with such kind of forward operators. 
For example, it is easy to verify that $K_j$ in \cref{eq:log_sum_exp}--\cref{eq:PhaseR}  
are convex with respect to $\bdf$. The operator convexity would be useful to develop a framework of 
extended primal-dual algorithm for these nonconvex problems. 

%\paragraph{Contributions}
%In summary, our contributions are as following
%\begin{itemize}
%    \item Based on the properties of Legendre transform and the strategy of alternately updating primal and dual variables, a framework of extended primal-dual algorithm is proposed, which produce new schemes and existing schemes.
%    \item The convergence of the proposed schemes is proved mathematically under the new assumptions where $G$ is only convex and $K$ satisfies some nonlinear restrict.
%    \item Convergence proof of constrained nonlinear least square in reconstruction problem using imaging model  \cref{eq:log_sum_exp} is also given by its special properties. Numerical experiments are designed for spectral CT reconstruction to verify the effectiveness of the algorithm and our theoretical results.
%\end{itemize}

%\paragraph{Outline}
The outline of this paper is organized as follows. \Cref{sec:preliminary} introduces the required mathematical
preliminaries. We propose a framework of extended primal-dual algorithm for nonconvex problems, 
and establish the relationship to existing algorithms in \cref{sec:proposed_framework}. 
The associated convergence is analyzed in \cref{sec:convergence_analysis}. 
In \cref{sec:specific_algorithms}, we apply the proposed algorithm framework to reconstruct image in   
spectral \ac{CT}, and also give the customized analysis and several numerical experiments. 
Finally, the paper is concluded by \cref{sec:conclusion}.

\section{Preliminaries}\label{sec:preliminary}
%In this section, we give several specific settings, notation and properties of convex function. 
%These properties will be used to derive our proposed algorithms from extended 
%primal-dual framework and prove the convergence of proposed algorithms in the following sections.

%Let $\mathcal{X}$ and $\mathcal{Y}_K,\mathcal{Y}_A$ be real finite-dimensional Euclidean space equipped with an inner product $\langle\cdot,\cdot\rangle$ and Euclidean norm $\|\cdot\|$. Suppose $G$ and the conjugates(see \cref{def:conjugate}) $F^{\ast},E^{\ast}$ are proper lower semicontinuous convex functions in $\mathcal{X}$ and $\mathcal{Y}_K,\mathcal{Y}_A$ respectively, let $K:\mathcal{X}\rightarrow \mathcal{Y}_K$ be a nonlinear continuous differentiable operator whose every component $K_j$ denotes a nonlinear proper convex function in $\mathcal{X}$. Let $A:\mathcal{X}\rightarrow \mathcal{Y}_A$ be a bounded linear operator.

%Denote $\langle\bx,\by\rangle := \bx^{\tra}\by$ and $\|\bx\| := \sqrt{\langle\bx,\bx\rangle}$ the inner product 
%and norm of Euclidean space, respectively. 
We introduce Legendre--Fenchel transform first.
\begin{definition}(\cite[Legendre--Fenchel transform]{rowe04})\label{def:Lengendretransform}
For any function $h: \mathcal{X} \rightarrow (-\infty,+\infty]$, function $h^{\ast} : \mathcal{X} \rightarrow (-\infty,+\infty]$ defined by
\begin{equation}\label{def:conjugate}
    h^{\ast}(\bdw) := \sup_{\bdf}\Bigl\{ \langle \bdf, \bdw \rangle - h(\bdf)  \Bigr\} 
\end{equation}
is conjugate to $h$, while function $h^{\ast\ast} =(h^{\ast})^{\ast}$ defined by
\begin{equation}\label{def:biconjugate}
    h^{\ast\ast}(\bdf):= \sup_{\bdw}\Bigl\{ \langle \bdf, \bdw \rangle - h^{\ast}(\bdw)  \Bigr\}
\end{equation}
is biconjugate to $h$. The mapping $h \mapsto h^{\ast}$ is Legendre--Fenchel transform.
\end{definition}

%The biconjugate funtion is always a lower bound on the original function, 
Note that for proper, \ac{l.s.c}, convex function, its biconjugate function equals to itself, i.e., $h^{\ast\ast}( \bdf)=h( \bdf)$. 
Moreover, we have the following significant results. 
\begin{proposition}(\cite[Inversion rule for subgradient relations]{rowe04})\label{prop:inversionrule}
For any proper, \ac{l.s.c}, convex function $h$, one has $\partial h^{\ast}=(\partial h)^{-1}$ 
and $\partial h=(\partial h^{\ast})^{-1}$. 
%Indeed,
% \begin{align*}
%     \bar{\bdw}\in \partial h(\bar{ \bdf})\Longleftrightarrow \bar{ \bdf}\in \partial h^{\ast}(\bdw)
% \end{align*}
Moreover, if $h$ is Fr\'echet differentiable, we have $\bar{\bdf}_{\!h}(\bdw) = \nabla h^{\ast}(\bdw)$ 
and $\bar{\bdw}_{\!h}(\bdf) = \nabla h(\bdf)$, where $\bar{\bdf}_{\!h}$ and $\bar{\bdw}_{\!h}$ are the optimal solutions to  
the right-hand sides in \cref{def:conjugate} and \cref{def:biconjugate}, respectively.
\end{proposition}

Subsequently, we introduce proximal mapping and its firm nonexpansivity property, 
which would be useful for the later convergence analysis.
\begin{definition}(\cite[Proximal mapping]{be17})\label{def:proximal}
Given function $h : \mathcal{X} \rightarrow (-\infty,+\infty]$, the proximal mapping of $h$ is an operator given by 
\begin{equation*}
    \prox_{h}(\bx):= \arg\min_{\by}\left\{h(\by) + \frac{1}{2}\|\by - \bx\|^2\right\} 
\end{equation*}
for any $\bx \in \mathcal{X}$.
\end{definition}
\begin{remark}
Let $\tau>0$. If $h$ is proper, convex and \ac{l.s.c}, then $\prox_{\tau h}(\bx)$ exists uniquely. 
Moreover, it is equivalent to the resolvent of subgradient, which is 
\begin{equation}\label{eq:resolvent}
    \prox_{\tau h}(\bx) = (\Id + \tau\partial h)^{-1}(\bx), 
    %= \arg\min_{\by}\left\{h(\by) + \frac{1}{2\tau}\|\by - \bx\|^2\right\}, 
\end{equation}
where $\Id$ denotes identity operator (matrix) of the associated vector space. 
\end{remark}
\begin{proposition}(\cite[Firm nonexpansivity of the proximal operator]{be17}) \label{nonexpan of prox}
Let $h$ be a proper closed and convex function. Then, for and $\bx, \by \in \mathcal{X}$,
\begin{align*}
   \| \normalfont \prox_{h}(\bx) - \prox_{h}(\by)\|^2 &\le \normalfont\langle \bx - \by,\prox_{h}(\bx) - \prox_{h}(\by) \rangle, \\
    \normalfont \|\prox_{h}(\bx) - \prox_{h}(\by)\|&\le \|\bx-\by\|.
\end{align*}
\end{proposition}

\section{The proposed extended primal-dual algorithm framework}
\label{sec:proposed_framework}

In this section, we propose a framework of extended primal-dual algorithm for 
the nonconvex problems in \cref{eq:opt_primal}, where each component $K_j$ of $\bdK$ is a proper Fr\'echet  
differentiable nonlinear convex function in $\mathcal{X}$.  

\subsection{The algorithm framework and generated schemes}

Let $\overline{\Real}_{+} := [0, +\infty]$. Assume that $G : \mathcal{X} \rightarrow \overline{\Real}_{+}$, 
$F^{\ast} : \mathcal{Y}_K \rightarrow \overline{\Real}_{+}$ 
and $E^{\ast} : \mathcal{Y}_A \rightarrow \overline{\Real}_{+}$ are proper, convex, \ac{l.s.c} functions, and $F^{\ast}$, 
$E^{\ast}$ are themselves the convex conjugates of convex \ac{l.s.c} functions $F$, $E$, respectively. 
By Legendre--Fenchel transform in \cref{def:Lengendretransform}, the general problem in \cref{eq:opt_primal} can 
be translated into the following saddle-point problem
\begin{equation}\label{eq:mini_max_problem}
    \min_{\bdf \in \mathcal{X}}\max_{\bdu \in \mathcal{Y}_K, \bdv \in \mathcal{Y}_A} \Bigl\{\langle \bdK(\bdf), \bdu\rangle +\langle A\bdf,\bdv \rangle - F^{\ast}(\bdu)-E^{\ast}(\bdv) + G(\bdf) \Bigr\},
\end{equation}
where $\bdu$ and $\bdv$ are dual variables regarding the nonlinear and linear operators, respectively.

%then some new algorithms can are derived
%, and existing algorithms in \cite{va14} are also deduced. 

With the convexity of $K_j$, by \cref{def:Lengendretransform} again, we obtain 
\begin{equation}\label{eq:conjugate_kj}
    K_j(\bdf) =\max_{\bdw_j}\Bigl\{ \langle \bdf,\bdw_j\rangle - K_j^{\ast}(\bdw_j)\Bigr\}, 
\end{equation}
where $K_j^{\ast}$ is the convex conjugate of $K_j$. Let $\bar{\bdw}_j(\bdf)$ be the optimal solution 
of \cref{eq:conjugate_kj}. Then, \cref{eq:mini_max_problem} becomes 
\begin{align}\label{def:Phi}
\min_{\bdf \in \mathcal{X}}\max_{\bdu \in \mathcal{Y}_K, \bdv \in \mathcal{Y}_A} \Phi\bigl(\bdf,\bar{\bdw}(\bdf),\bdu,\bdv\bigr), 
\end{align}
where 
\begin{equation*}
\Phi\bigl(\bdf,\bar{\bdw}(\bdf),\bdu,\bdv\bigr) := \sum_{j} u_j\left\{ \langle \bdf,\bar{\bdw}_j(\bdf)\rangle -K_j^{\ast}\bigl(\bar{\bdw}_j(\bdf)\bigr)\right\} +\langle A\bdf,\bdv\rangle -F^{\ast}(\bdu) -E^{\ast}(\bdv) + G(\bdf),
\end{equation*}
and $\bar{\bdw} = [\bar{\bdw}_1, \ldots, \bar{\bdw}_J]$. 
%$\bar{\bdw}_{j} := \bar{\bdw}_{j}(\bdf)$.
Notice that there are many variables in \cref{def:Phi}, which is difficult to solve simultaneously. 
Thus, it is natural to fix some variables and alternatively solve the associated subproblems. 
Particularly, we need to update $\bdf$ and $\bar{\bdw}(\bdf)$ alternatively for the existing nonlinearity.  
The proposed framework of extended primal-dual algorithm is presented as the 
following \cref{algo:proposed_NLPD_algorithm}.  
\begin{algorithm}[htbp]
\caption{The proposed extended primal-dual algorithm framework}
\label{algo:proposed_NLPD_algorithm}
\begin{algorithmic}[1]
\STATE \emph{Initialize}: Given $\tau > 0$, $\sigma_{\!\!K} > 0$, $\sigma_{\!\!A} > 0$, $\theta \in [0, 1]$, and initial point $(\bdf^0, \bar{\bdw}^0, \bdu^0, \bdv^0)$ with 
$\bar{\bdw}^0 = \bar{\bdw}(\bdf^0)$. Let $\bdf_{\theta}^0 \gets \bdf^0$ and $n \gets 0$. 
\STATE \emph{Loop}: Update $(\bdf^{n+1}, \bar{\bdw}^{n+1}, \bdu^{n+1}, \bdv^{n+1})$ by 
\begin{subequations}
\begin{align}
\label{eq:update_f} \bdf^{n+1} &= \arg\min_{\bdf \in \mathcal{X}}\left\{\Phi(\bdf, \bar{\bdw}^n, \bdu^n,\bdv^n ) + \frac{1}{2\tau}\|\bdf - \bdf^{n}\|^2\right\}, \\
\label{eq:accelerate_f} \bdf_{\theta}^{n+1} &= \bdf^{n+1} + \theta(\bdf^{n+1}  - \bdf^n), \\
\label{eq:update_w} \bar{\bdw}^{n+1} &= \bar{\bdw}(\bar{\bdf}^{n+1})\quad \text{with} ~ \bar{\bdf}^{n+1} = \bdf_{\theta}^{n+1},~\bdf^{n+1}~\text{or}~\bdf^n, \\ 
\label{eq:update_u} \bdu^{n+1} &= \arg\min_{\bdu \in \mathcal{Y}_K}\left\{-\Phi(\bdf_{\theta}^{n+1}, \bar{\bdw}^{n+1}, \bdu,\bdv^n) + \frac{1}{2\sigma_{\!\!K}}\|\bdu - \bdu^{n}\|^2\right\},\\
\label{eq:update_v} \bdv^{n+1} &= \arg\min_{\bdv \in \mathcal{Y}_A}\left\{-\Phi(\bdf_{\theta}^{n+1}, \bar{\bdw}^{n+1}, \bdu^{n+1},\bdv) + \frac{1}{2\sigma_{\!\!A}}\|\bdv - \bdv^{n}\|^2\right\},\\
\label{eq:update_w_last} \bar{\bdw}^{n+1} &= \bar{\bdw}(\widetilde{\bdf}^{n+1})\quad \text{with} ~ \widetilde{\bdf}^{n+1} = \bdf_{\theta}^{n+1}~\text{or}~\bdf^{n+1}.
\end{align} 
\end{subequations}
If some given termination condition is satisfied, then \textbf{output} $\bdf^{n+1}$; 
otherwise, let $n \gets n+1$, \textbf{goto} \emph{Loop}.
\end{algorithmic}
\end{algorithm} 

Note that one can update variable $\bar{\bdw}$ by different ways, which is subsequently generated  
different extended primal-dual schemes. If the updates in \cref{eq:update_w} and \cref{eq:update_w_last} 
are the same, we just keep \cref{eq:update_w}. 

Before studying more details in \cref{algo:proposed_NLPD_algorithm}, 
we present a useful result by the following theorem. 
%if we update $\bdf$ and $\bar{\bdw}(\bdf)$ respectively, that is, $\bar{\bdw}$ is generated by some point $\bdf^{\prime}$ which may not be $\bdf$. Generally, by \cref{prop:inversionrule}, if $\bdw$ in \cref{def:biconjugate} takes other value which forms like the optimal solution of right side of \cref{def:biconjugate}, then the right term of \cref{def:biconjugate} happens to be the first order Taylor expansion of $h( \bdf)$. This inference will play a key role in the subsequent derivation of the proposed algorithm. 

\begin{theorem}\label{thm:conjugate_equals_first_taylor}
For any proper Fr\'echet differentiable convex function $h(\bdf)$, suppose $\bar{\bdw}_h(\bdf^{\prime})$ is the optimal solution 
in the right side of \cref{def:biconjugate} corresponding to point $ \bdf^{\prime}$. Then 
\begin{equation}\label{eq:conjugate_taylor}
   \langle  \bdf,\bar{\bdw}_h(\bdf^{\prime})\rangle - h^{\ast}\bigl(\bar{\bdw}_h(\bdf^{\prime})\bigr)=h( \bdf^{\prime})+\nabla h( \bdf^{\prime})( \bdf- \bdf^{\prime}).
\end{equation}
\end{theorem}
\begin{proof} 
Since $\bar{\bdw}_h(\bdf^{\prime})$ is the optimal solution in the right side of \cref{def:biconjugate} at $ \bdf^{\prime}$, by \cref{prop:inversionrule}, we obtain $\bar{\bdw}_h(\bdf^{\prime})=\nabla h(\bdf^{\prime})$. Taking $\bar{\bdw}_h(\bdf^{\prime})$ into \cref{def:conjugate}, then for the optimal solution $\bar{\bdf}_{\!h}\bigl(\bar{\bdw}_h(\bdf^{\prime})\bigr)$ in the right side of \cref{def:conjugate}, by \cref{prop:inversionrule} again, we obtain
\begin{equation*}
    \bar{\bdf}_{\!h}\bigl(\bar{\bdw}_h(\bdf^{\prime})\bigr) = \nabla h^{\ast}\bigl(\bar{\bdw}_h(\bdf^{\prime})\bigr)=\nabla h^{\ast}\bigl(\nabla h(\bdf^{\prime})\bigr)=\bdf^{\prime}.
\end{equation*}
Then, by \cref{def:conjugate}, and substituting the above equations into the left side of \cref{eq:conjugate_taylor} yields its right side.  
%\begin{align*}
%    \langle  \bdf,\bar{\bdw}_h(\bdf^{\prime})\rangle-h^{\ast}\bigl(\bar{\bdw}_h(\bdf^{\prime})\bigr)
%    & =\langle  \bdf,\nabla h( \bdf^{\prime})\rangle +(h( \bdf^{\prime})-\langle  \bdf^{\prime},\nabla h( \bdf^{\prime})\rangle)\\
%    & = h( \bdf^{\prime})+\nabla h( \bdf^{\prime})( \bdf- \bdf^{\prime}).
%\end{align*}
%which is our desired equality.
\end{proof}
Specifically, using \cref{thm:conjugate_equals_first_taylor}, and 
taking $h(\bdf)=K_j(\bdf)$, and accordingly $\bar{\bdw}_h(\bdf^{\prime}) = \bar{\bdw}_j(\bdf^{\prime})$, we obtain 
\begin{equation}\label{eq:taylor_Kj}
    \langle \bdf, \bar{\bdw}_j(\bdf^{\prime}) \rangle-K_j^{\ast}\bigl(\bar{\bdw}_j(\bdf^{\prime})\bigr)=K_j(\bdf^{\prime}) +\nabla K_j(\bdf^{\prime})(\bdf-\bdf^{\prime}).
\end{equation}
By \cref{def:Phi}, using \cref{eq:resolvent} and \cref{eq:taylor_Kj}, the update in \cref{eq:update_f} can be written as 
%Before considering different updating cases, for primal variables update \cref{eq:update_f}, 
%invoking \cref{corollary: conjugate equals to first taylor approx}, we obtain 
\begin{align*}
    \bdf^{n+1} 
%    &\overset{\cref{eq:update_f}}{=} \argmin_{\bdf}\left\{\Phi\bigl(\bdf, \bar{\bdw}^n, \bdu^n,\bdv^n\bigr) + \frac{1}{2\tau}\|\bdf - \bdf^{n}\|^2\right\}\\
    &= \arg\min_{\bdf \in \mathcal{X}}\Biggl\{ \sum_{j} u_j^n \Bigl\{ \langle \bdf,\bar{\bdw}_j(\bar{\bdf}^n)\rangle -K_j^{\ast}\bigl(\bar{\bdw}_j(\bar{\bdf}^n)\bigr)\Bigr\}+\langle A\bdf,\bdv^n \rangle  +G(\bdf)+ \frac{1}{2\tau}\|\bdf - \bdf^{n}\|^2\Biggr\}\\
&\overset{\cref{eq:taylor_Kj}}{=}\arg\min_{\bdf \in \mathcal{X}}\Biggl\{ \sum_{j} u_{j}^n \nabla K_j(\bar{\bdf}^n)\bdf + G(\bdf)+A^{\tra}\bdv^n + \frac{1}{2\tau}\|\bdf - \bdf^{n}\|^2\Biggr\}\\
% &= \arg\min_{\bdf}\left\{ G(\bdf) + \frac{1}{2\tau}\|\bdf - \bdf^{n}+\tau[\nabla \bdK(\bar{\bdf}^n) ]^{\tra}\bdu^n +\tau A^{\tra}\bdv^n\|^2\right\}\\
&\overset{\cref{eq:resolvent}}{=} (\Id+ \tau\partial G)^{-1}\Bigl(\bdf^n - \tau[\nabla \bdK(\bar{\bdf}^n)]^{\tra}\bdu^n -\tau A^{\tra}\bdv^n \Bigl). 
\end{align*}
Similarly, the update in \cref{eq:update_u} becomes 
\begin{align*}
    \bdu^{n+1} 
%    &\overset{\cref{eq:update_u}}{=} \argmin_{\bdu \in \mathcal{Y}_K}\left\{-\Phi(\bdf_{\theta}^{n+1}, \bar{\bdw}^{n+1}, \bdu,\bdv^n) + \frac{1}{2\sigma_{\!\!K}}\|\bdu - \bdu^{n}\|^2\right\}\\
        &= \arg\min_{\bdu \in \mathcal{Y}_K}\Biggl\{ -\sum_{j} u_j \Bigl\{ \langle \bdf^{n+1}_{\theta},\bar{\bdw}_j(\bar{\bdf}^{n+1})\rangle -K_j^{\ast}\bigl(\bar{\bdw}_j(\bar{\bdf}^{n+1})\bigr)\Bigr\} +F^{\ast}(\bdu) +\frac{1}{2\sigma_{\!\!K}}\|\bdu - \bdu^{n}\|^2\Biggr\}\\
    &\overset{\cref{eq:taylor_Kj}}{=}  \arg\min_{\bdu \in \mathcal{Y}_K}\left\{ -\langle \bdK(\bar{\bdf}^{n+1})+\nabla \bdK(\bar{\bdf}^{n+1})(\bdf_{\theta}^{n+1}-\bar{\bdf}^{n+1}),\bdu\rangle + F^{\ast}(\bdu)+ \frac{1}{2\sigma_{\!\!K}}\|\bdu - \bdu^{n}\|^2\right\}\\
%    & = \arg\min_{\bdu \in \mathcal{Y}_K}\left\{ F^{\ast}(\bdu)+ \frac{1}{2\sigma_{\!\!K}}\|\bdu - \bdu^{n}-[K(\bar{\bdf}^{n+1})+\nabla \bdK(\bar{\bdf}^{n+1})(\bdf_{\theta}^{n+1}-\bar{\bdf}^{n+1})]\|^2\right\}\\
    &\overset{\cref{eq:resolvent}}{=} (\Id + \sigma_{\!\!K}\partial F^{\ast})^{-1}\Bigl(\bdu^n + \sigma_{\!\!K} [\bdK(\bar{\bdf}^{n+1})+\nabla \bdK(\bar{\bdf}^{n+1})(\bdf_{\theta}^{n+1}-\bar{\bdf}^{n+1})]\Bigr). 
\end{align*}
Moreover, we have the update in \cref{eq:update_v} as 
\begin{align*}
     \bdv^{n+1} 
%    &\overset{\cref{eq:update_v}}{=} \argmin_{\bdv}\left\{-\Phi(\bdf_{\theta}^{n+1}, \bar{\bdw}^{n+1}, \bdu^{n+1},\bdv) + \frac{1}{2\sigma_{\!\!A}}\|\bdv - \bdv^{n}\|^2\right\}\\
    &= \arg\min_{\bdv \in \mathcal{Y}_A}\left\{ - \langle A\bdf_{\theta}^{n+1},\bdv\rangle +E^{\ast}(\bdv) +\frac{1}{2\sigma_{\!\!A}}\|\bdv - \bdv^{n}\|^2\right\}\\
     &\overset{\cref{eq:resolvent}}{=} (\Id + \sigma_{\!\!A}\partial E^{\ast})^{-1}\bigl(\bdv^n + \sigma_{\!\!A} A\bdf_{\theta}^{n+1} \bigr). 
\end{align*}

Next, we give compact forms of the specific schemes which are generated by updating \cref{eq:update_w} 
and \cref{eq:update_w_last} in different ways.
\paragraph{I. Let $\bar{\bdf}^{n+1}=\bdf_{\theta}^{n+1}$ in \cref{eq:update_w} and $\widetilde{\bdf}^{n+1} = \bdf_{\theta}^{n+1}$ in \cref{eq:update_w_last}}
Then \cref{algo:proposed_NLPD_algorithm} reduces to the following scheme. We call it EPD-Exact algorithm.
\begin{subequations}\label{eq:algo_proposed1}
\begin{align}
\label{eq:updatef_1} \bdf^{n+1} &= (\Id + \tau\partial G)^{-1}\bigl(\bdf^n - \tau[\nabla \bdK(\bdf_{\theta}^n)]^{\tra}\bdu^n-\tau A^{\tra}\bdv^n \bigl), \\
\label{eq:acceleratexf_1} \bdf_{\theta}^{n+1} &= \bdf^{n+1} + \theta(\bdf^{n+1}- \bdf^n), \\
\label{eq:updateu_1} \bdu^{n+1} &=  (\Id + \sigma_{\!\!K}\partial F^{\ast})^{-1}\bigl(\bdu^n + \sigma_{\!\!K} \bdK(\bdf_{\theta}^{n+1})\bigr), \\
\label{eq:updatev_1} \bdv^{n+1}&= (\Id + \sigma_{\!\!A}\partial E^{\ast})^{-1}\bigl(\bdv^n + \sigma_{\!\!A} A\bdf_{\theta}^{n+1} \bigr).
\end{align} 
\end{subequations}

\paragraph{II. Let $\bar{\bdf}^{n+1}=\bdf^{n+1}$ in \cref{eq:update_w} and $\widetilde{\bdf}^{n+1} = \bdf^{n+1}$ 
in \cref{eq:update_w_last}} Then \cref{algo:proposed_NLPD_algorithm} 
becomes the following compact scheme, called EPD-Linearized algorithm.
\begin{subequations}\label{eq:algo_proposed2}
\begin{align}
\label{eq:updatef_2} \bdf^{n+1} &= (\Id + \tau\partial G)^{-1}\bigl(\bdf^n - \tau[\nabla \bdK(\bdf^n)]^{\tra}\bdu^n-\tau A^{\tra}\bdv^n\bigl), \\
\label{eq:acceleratexf_2} \bdf_{\theta}^{n+1} &= \bdf^{n+1} + \theta(\bdf^{n+1}  - \bdf^n), \\
\label{eq:updated_2} \bdu^{n+1} &=  (\Id + \sigma_{\!\!K}\partial F^{\ast})^{-1}\bigl(\bdu^n + \sigma_{\!\!K} [\bdK(\bdf^{n+1}) + \nabla \bdK(\bdf^{n+1})(\bdf_{\theta}^{n+1} - \bdf^{n+1})]\bigr) , \\
\label{eq:updatev_2} \bdv^{n+1}&= (\Id + \sigma_{\!\!A}\partial E^{\ast})^{-1}\bigl(\bdv^n + \sigma_{\!\!A} A\bdf_{\theta}^{n+1} \bigr)
\end{align} 
\end{subequations}
%Note that the terms in square brackets of \cref{eq:updated_2} is clearly the first-order 
%Taylor expansion of $\bdK(\bdf_{\theta}^{n+1})$ in \cref{eq:updateu_1} at $\bdf^{n+1}$.

\paragraph{III. Let $\bar{\bdf}^{n+1}=\bdf_{\theta}^{n+1}$ in \cref{eq:update_w} 
and $\widetilde{\bdf}^{n+1} = \bdf^{n+1}$ in \cref{eq:update_w_last}} Then we get
\begin{subequations}\label{eq:Exact_NL_PDHGM}
\begin{align}
\label{eq:updatef_va1} \bdf^{n+1} &= (\Id + \tau\partial G)^{-1}\Bigl(\bdf^n - \tau[\nabla \bdK(\bdf^n)]^{\tra}\bdu^n-\tau A^{\tra}\bdv^n \Bigl), \\
\label{eq:acceleratexf_va1} \bdf_{\theta}^{n+1} &= \bdf^{n+1} + \theta(\bdf^{n+1}  - \bdf^n), \\
\label{eq:updated_va1} \bdu^{n+1} &=  (\Id + \sigma_{\!\!K}\partial F^{\ast})^{-1}\bigl(\bdu^n + \sigma_{\!\!K} \bdK(\bdf_{\theta}^{n+1})\bigr), \\
\label{eq:updatev_3} \bdv^{n+1}&= (\Id + \sigma_{\!\!A}\partial E^{\ast})^{-1}\bigl(\bdv^n + \sigma_{\!\!A} A\bdf_{\theta}^{n+1} \bigr). 
\end{align} 
\end{subequations}

\paragraph{IV. Let $\bar{\bdf}^{n+1}=\bdf^{n}$ in \cref{eq:update_w} and $\widetilde{\bdf}^{n+1} = \bdf^{n+1}$ in \cref{eq:update_w_last}} The resulting iterative scheme is   
\begin{subequations}\label{eq:Linearized_NL_PDHGM}
\begin{align}
\label{eq:updatef_va2} \bdf^{n+1} &= (\Id + \tau\partial G)^{-1}\Bigl(\bdf^n - \tau[\nabla \bdK(\bdf^n)]^{\tra}\bdu^n-\tau A^{\tra}\bdv^n\Bigl), \\
\label{eq:acceleratexf_va2} \bdf_{\theta}^{n+1} &= \bdf^{n+1} + \theta(\bdf^{n+1}  - \bdf^n), \\
\label{eq:updated_va2} \bdu^{n+1} &=  (\Id + \sigma_{\!\!K}\partial F^{\ast})^{-1}\bigl(\bdu^n + \sigma_{\!\!K} [\bdK(\bdf^n) + \nabla \bdK(\bdf^n)(\bdf_{\theta}^{n+1} - \bdf^n)]\bigr), \\
\label{eq:updatev_4} \bdv^{n+1}&= (\Id + \sigma_{\!\!A}\partial E^{\ast})^{-1}\bigl(\bdv^n + \sigma_{\!\!A} A\bdf_{\theta}^{n+1} \bigr). 
\end{align} 
\end{subequations} 
Furthermore, we have two more extensions.  

\paragraph{V. Let $\bar{\bdf}^{n+1} = \bdf^{n+1}$ in \cref{eq:update_w} and $\widetilde{\bdf}^{n+1}=\bdf_{\theta}^{n+1}$ in \cref{eq:update_w_last}} The associated compact form is translated into 
\begin{subequations}\label{eq:Exact_NL_PDHGM_2}
\begin{align}
\label{eq:updatef_va3} \bdf^{n+1} &= (\Id + \tau\partial G)^{-1}\Bigl(\bdf^n - \tau[\nabla \bdK(\bdf_{\theta}^n)]^{\tra}\bdu^n-\tau A^{\tra}\bdv^n \Bigl), \\
\label{eq:acceleratexf_va3} \bdf_{\theta}^{n+1} &= \bdf^{n+1} + \theta(\bdf^{n+1}  - \bdf^n), \\
\label{eq:updated_va3} \bdu^{n+1} &=  (\Id + \sigma_{\!\!K}\partial F^{\ast})^{-1}\bigl(\bdu^n + \sigma_{\!\!K} [\bdK(\bdf^{n+1}) + \nabla \bdK(\bdf^{n+1})(\bdf_{\theta}^{n+1} - \bdf^{n+1})]\bigr), \\
\label{eq:updatev_5} \bdv^{n+1}&= (\Id + \sigma_{\!\!A}\partial E^{\ast})^{-1}\bigl(\bdv^n + \sigma_{\!\!A} A\bdf_{\theta}^{n+1} \bigr). 
\end{align} 
\end{subequations}

\paragraph{VI. Let $\bar{\bdf}^{n+1}=\bdf^{n}$ in \cref{eq:update_w} and $\widetilde{\bdf}^{n+1} = \bdf_{\theta}^{n+1}$ in \cref{eq:update_w_last}} The generated iterative scheme is formulated as 
\begin{subequations}\label{eq:Linearized_NL_PDHGM_2}
\begin{align}
\label{eq:updatef_va6} \bdf^{n+1} &= (\Id + \tau\partial G)^{-1}\Bigl(\bdf^n - \tau[\nabla \bdK(\bdf_{\theta}^n)]^{\tra}\bdu^n-\tau A^{\tra}\bdv^n\Bigl), \\
\label{eq:acceleratexf_va6} \bdf_{\theta}^{n+1} &= \bdf^{n+1} + \theta(\bdf^{n+1}  - \bdf^n), \\
\label{eq:updated_va6} \bdu^{n+1} &=  (\Id + \sigma_{\!\!K}\partial F^{\ast})^{-1}\bigl(\bdu^n + \sigma_{\!\!K} [\bdK(\bdf^n) + \nabla \bdK(\bdf^n)(\bdf_{\theta}^{n+1} - \bdf^n)]\bigr), \\
\label{eq:updatev_6} \bdv^{n+1}&= (\Id + \sigma_{\!\!A}\partial E^{\ast})^{-1}\bigl(\bdv^n + \sigma_{\!\!A} A\bdf_{\theta}^{n+1} \bigr). 
\end{align} 
\end{subequations}

\subsection{Relationship to existing algorithms}

In this section, we will establish the relationship between the proposed algorithm framework and existing algorithms. 

\paragraph{Primal-dual algorithm for convex problem}
If $\bdK$ becomes linear, we execute the above derivation again, and it is easy to validate 
that the above framework and generated schemes are reduced to the primal-dual 
algorithm for the convex problem as in \cref{eq:convex_opt_primal} (or Chambolle--Pock algorithm in \cite{chpo11}). 
More precisely, we let $\bdK(\bdf) = \bdK\bdf$ in this paragraph, 
then $\nabla \bdK(\bdf) = \bdK$, and $\Phi\bigl(\bdf,\bar{\bdw}(\bdf),\bdu,\bdv\bigr)$ is rewritten as 
\[
\widehat{\Phi}(\bdf, \bdu, \bdv) := 
\langle \bdK\bdf, \bdu\rangle +\langle A\bdf,\bdv \rangle - F^{\ast}(\bdu)-E^{\ast}(\bdv) + G(\bdf). 
\]
\Cref{algo:proposed_NLPD_algorithm} can be reformulated as the following \cref{algo:PD_algorithm}. 
\begin{algorithm}[htpb]
\caption{The primal-dual algorithm for convex problem}
\label{algo:PD_algorithm}
\begin{algorithmic}[1]
\STATE \emph{Initialize}: Given $\tau > 0$, $\sigma_{\!\!K} > 0$, $\sigma_{\!\!A} > 0$, $\theta \in [0, 1]$, and initial point $(\bdf^0, \bdu^0, \bdv^0)$. 
Let $n \gets 0$. 
\STATE \emph{Loop}: Update $(\bdf^{n+1}, \bdu^{n+1}, \bdv^{n+1})$ by 
\begin{subequations}
\label{eq:PD_algorithm_steps}
\begin{align}
\label{eq:update_f_PD} \bdf^{n+1} &= \arg\min_{\bdf \in \mathcal{X}}\left\{\widehat{\Phi}(\bdf, \bdu^n, \bdv^n ) + \frac{1}{2\tau}\|\bdf - \bdf^{n}\|^2\right\}, \\
\label{eq:accelerate_f_PD} \bdf_{\theta}^{n+1} &= \bdf^{n+1} + \theta(\bdf^{n+1}  - \bdf^n), \\
\label{eq:update_u_PD} \bdu^{n+1} &= \arg\min_{\bdu \in \mathcal{Y}_K}\left\{-\widehat{\Phi}(\bdf_{\theta}^{n+1}, \bdu, \bdv^n) + \frac{1}{2\sigma_{\!\!K}}\|\bdu - \bdu^{n}\|^2\right\},\\
\label{eq:update_v_PD} \bdv^{n+1} &= \arg\min_{\bdv \in \mathcal{Y}_A}\left\{-\widehat{\Phi}(\bdf_{\theta}^{n+1}, \bdu^{n+1}, \bdv) + \frac{1}{2\sigma_{\!\!A}}\|\bdv - \bdv^{n}\|^2\right\}. 
\end{align} 
\end{subequations}
If some given termination condition is satisfied, then \textbf{output} $\bdf^{n+1}$; 
otherwise, let $n \gets n+1$, \textbf{goto} \emph{Loop}.
\end{algorithmic}
\end{algorithm} 

Then, all of the generated schemes I--VI are translated into the following unique one. 
\begin{subequations}\label{eq:compact_form_PD}
\begin{align}
\label{eq:updatef_va_PD} \bdf^{n+1} &= (\Id + \tau\partial G)^{-1}\bigl(\bdf^n - \tau\bdK^{\tra}\bdu^n-\tau A^{\tra}\bdv^n\bigl), \\
\label{eq:acceleratexf_va_PD} \bdf_{\theta}^{n+1} &= \bdf^{n+1} + \theta(\bdf^{n+1}  - \bdf^n), \\
\label{eq:updated_va_PD} \bdu^{n+1} &=  (\Id + \sigma_{\!\!K}\partial F^{\ast})^{-1}\bigl(\bdu^n + \sigma_{\!\!K} \bdK\bdf_{\theta}^{n+1}\bigr), \\
\label{eq:updatev_PD} \bdv^{n+1}&= (\Id + \sigma_{\!\!A}\partial E^{\ast})^{-1}\bigl(\bdv^n + \sigma_{\!\!A} A\bdf_{\theta}^{n+1} \bigr). 
\end{align} 
\end{subequations} 

Hence, the proposed algorithm framework can be seen as an extension of the primal-dual algorithm for convex problem. 

\paragraph{Exact/Linearized nonlinear \ac{PDHGM}} 
It is easy to observe that the exact and linearized nonlinear \ac{PDHGM} (i.e., 
Exact NL-PDHGM and Linearized NL-PDHGM) in \cite{va14} correspond to 
schemes \cref{eq:Exact_NL_PDHGM} and \cref{eq:Linearized_NL_PDHGM} respectively. 
It means that we have incorporated these two NL-PDHGM schemes into the proposed algorithm framework. 
By \cref{eq:conjugate_taylor} in \cref{thm:conjugate_equals_first_taylor}, we 
establish the relationship between the biconjugate formula and the first-order Taylor expension of the convex function itself. 
Equivalently, we provide the complete mathematical explanation for the two iterative schemes, which is easier to 
understand than the way by using the first-order Taylor expension as the linear approximation in \cite{va14}. 

Furthermore, the proposed algorithm framework contains more iterative schemes other than the two above, which 
provides more alternatives to solve the nonconvex optimization problem in \cref{eq:opt_primal}. 

\paragraph{Nonconvex primal-dual algorithm} 
A nonconvex primal-dual algorithm, namely, nonconvex Chambolle--Pock algorithm, 
is proposed to solve the image reconstruction model in spectral \ac{CT} with 
various convex functionals including \ac{TV} norm 
as constraint rather than regularization term \cite{chpa18,chpan21}, but which can be also extended as an alternative to solving 
the model \cref{eq:opt_primal}. Here we do not state that algorithm itself in detail, 
but just focus on the way to linearize the nonlinear $\bdK(\bdf)$ at related step. 

In \cite[eq. (10)]{chpan21} or \cite[Algorithm 2]{chpa18}, $\bdK(\bdf)$ is linearized at  
current point $\bdf^n$ equivalently by  
\begin{equation}\label{eq:NCPD_pan}
\bdK(\bdf) \approx \bdK(\bdf^n) + \nabla \bdK(0)(\bdf - \bdf^n), 
\end{equation} 
where $\nabla \bdK(0)$ is definitely the constant matrix $\mathcal{H}$ in that algorithm. 
Thus, in what follows we can obtain the corresponding general iterative scheme as  
\begin{subequations}\label{eq:NCPD}
\begin{align}
\label{eq:updatef_va_NCPD} \bdf^{n+1} &= (\Id + \tau\partial G)^{-1}\Bigl(\bdf^n - \tau[\nabla \bdK(0)]^{\tra}\bdu^n-\tau A^{\tra}\bdv^n \Bigl), \\
\label{eq:acceleratexf_va_NCPD} \bdf_{\theta}^{n+1} &= \bdf^{n+1} + \theta(\bdf^{n+1}  - \bdf^n), \\
\label{eq:updated_va_NCPD} \bdu^{n+1} &=  (\Id + \sigma_{\!\!K}\partial F^{\ast})^{-1}\bigl(\bdu^n + \sigma_{\!\!K} [\bdK(\bdf^n) + \nabla \bdK(0)(\bdf_{\theta}^{n+1} - \bdf^n)]\bigr), \\
\label{eq:updatev_NCPD} \bdv^{n+1}&= (\Id + \sigma_{\!\!A}\partial E^{\ast})^{-1}\bigl(\bdv^n + \sigma_{\!\!A} A\bdf_{\theta}^{n+1} \bigr). 
\end{align} 
\end{subequations}

Remark that the linearization in \cref{eq:NCPD_pan} is not exactly the first-order Taylor expansion of $\bdK(\bdf)$ at $\bdf^n$, 
which applies $\nabla \bdK(0)$ instead of gradient $\nabla \bdK(\bdf^n)$. 
The accuracy of the linear approximation in \cref{eq:NCPD_pan} is only $O(\|\bdf - \bdf^n\|)$, which carries with much 
lower accuracy than $O(\|\bdf - \bdf^n\|^2)$ by the first-order Taylor expansion. However, if $\bdK$ is 
much close to a linear operator, such approximation would be effective. 

%As we observe, the proposed algorithm \cref{eq:algo_proposed1} and \cref{eq:algo_proposed2} can be seen as a 
%modified Exact-PDHGM to \cref{eq:Exact_NL_PDHGM} and modified Linearized-PDHGM 
%to \cref{eq:Linearized_NL_PDHGM} respectively. More precisely, we estimate $\nabla \bdK$ at $\bdf_{\theta}^n$ 
%in \cref{eq:updatef_1} rather than $\bdf^n$ in \cref{eq:updatef_va1}, and use the first-order Taylor 
%expansion of $K(\bdf_{\theta}^{n+1})$  at $\bdf^{n+1}$ in \cref{eq:updated_2} rather than $\bdf^n$ 
%in \cref{eq:updated_va2}. Evidently, for each iteration, the update variables of interest are used as latest as possible in our proposed algorithm. 

\paragraph{Nonconvex \ac{ADMM}} The authors proposed a nonconvex \ac{ADMM} in \cite[Algorithm 1]{BaSi21} to solve 
the optimization problem as 
\begin{equation} \label{eq:opt_primal_admm}
\min_{\bdf, \by} \bigl\{F_1(\bdf) + F_2(\by)\bigr\},  \quad \text{s.t. $A\bdf + B\by = \bdc$}, 
\end{equation}   
where $F_1$ and $F_2$ are potentially nonconvex and/or nondifferentiable. 

As an alternative, we can use the nonconvex \ac{ADMM} to solve model \cref{eq:opt_primal} explicitly. 
To proceed, we need to introduce one auxiliary variable $\by = A\bdf$, and then model \cref{eq:opt_primal} can be rewritten as 
\begin{equation} \label{eq:opt_primal_admm_2}
\min_{\bdf, \by} \bigl\{\underbrace{(F\circ\bdK)(\bdf) + G(\bdf)}_{F_1(\bdf)} + \underbrace{E(\by)}_{F_2(\by)}\bigr\},  \quad \text{s.t. $A\bdf - \by = 0$}, 
\end{equation}    
%let 
%\[
%F_1(\bdf) := (F\circ\bdK)(\bdf) + G(\bdf) \quad \text{and} \quad F_2(\by) := E(\by), 
%\]
where assume that $F\circ\bdK$ is differentiable but nonconvex, and $G$ and $E$ are convex but nondifferentiable. 
Then, the nonconvex \ac{ADMM} for problem \cref{eq:opt_primal_admm_2} is given 
as the following \cref{algo:nonconvex_ADMM_algorithm}. 
\begin{algorithm}[htbp]
\caption{The nonconvex \ac{ADMM} for the model in \cref{eq:opt_primal}}
\label{algo:nonconvex_ADMM_algorithm}
\begin{algorithmic}[1]
\STATE \emph{Initialize}: Given $\tau > 0$, $\sigma > 0$, initial point $(\bdf^0, \by^0, \bz^0)$, and 
positive matrices $P_0$, $P_1$ and $P_2$. Let $n \gets 0$. 
\STATE \emph{Loop}: Update $(\bdf^{n+1}, \by^{n+1}, \bz^{n+1})$ by 
\begin{subequations}
\label{eq:NCADMM_algorithm_steps}
\begin{align}
\label{eq:NCADMM_update_f_PD} \bdf^{n+1} &= \arg\min_{\bdf}\Bigl\{G(\bdf) + \langle\bdf, \nabla(F\circ\bdK)(\bdf^n) + A^{\tra}\bz^n \rangle 
\\ 
\nonumber&\hspace{20mm}+ \frac{1}{2}\|A\bdf - \by^n\|^2_{P_0} + \frac{1}{2\tau}\|\bdf - \bdf^{n}\|^2_{P_1}\Bigr\}, \\
\label{eq:NCADMM_update_u_PD} \by^{n+1} &= \arg\min_{\by}\left\{E(\by) - \langle\by, \bz^n\rangle + \frac{1}{2}\|A\bdf^{n+1} - \by\|^2_{P_0} + \frac{1}{2\sigma}\|\by - \by^{n}\|^2_{P_2}\right\},\\
\label{eq:NCADMM_update_v_PD} \bz^{n+1} &= \bz^n + A\bdf^{n+1} - \by^{n+1}.  
\end{align} 
\end{subequations}
If some given termination condition is satisfied, then \textbf{output} $\bdf^{n+1}$; 
otherwise, let $n \gets n+1$, \textbf{goto} \emph{Loop}.
\end{algorithmic}
\end{algorithm} 

Note that the auxiliary variables are often required to introduce in order to use 
nonconvex \ac{ADMM}. Since $F\circ\bdK$ is a nonlinear differentiable function, 
linear approximation also needs to be used directly and taken 
as its first-order Taylor expansion at current point $\bdf^n$ in \cref{eq:NCADMM_update_f_PD}. 
Moreover, the inner iteration is probably required to solve the subproblem 
in \cref{eq:NCADMM_update_f_PD} for the existence of $A$ in the third term of the objective function.

\section{Convergence analysis}
\label{sec:convergence_analysis} 
Without loss of generality, we mainly present the convergence analysis for iterative 
schemes \cref{eq:algo_proposed1} and \cref{eq:algo_proposed2}. 
For ease of analyzing algorithm convergence, we take acceleration parameter $\theta=1$ 
in \cref{algo:proposed_NLPD_algorithm}. We assume that problem \cref{eq:mini_max_problem} 
has at least one solution $\widehat{\bdbeta} := [\whf^{\tra},\whu^{\tra},\whv^{\tra}]^{\tra}$ such that 
\begin{equation}\label{eq:optimal_conditions}
0\in H(\widehat{\bdbeta}):=\begin{bmatrix}
 \partial G(\whf)+[\nabla \bdK(\whf)]^{\tra} \whu + A^{\tra}\whv \\
\partial F^{*}(\whu) - \bdK(\whf)\\
\partial E^{\ast}(\whv)-A\whf
\end{bmatrix}. %\tag{optimal conditions}
\end{equation}
By the definition and property of proximal mapping, we can rewrite the updates in \cref{eq:algo_proposed1} as 
\begin{align}
\label{eq:update_1_equal} 0 \in 
\begin{bmatrix}
 \partial G(\bdf^{n+1}) + [\nabla \bdK(\bdf_{\theta}^n)]^{\tra}\bdu^n + A^{\tra} \bdv^n + \tau^{-1}(\bdf^{n+1} -\bdf^n) \\
\partial  F^{\ast}(\bdu^{n+1}) - \bdK(\bdf_{\theta}^{n+1})  + \sigma_{\!\!K}^{-1}(\bdu^{n+1} - \bdu^n)\\
\partial  E^{\ast}(\bdv^{n+1}) - A\bdf_{\theta}^{n+1}  + \sigma_{\!\!A}^{-1}(\bdv^{n+1} - \bdv^n)
\end{bmatrix},
\end{align} 
and the updates in \cref{eq:algo_proposed2} as 
\begin{align}
\label{eq:update_2_equal} 0 \in 
\begin{bmatrix}
\partial G(\bdf^{n+1}) + [\nabla \bdK(\bdf^n)]^{\tra}\bdu^n + A^{\tra} \bdv^n + \tau^{-1} (\bdf^{n+1} - \bdf^n) \\ 
\partial F^{\ast}(\bdu^{n+1}) - \bdK(\bdf^{n+1}) - \nabla \bdK(\bdf^{n+1})(\bdf_{\theta}^{n+1} - \bdf^{n+1})+ \sigma_{\!\!K}^{-1}  (\bdu^{n+1} - \bdu^n)\\
\partial  E^{\ast}(\bdv^{n+1}) - A\bdf_{\theta}^{n+1}  + \sigma_{\!\!A}^{-1}(\bdv^{n+1} - \bdv^n)
\end{bmatrix}.
\end{align} 
Let $\bdbeta^n := [{\bdf^n}^{\tra}, {\bdu^n}^{\tra}, {\bdv^n}^{\tra}]^{\tra}$ and 
\begin{equation}\label{eq:Mf}
 M(\bdf) :=\begin{bmatrix}
\tau^{-1} \Id & -[\nabla \bdK(\bdf)]^{\tra}& -A^{\tra} \\
- \nabla \bdK(\bdf) & \sigma_{\!\!K}^{-1}\Id & 0 \\
-A &0& \sigma_{\!\!A}^{-1}\Id
\end{bmatrix}.
\end{equation}
Then we respectively simplify the above update formulas in \cref{eq:update_1_equal} and \cref{eq:update_2_equal} as  
\begin{equation}\label{eq:update_1_simplify}
    0 \in \wtH^{\one}_{n+1}(\bdbeta^{n+1})+M_{n+1}^{\one}(\bdbeta^{n+1}-\bdbeta^{n})
\end{equation}
and
\begin{equation}\label{eq:update_2_simplify}
\quad 0 \in \wtH^{\two}_{n+1}(\bdbeta^{n+1})+M^{\two}_{n+1}(\bdbeta^{n+1}-\bdbeta^{n}),
\end{equation}
where $M_{n+1}^{\one} := M(\bdf_{\theta}^n)$, $M_{n+1}^{\two} := M(\bdf^n)$, 
\begin{equation*}
\wtH^{\one}_{n+1}(\bdbeta^{n+1}) :=
 \begin{bmatrix}
 \partial G(\bdf^{n+1})+\left[\nabla \bdK(\bdf_{\theta}^{n})\right]^{\tra} \bdu^{n+1} +A^{\tra}\bdv^{n+1}\\
\partial F^{*}(\bdu^{n+1})-\bdK(\bdf_{\theta}^{n+1}) + \nabla \bdK(\bdf_{\theta}^n)(\bdf^{n+1}-\bdf^n) \\
\partial E^{*}(\bdv^{n+1})-A\bdf_{\theta}^{n+1} + A(\bdf^{n+1}-\bdf^n)
\end{bmatrix},
\end{equation*} 
and
\begin{equation*}
\wtH^{\two}_{n+1}(\bdbeta^{n+1}) := \\ \begin{bmatrix}
 \partial G(\bdf^{n+1})+\left[\nabla \bdK(\bdf^{n})\right]^{\tra} \bdu^{n+1}+A^{\tra}\bdv^{n+1}\\
 \partial F^{*}(\bdu^{n+1})-\bdK(\bdf^{n+1})-
\nabla \bdK(\bdf^{n+1})(\bdf_{\theta}^{n+1}-\bdf^{n+1})+\nabla \bdK({\bdf}^{n})(\bdf^{n+1}-\bdf^{n})\\
\partial E^{*}(\bdv^{n+1})-A\bdf_{\theta}^{n+1} + A(\bdf^{n+1}-\bdf^n)
\end{bmatrix}. 
\end{equation*}

Subsequently, as a representative,  we just give the convergence analysis for \cref{eq:update_1_simplify} and \cref{eq:update_2_simplify}, 
which represent schemes \cref{eq:algo_proposed1} and \cref{eq:algo_proposed2} respectively. 

\subsection{Basic assumption and result}\label{subsec:Assumptions}

Here we introduce some local assumptions. Before this, we define a neighborhood of $\widehat{\bdbeta}$ as 
\begin{equation*}
    \mathcal{O}(\rho_{\bdf},\rho_{\bdu},\rho_{\bdv}):=\mathcal{B}_{\whf}(\rho_{\bdf}) \times \mathcal{B}_{\whu}(\rho_{\bdu})\times \mathcal{B}_{\whv}(\rho_{\bdv}),
\end{equation*}
where $\mathcal{B}_{\whf}(r) := \bigl\{ \bdf\in\mathcal{X}  : \|\bdf-\whf\| \le r \bigr\}$,
and $\mathcal{B}_{\whu}(r)$, $\mathcal{B}_{\whv}(r)$ are defined in the same way. 
%\cref{eq:optimal_conditions} shows that $K(\whf)\in\partial F^{\ast}(\whu)$, here we 
%assume that $F^{\ast}$ has stronger property at $\whu$ than convexity.

\begin{assumption}(Strongly monotone of $\partial F^{\ast}$)\label{ass:monotone_of_F}
The set-valued mapping $\partial F^{\ast}$ is $\gamma_{F^{\ast}}$-strongly monotone 
at $\whu$ in $\mathcal{B}_{\whu}(\rho_{\bdu})$, which holds 
\begin{equation*}
    \langle \partial F^{\ast}(\bdu)-\partial F^{\ast}(\whu), \bdu-\whu \rangle \ge \gamma_{F^{\ast}} \|\bdu-\whu\|^2 
\end{equation*}
for any $\bdu\in\mathcal{B}_{\whu}(\rho_{\bdu})$. 
\end{assumption} 

\begin{remark}\label{rem:strongly_monotone}
Here we only need $\partial F^{\ast}$ rather than both $\partial F^{\ast}$ and $\partial G$ to be strongly monotone, which implies  
a weaker assumption than \cite[Assumption 3.3]{accle_Val19}. 
\end{remark}
%We also require that $K$ have following local properties in $\mathcal{B}_{\whf}(\rho_{\bdf}) \times \mathcal{B}_{\whu}(\rho_{\bdu})$.

\begin{assumption}(Locally Lipschitz of $\nabla \bdK$)\label{ass:locally_Lipschitz}
Let $\bdK: \mathcal{X} \rightarrow \mathcal{Y}_K$ be Fr\'echet differentiable, and there are some $L \geq 0$ such that
\begin{equation}\label{ineq: locally Lipschitz}
  \|\nabla \bdK(\bdf)-\nabla \bdK(\bx)\| \leq L\|\bdf-\bx\| 
\end{equation}
for any $\bdf, \bx \in \mathcal{B}_{\whf}(\rho_{\bdf})$.
\end{assumption}

In addition, we require a restriction on the nonlinearity of $\bdK$. 

\begin{assumption}(Nonlinearity restriction of $\bdK$)\label{ass:nonlinear_restrict}
 Assume that $\whu\neq0$. For any $\bdf, \bx \in \mathcal{B}_{\whf}(\rho_{\bdf})$, $\bdu\in\mathcal{B}_{\whu}(\rho_{\bdu})$, there are some $\gamma_{1} \in \mathbb{R}$ and $\lambda_1 \geq 0,$ such that 
\begin{multline*}
     \langle[\nabla \bdK(\bx)-\nabla \bdK(\whf)] (\bdf-\whf), \whu \rangle +\langle \bdK(\whf) - \bdK(\bdf)-\nabla \bdK(\bdf)(\whf-\bdf), \bdu-\whu\rangle \notag 
\\ \geq -\gamma_{1}\|\bdu-\whu\|^{2} - \lambda_1\|\bdf-\bx\|^{2}.
\end{multline*}
\end{assumption}
\begin{remark} 
If $\bdK$ is linear, the above inequality holds explicitly for 
any $\gamma_1\ge0$ and $\lambda_1\ge0$. In fact, \cref{ass:nonlinear_restrict} restricts 
operator $\bdK$ not too far from linear operator, in other words, whose second-order remainder is sufficiently small. 
\end{remark}
\begin{remark}\label{remark for ass3}
If $\whu=0$, for any $\bx=\bdf \in \mathcal{B}_{\whf}(\rho_{\bdf})$, we have
\begin{align*}
   \langle \bdK(\whf) - \bdK(\bdf)-\nabla \bdK(\bdf)(\whf-\bdf), \bdu\rangle  \geq -\gamma_{1}\|\bdu\|^{2}. 
\end{align*}
Now we fix $\bdf$ and let $\bdu^{\prime}=\bdK(\whf) - \bdK(\bdf)-\nabla \bdK(\bdf)(\whf-\bdf) \neq 0$. 
We can take $\bdu=c \bdu^{\prime}$ for $0 < c < c_1$ so that $\bdu\in\mathcal{B}_{\whu}(\rho_{\bdu})$. 
Substituting them into the above inequality yields $ 1 \le \gamma_1 c$. 
However we can always take $c$ sufficiently small to make this inequality invalid. 
So this is the reason why we need to restrict $\whu\neq0$ in \cref{ass:nonlinear_restrict}. 
%Moreover, the special case of $\whu=0$ on \cref{eq:opt} is given in \cref{subsec:thm_nls}.
\end{remark}

The following theorem in \cite{accle_Val19} would be useful for our convergence analysis. 
%which is also an important result for the analysis with $\bdK$ linear in \cite{chpo11}.

\begin{theorem}\label{thm:basis convergence}(\cite{accle_Val19}) 
In a Hilbert space $\mathcal{H}$, let $M_{n+1}\in\mathcal{L}(\mathcal{H},\mathcal{H})$, suppose 
\begin{equation*}
    0 \in \wtH_{n+1}(\bdbeta^{n+1})+M_{n+1}(\bdbeta^{n+1}-\bdbeta^{n})
\end{equation*}
is solvable for iterates $\{\bdbeta^n\}_{n\in\mathbb{N}}$. If $M_{n+1}$ is symmetric positive semidefinite and
\begin{equation}\label{inq:descent estimate on fixed step}
    \begin{aligned}
\langle\wtH_{n+1}(\bdbeta^{n+1}), \bdbeta^{n+1}-\widehat{\bdbeta}\rangle- \frac{1}{2}\|\bdbeta^{n+1}-\widehat{\bdbeta}\|_{M_{n+2}-M_{n+1}}^{2} 
+\frac{1}{2}\|\bdbeta^{n+1}-\bdbeta^{n}\|_{M_{n+1}}^{2} \ge \mathcal{V}_{n+1}
\end{aligned}
\end{equation}
for a fixed $n\in \mathbb{N}$, some $\widehat{\bdbeta}\in \mathcal{H}$ and $\mathcal{V}_{n+1}\in\mathbb{R}$, then 
\begin{align}\label{inq:descent estimate for neighbor step}
    \frac{1}{2}\|\bdbeta^{n+1}-\widehat{\bdbeta}\|_{M_{n+2}}^{2}+\mathcal{V}_{n+1} \leq \frac{1}{2}\|\bdbeta^{n}-\widehat{\bdbeta}\|_{M_{n+1}}^{2}.
\end{align}
Moreover, if \cref{inq:descent estimate for neighbor step} holds for all $0\le n\le N$ ($N \geq 1$), then we have 
\begin{align}\label{inq: N steps descent estimate}
    \frac{1}{2}\|\bdbeta^{N+1}-\widehat{\bdbeta}\|_{M_{N+2}}^{2}+\sum_{n=0}^{N} \mathcal{V}_{n+1}  \leq \frac{1}{2}\|\bdbeta^{0}-\widehat{\bdbeta}\|_{M_{1}}^{2}.
\end{align}
\end{theorem}

%\begin{remark}
%Note \cref{inq:descent estimate for neighbor step} holds as long as \cref{inq:descent estimate on fixed step} is satisfied for some fixed $n$, which is irrelevant to \cref{inq:descent estimate on fixed step} whether satisfies or not for other index. The proof see \cite{valkonen2020testing}.
%\end{remark}

\subsection{Convergence results}\label{sec:convergent_analysis}

Here we present convergence results for the proposed schemes. As we can see, 
 \cref{thm:basis convergence} requires that a local metric is constructed by $M_{n+1}$ and the descent estimate inequality \cref{inq:descent estimate on fixed step} holds. To this end, 
%we assume that the step sizes $\tau$, $\sigma_{\!\!K}$, $\sigma_{\!\!A}$ satisfy
%\begin{align}\label{eq:step_assumption}
%   \tau\sigma_{\!\!K} C_{\!K}^2 < s(1-\kappa),\quad  \tau\sigma_{\!\!A} \|A\|^2 < (1-s)(1-\kappa)
%\end{align}
%for some $0< s, \kappa<1$, where $C_{\!K} := \sup_{\mathcal{B}_{\whf}(\rho_{\bdf})} \|\nabla \bdK(\bdf)\|$. 
we first give the following lemma. 
%which is similar to the lemma 3.4 in \cite{accle_Val19}.
\begin{lemma}\label{lem:M_self_adjoint}
Let $0< s, \kappa<1$, $C_{\!K} = \sup_{\mathcal{B}_{\whf}(\rho_{\bdf})} \|\nabla \bdK(\bdf)\|$ and 
\begin{align*}
    B_1(\bdf) &:= \diag\biggl(\frac{\kappa}{\tau} \Id, \frac{1}{\sigma_{\!\!K}}\Id-\frac{\tau}{s(1-\kappa)}\nabla \bdK(\bdf)\nabla \bdK(\bdf)^{\tra}, \frac{1}{\sigma_{\!\!A}}\Id-\frac{\tau}{(1-s)(1-\kappa)}A A^{\tra}\biggr),\\
    B_2(\bdf) &:= \diag\biggl(\frac{1}{\tau}\Id-\frac{\sigma_{\!\!K}}{1-\kappa}\nabla \bdK(\bdf)^{\tra}\nabla \bdK(\bdf) - \frac{\sigma_{\!\!A}}{1-\kappa} A^{\tra}A, 
\frac{\kappa}{\sigma_{\!\!K}} \Id, \frac{\kappa}{\sigma_{\!\!A}} \Id\biggr).
\end{align*}
Assume that step sizes $\tau$, $\sigma_{\!\!K}$, and $\sigma_{\!\!A}$ satisfy
\begin{align}\label{eq:step_assumption}
   \tau\sigma_{\!\!K} C_{\!K}^2 < s(1-\kappa),\quad  \tau\sigma_{\!\!A} \|A\|^2 < (1-s)(1-\kappa), 
\end{align} 
then $M(\bdf)$ defined in \cref{eq:Mf} is symmetric positive semidefinite such that 
\[
M(\bdf) \succeq B_1(\bdf) \succeq 0 \quad\text{and \quad} M(\bdf) \succeq B_2(\bdf) \succeq 0
\]
for any $\bdf \in \mathcal{B}_{\whf}(\rho_{\bdf})$.  

\end{lemma}
\begin{proof}
See \cref{proof:M_self_adjoint} for the proof.
\end{proof}

Next, we present the following lemma.
\begin{lemma}\label{lemma: alg_estimate}
Assume that $G : \mathcal{X} \rightarrow \overline{\Real}_{+}$, $F^{\ast} : \mathcal{Y}_K \rightarrow \overline{\Real}_{+}$ 
and $E^{\ast} : \mathcal{Y}_A \rightarrow \overline{\Real}_{+}$ are proper, convex, \ac{l.s.c} functions, and $\bdf_{\theta}^{n+1}\in\mathcal{B}_{\whf}(\rho_{\bdf})$, $\bdbeta^{n},\bdbeta^{n+1}\in \mathcal{O}(\rho_{\bdf},\rho_{\bdu},\rho_{\bdv})$. Moreover, \cref{ass:monotone_of_F}--\cref{ass:nonlinear_restrict} and the conditions in \cref{lem:M_self_adjoint} hold in $\mathcal{O}(\rho_{\bdf},\rho_{\bdu},\rho_{\bdv})$. Then, we have the following estimate for scheme \cref{eq:algo_proposed1}
\begin{multline}\label{eq:alg1_H_estimate}
\langle\wtH^{\one}_{n+1}(\bdbeta^{n+1}), \bdbeta^{n+1}-\widehat{\bdbeta}\rangle  - \frac{1}{2}\|\bdbeta^{n+1}-\widehat{\bdbeta}\|_{M_{n+2}^{\one}-M_{n+1}^{\one}}^{2} 
+\frac{1}{2}\|\bdbeta^{n+1}-\bdbeta^{n}\|_{M_{n+1}^{\one}}^{2} \\ 
\ge  (\gamma_{F^{\ast}}-\gamma_1-\tilde{C})\|\bdu^{n+1}-\whu\|^2 +\Lambda_{\one}(\tau) \|\bdf^{n+1}-\bdf^n\|^{2} \\
-\Bigl(2\lambda_1+\frac{L\rho_{\bdu}}{2}\Bigr) \|\bdf^{n-1}-\bdf^n\|^{2},
\end{multline}
where $\tilde{C}:=C_{\!K}+ L\rho_{\bdf}/2,~
    \Lambda_{\one}(\tau):=\kappa/2\tau-(\tilde{C}+2\lambda_1+5L\rho_{\bdu}/2)$, 
and the following estimate for scheme \cref{eq:algo_proposed2}
\begin{multline}\label{eq:alg2_H_estimate}
\langle\wtH^{\two}_{n+1}(\bdbeta^{n+1}), \bdbeta^{n+1}-\widehat{\bdbeta}\rangle  -\frac{1}{2}\|\bdbeta^{n+1}-\widehat{\bdbeta}\|_{M_{n+2}^{\two}-M_{n+1}^{\two}}^{2} +\frac{1}{2}\|\bdbeta^{n+1}-\bdbeta^{n}\|_{M_{n+1}^{\two}}^{2}  \\
\ge  (\gamma_{F^{\ast}}-\gamma_1) \|\bdu^{n+1}-\whu\|^2 +\Lambda_{\two}(\tau)\|\bdf^{n+1}-\bdf^n\|^{2},
\end{multline}
where $\Lambda_{\two}({\tau}):= \kappa/2\tau -(\lambda_1+L\rho_{\bdu})$. 
% Here 
% \begin{multline*}
%     \Gamma_{n+1}(\bdf):=
%     \frac{1}{2\sigma_{\!\!K}}\biggl(\|\bdu^{n+1}-\bdu^n\|^2-\frac{\tau\sigma_{\!\!K}}{s(1-\kappa)}\|[\nabla \bdK(\bdf)]^{\tra}(\bdu^{n+1}-\bdu^n)\|^2\biggr) \\
%     +\frac{1}{2\sigma_{\!\!A}}\biggl(\|\bdv^{n+1}-\bdv^n\|^2-\frac{\tau\sigma_{\!\!A}}{(1-s)(1-\kappa)}\|A^{\tra}(\bdv^{n+1}-\bdv^n)\|^2\biggr). 
% \end{multline*}
\end{lemma}
\begin{proof}
See \cref{proof:alg_estimate} for the proof. 
\end{proof}

Then we have the following corollary. 
\begin{corollary}\label{feifei}
 Under the assumptions of \cref{lemma: alg_estimate}, suppose further that $\gamma_{F^{\ast}} > \gamma_1+\tilde{C}$. If step size $\tau$ satisfies
 \begin{equation}\label{eq:step_restrict1}
     \tau < \frac{\kappa}{2(\tilde{C}+4\lambda_1+3L\rho_{\bdu})},
  \end{equation}
 then $\{\bdbeta^n\}_{n\in\mathbb{N}}$ that generated by scheme \cref{eq:algo_proposed1} 
 fulfils   
 \begin{equation}\label{eq:alg1_iterates_estimate}
     \|\bdbeta^{n+1}-\widehat{\bdbeta}\|_{M_{n+2}^{\one}} \le \|\bdbeta^{0}-\widehat{\bdbeta}\|_{M_{1}^{\one}}.
 \end{equation}
Furthermore, if step size $\tau$ satisfies
 \begin{equation}\label{eq:step_restrict2}
     \tau < \frac{\kappa}{2(\lambda_1 + L\rho_{\bdu})},
 \end{equation}
 then $\{\bdbeta^n\}_{n\in\mathbb{N}}$ produced by scheme \cref{eq:algo_proposed2} 
 have the following relationship
 \begin{align}\label{eq:alg2_iterates_estimate}
     \|\bdbeta^{n+1}-\widehat{\bdbeta}\|_{M_{n+2}^{\two}} \leq \|\bdbeta^{n}-\widehat{\bdbeta}\|_{M_{n+1}^{\two}} \le \cdots\le \|\bdbeta^{0}-\widehat{\bdbeta}\|_{M_{1}^{\two}}.
 \end{align}
\end{corollary}
\begin{proof}
The proof is referred to \cref{proof:roughly_estimate}.
\end{proof}

Moreover, we need to verify that assumptions $\bdf_{\theta}^{n+1}\in \mathcal{B}_{\whf}(\rho_{\bdf})$ and $\bdbeta^n,\bdbeta^{n+1}\in\mathcal{O}(\rho_{\bdf},\rho_{\bdu},\rho_{\bdv})$ hold in \cref{lemma: alg_estimate}.
The case for scheme \cref{eq:algo_proposed1} will be proven in the following lemma, and it can be proven for scheme \cref{eq:algo_proposed2} by the same method.
\begin{lemma}\label{lemma:ensure_beta_n_neigborhood}
Given $\rho_{\bdf}, \rho_{\bdu}, \rho_{\bdv} > 0$, assume that \cref{ass:monotone_of_F} with $\gamma_{F^{\ast}} > \gamma_1+\tilde{C}$, 
and \cref{ass:locally_Lipschitz} and \cref{ass:nonlinear_restrict} hold, 
and step sizes $\tau$, $\sigma_{\!\!K}$ and $\sigma_{\!\!A}$ satisfy \cref{eq:step_assumption} and \cref{eq:step_restrict1}, 
and initial point $\bdbeta^0 \in \mathcal{O}(r_{\bdf},r_{\bdu},r_{\bdv})$ such that 
\begin{equation}\label{eq:initial_point_assumption}
   \|\bdbeta^0-\widehat{\bdbeta}\|_{M_1^{\one}}\le \min\Bigl\{ r_{\bdf}\sqrt{\kappa / \tau}, r_{\bdu}\sqrt{\kappa / \sigma_{\!\!K}},r_{\bdv}\sqrt{\kappa / \sigma_{\!\!A}} \Bigr\}, %\tag{initial-ass}
\end{equation}
where $r_{\bdf}+2C_{\bdf}\le \rho_{\bdf}$, $r_{\bdu}+C_{\bdu}\le \rho_{\bdu}$, $r_{\bdv}+C_{\bdv}\le \rho_{\bdv}$, and 
\begin{equation*}
    C_{\bdf}:=\tau\left((r_{\bdu}+2\|\whu\|)C_{\!K}+\|A\|r_{\bdv}\right),~
    C_{\bdu}:=\sigma_{\!\!K} (r_{\bdf}+2C_{\bdf}) C_{\!K}, ~
    C_{\bdv}:=\sigma_{\!\!A} (r_{\bdf}+2C_{\bdf})\|A\|. 
\end{equation*} 
Then for the iterate sequence generated by scheme \cref{eq:algo_proposed1}, $\bdbeta^{n}\in \mathcal{O}(\rho_{\bdf},\rho_{\bdu},\rho_{\bdv})$, and $\bdf_{\theta}^{n+1}\in \mathcal{B}_{\whf}(\rho_{\bdf})$ hold for all $n$.
\end{lemma}
\begin{proof}
See \cref{proof:ensure_beta_O} for the proof.
\end{proof}
Finally, we apply the previous analysis to conclude convergent results for the proposed algorithm framework. 
Here we only give the convergence theorem for scheme \cref{eq:algo_proposed1}, without loss of generality. 
\begin{theorem}\label{thm:PD alg converge}
Assume that $G : \mathcal{X} \rightarrow \overline{\Real}_{+}$, $F^{\ast} : \mathcal{Y}_K \rightarrow \overline{\Real}_{+}$ 
and $E^{\ast} : \mathcal{Y}_A \rightarrow \overline{\Real}_{+}$ are proper, convex, \ac{l.s.c} functions. 
Under the conditions in \cref{lemma:ensure_beta_n_neigborhood}, iterate sequence $\{\bdbeta^n\}_{n\in\mathbb{N}}$ 
generated by scheme \cref{eq:algo_proposed1} converges to some point which satisfies the optimal conditions in  \cref{eq:optimal_conditions}.
\end{theorem}
\begin{proof}
The proof is referred to \cref{proof:convergence_thm}.
\end{proof}
\begin{remark}
Since problem \cref{eq:mini_max_problem} may have more than one solutions, \cref{thm:PD alg converge} only shows that \cref{algo:proposed_NLPD_algorithm} converges to some point which may not be the solution we assumed previously. 
% However, recall that $-[\nabla \bdK(\whf)]^{\tra} \whu\in  \partial G(\whf)$, by the convexity of $G$, 
% \begin{align*}
%     \langle -[\nabla \bdK(\whf^{\prime})]^{\tra}\whu +[\nabla \bdK(\whf)]^{\tra}\whu, \whf^{\prime} - \whf \rangle \ge 0
% \end{align*}
% by \cref{remark for ass3}, taking $\bdf=\whf^{\prime}$, together with the above inequality, we have
% \begin{align*}
%     \langle -[\nabla \bdK(\whf^{\prime})]^{\tra}\whu +[\nabla \bdK(\whf)]^{\tra}\whu, \whf^{\prime} - \whf \rangle = 0
% \end{align*}
% then by the definition of subgradient, we get $G(\whf^{\prime})=G(\whf)$. 
\end{remark}
\begin{remark}
Similarly, under the assumptions in \cref{subsec:Assumptions}, the convergence of the other schemes can be given. Just some conditions need to be slightly modified, for instance, $\gamma_{F^{\ast}}>\gamma_1$ required; the two linearized schemes \cref{eq:algo_proposed2} and \cref{eq:Linearized_NL_PDHGM} require that \cref{eq:step_restrict2} is satisfied; 
scheme \cref{eq:Exact_NL_PDHGM} requires $\tau< \kappa/(2\lambda_1+L\rho_{\bdu})$. The left proofs of them are analogous. 
% \begin{align*}
%     \tau < \frac{\kappa}{2\lambda_1 +L\rho_{\bdu}}
% \end{align*}
\end{remark}

\section{Application to nonlinear imaging}
\label{sec:specific_algorithms}

We illustrate the effectiveness of the proposed algorithm framework for nonlinear 
imaging problems. In this section, we specifically consider the image reconstruction in spectral CT. 

To this end,  in \cref{eq:opt}, we take the data fidelity term as 
\begin{equation}\label{eq:fidelity}
 \DistFunc\bigl(\bdK(\bdf), \bdg\bigr) := \frac{1}{2}\|\bdK(\bdf) - \bdg\|_2^2, 
\end{equation} 
where the component of $\bdK(\bdf)$ is defined as \cref{eq:log_sum_exp}, and 
the regularization term is given as \ac{TV} functional \cite{RuOsFa92}, i.e., 
\[
\RegFunc(\bdf) := \lambda \|\nabla \bdf \|_1. 
\]
Here $\nabla = \diag\bigl(\nabla_1, \ldots, \nabla_D\bigr)$ (i.e., the $A := \nabla$ in \cref{eq:sparse_rep}), $\nabla_d$ denotes associated 
discrete form of the gradient operator for $\bdf_{\!\!d}$, and $\domain := \mathbb{R}_{+}^{DN}$ represents 
the space of nonnegative vectors. Hence, the specific optimization problem for image reconstruction in 
spectral CT can be formulated as
\begin{equation}\label{eq:specific_opt}
\min_{\bdf \in \mathbb{R}_{+}^{DN}} \left\{\frac{1}{2}\|\bdK(\bdf) - \bdg\|_2^2
+ \lambda \|\nabla \bdf\|_1\right\}.
\end{equation} 
Note that model \cref{eq:specific_opt} is actually equivalent to the model that changed the \ac{TV} regularization 
term into an inequality constraint, and the latter one has been considered in \cite{chpan21}. 
For \cref{eq:specific_opt}, the general framework in \cref{eq:opt_primal} is specified by 
\begin{align}
\label{eq:specific_F}&F\bigl(\bdK(\bdf)\bigr) := \frac{1}{2}\|\bdK(\bdf) - \bdg\|_2^2, \\
\label{eq:specific_E}&E(A\bdf) := \lambda \|\nabla \bdf\|_1, \\
\label{eq:specific_G}&G(\bdf) := \begin{cases} 
     0 & \text{if}~\bdf\in \mathbb{R}_{+}^{DN},  \\
    +\infty &  \text{otherwise}. 
   \end{cases}
\end{align}
Accordingly, we derive that 
\begin{align} 
\label{eq:specific_F_star}&F^{\ast}(\bdu) = \frac{1}{2}\|\bdu\|^2 + \langle \bdg, \bdu\rangle,  \\
\label{eq:specific_E_star}&E^{\ast}(\bdv) = \delta_{\textsf{Box}(\lambda)}(\bdv) = \begin{cases} 
     0, & \text{if}~\bdv\in \textsf{Box}(\lambda) := \{\bdv: \|\bdv\|_{\infty} \leq \lambda\},  \\
    +\infty, &  \text{otherwise}. 
  \end{cases}
\end{align}
It is obvious that here $G$, $F^{\ast}$ and $E^{\ast}$ are proper, convex, \ac{l.s.c} functions, and $F^{\ast}$ is strongly convex. 
Furthermore, we give some significant results about the $K_j$ in \cref{eq:log_sum_exp} by the following theorem. 
\begin{theorem}\label{thm:convex_lips_g}
For any given finite sequence $\{b_{dm}\}$ for $1\le d \le D$ and $1\le m \le M$, $K_j$ defined in \cref{eq:log_sum_exp} is 
Fr\'echet differentiable and convex with respect to $\bdf$. Furthermore, its gradient $\nabla K_j$ is global Lipschitz continuous.  
\end{theorem}
\begin{proof}
The proof is given in \cref{sec:convex_lips_g_spec_CT}. 
\end{proof}

As a result, the proposed algorithm framework can be used to solve the image reconstruction problem 
for spectral \ac{CT} in \cref{eq:specific_opt}. By \cref{def:proximal}, \cref{eq:resolvent} and the 
definitions of $G$, $F^{\ast}$ and $E^{\ast}$ as the above, 
we have $(\Id + \tau\partial G)^{-1}(\bdf) = \proj_{\domain}(\bdf)$, 
$\quad (\Id + \sigma_{\!\!A} \partial E^{\ast})^{-1}(\bdv) = \proj_{\textsf{Box}(\lambda)}(\bdv)$ and
\[
(\Id + \sigma_{\!\!K} \partial F^{\ast})^{-1}(\bdu) = \frac{\bdu - \sigma_{\!\!K} \bdg}{\sigma_{\!\!K} + 1}, 
\]
where $\proj_{\domain}$ and $\proj_{\textsf{Box}(\lambda)}$ denote the orthogonal projection 
operators onto  $\domain$ and $\textsf{Box}(\lambda)$ respectively. Hence, the proposed scheme 
in \cref{eq:algo_proposed1} can be explicitly written as 
\begin{subequations}\label{eq:algo_proposed1_explicit}
\begin{align}
\label{eq:updatef_1_explicit} \bdf^{n+1} &= \proj_{\domain}\Bigl(\bdf^n - \tau \bigl[\bigl(\nabla\bdK(\bdf_{\theta}^n)\bigr)^{\tra}\bdu^n + \nabla^{\tra}\bdv^n \bigr] \Bigr), \\
\label{eq:acceleratexf_1_explicit} \bdf_{\theta}^{n+1} &= \bdf^{n+1} + \theta(\bdf^{n+1}  - \bdf^n), \\
\label{eq:updated_1_explicit_u}  \bdu^{n+1} &=  \frac{\bdu^n + \sigma_{\!\!K} \bigl(\bdK(\bdf_{\theta}^{n+1}) - \bdg\bigr)}{\sigma_{\!\!K} + 1}, \\
\label{eq:updated_1_explicit_v}  \bdv^{n+1} &= \proj_{\textsf{Box}(\lambda)}\bigl(\bdv^n + \sigma_{\!\!A} \nabla\bdf_{\theta}^{n+1}\bigr). 
\end{align} 
\end{subequations}
%and 
%\begin{subequations}\label{eq:algo_proposed2_explicit}
%\begin{align}
%\label{eq:updatef_2_explicit} \bdf^{n+1} &= \proj_{\domain}\Bigl(\bdf^n - \tau\bigl[\bigl(\nabla\bdK(\bdf^n)\bigr)^{\tra}\bdu^n + \nabla^{\tra}\bdv^n \bigr]\Bigl), \\
%\label{eq:acceleratexf_2_explicit} \bdf_{\theta}^{n+1} &= \bdf^{n+1} + \theta(\bdf^{n+1}  - \bdf^n), \\
%\label{eq:updated_2_explicit_u}  \bdu^{n+1} &= \frac{\bdu^n + \sigma_{\!\!K} \bigl(\bdK(\bdf^{n+1}) + \nabla \bdK(\bdf^{n+1})(\bdf_{\theta}^{n+1} - \bdf^{n+1}) - \bdg\bigr)}{\sigma_{\!\!K} + 1},\\
%\label{eq:updated_2_explicit_v}  \bdv^{n+1} &= \proj_{\textsf{Box}(\lambda)}\bigl(\bdv^n + \sigma_{\!\!A} \nabla \bdf_{\theta}^{n+1}\bigr), 
%\end{align} 
%\end{subequations}
%respectively.  
Similarly, we can easily figure out the other specific schemes using \cref{eq:algo_proposed2}--\cref{eq:Linearized_NL_PDHGM_2}.

\subsection{Convergence analysis for $\whu=0$}\label{subsec:thm_nls}

Note that $\whu$ is an optimal dual variable which corresponds to 
the nonlinear operator in problem \cref{eq:mini_max_problem}. 
In \cref{sec:convergence_analysis}, we analyzed the convergence of the proposed schemes
under the case of $\whu \neq 0$ as in \cref{ass:nonlinear_restrict}. 
Here we further analyze the convergence for the case $\whu=0$ specially to  
scheme \cref{eq:algo_proposed1_explicit} in spectral CT reconstruction. 
Recalling the optimal conditions in \cref{eq:optimal_conditions}, and by \cref{eq:specific_F_star}, 
we have $\whu+\bdg - \bdK(\whf) = 0$. Thus, $\bdK(\whf)=\bdg$ when $\whu = 0$. 

As a representative, in what follows we 
prove the convergence of scheme \cref{eq:algo_proposed1_explicit}. 
First of all, we give the following lemma. 
\begin{lemma}\label{lem:R_f}
Let the component of $\bdK(\bdf)$ be defined as \cref{eq:log_sum_exp}, and $b_{dm} > 0$ for $1\le d \le D$ and $1\le m \le M$. There are a family of uniformly bounded operators $R_j(\bdf,\tilde{\bdf})$ such that 
\begin{equation}\label{eq:Kj_Rj}
    \nabla K_j(\bdf)=R_j(\bdf,\tilde{\bdf})\nabla K_j(\tilde{\bdf}) 
\end{equation}
for any $\bdf,\tilde{\bdf}\in \mathbb{R}^{DN}_{+}$. Furthermore, 
define $R(\bdf,\tilde{\bdf}) := \diag\bigl(R_1(\bdf,\tilde{\bdf}), \ldots, R_J(\bdf,\tilde{\bdf})\bigr)$. 
%such that $\nabla \bdK(\bdf)=R(\bdf,\tilde{\bdf})\nabla \bdK(\tilde{\bdf})$, 
There exists a positive constant $C_R$ such that 
\[
    \|R(\bdf,\tilde{\bdf})-I\|\le C_R\|\bdf -\tilde{\bdf}\|. 
\] 
\end{lemma}
\begin{proof}
The proof is referred to \cref{sec:proof_R_f}. 
\end{proof}

Then we have the following result.
\begin{lemma}\label{lem:local_property_g_hat}
Assume the conditions in \cref{lem:R_f} hold. If $\rho_{\bdf} \in (0, 1/C_R)$, 
\begin{align}\label{local_property_g}
    \|\bdK(\bdf)-\bdK(\tilde{\bdf})-\nabla \bdK(\tilde{\bdf})(\bdf-\tilde{\bdf})\|\le \frac{\rho_{\bdf} C_R}{1-\rho_{\bdf} C_R} \|\bdK(\bdf)-\bdK(\tilde{\bdf})\| 
\end{align}
for any $\bdf,\tilde{\bdf}\in \mathcal{B}_{\whf}(\rho_{\bdf})$. 
\end{lemma}
\begin{proof}
See \cref{sec:proof_local_property_g_hat} for the proof. 
\end{proof}

In what follows we have the theorem on the convergence of 
scheme \cref{eq:algo_proposed1_explicit}. 

%To the best of our knowledge, it is the first result that applies the Landweber local 
%condition \cref{local_property_g} to the proof of convergence on primal-dual schemes. 
\begin{theorem}\label{thm:hat_u_eqs_0}
%For the noise-free case \cref{eq:specific_opt_reduced}, where $\bdK$, $F$ and $G$ are defined as \cref{eq:log_sum_exp}, \cref{eq:specific_F} and \cref{eq:specific_G} respectively. 
Assume that the conditions in \cref{lem:M_self_adjoint} and \cref{lem:R_f} hold, and $\whu=0$.
Let $\rho_{\bdf} \in (0, 1/C_R)$, $\rho_{\bdu} > 0$, $\rho_{\bdv}>0$ and $\eta = \rho_{\bdf} C_R / (1-\rho_{\bdf} C_R)$. 
Suppose initial point $\bdbeta^0 $ satisfies \cref{eq:initial_point_assumption}, and 
step sizes $\tau$, $\sigma_{\!\!K}$ further satisfy
\begin{align} \label{special_step_restrict}
      \tau\sigma_{\!\!K} C_{\!K}^2 < (1-\eta)s(1-\kappa),\quad \tau<\frac{\kappa}{6L\rho_{\bdu}},\quad \sigma_{\!\!K} > \frac{\eta}{2(1-\eta)}.
\end{align}
Then iterate sequence $\{\bdf^n\}$ generated by \cref{eq:algo_proposed1_explicit} converges to some $\whf^{\prime}$ such that $\bdK(\whf^{\prime})=\bdg$.
\end{theorem}
\begin{proof}
The proof is referred to \cref{proof:case_spectral_CT}.
\end{proof}
\begin{remark}
By $\bdK(\whf^{\prime})=\bdK(\whf)$, $\whf^{\prime}\in\mathcal{B}_{\whf}(\rho_{\bdf})$ and \cref{lem:local_property_g_hat}, we get $\nabla \bdK(\whf)(\whf - \whf^{\prime})=0$.
If $\nabla \bdK(\whf)$ is column full rank, which is available when the scan is sufficient, we obtain $\whf^{\prime}=\whf$, 
otherwise, $\whf^{\prime}$ is not definitely equivalent to $\whf$.
\end{remark}

\subsection{Numerical experiments} 
\label{subsec:tests}
In this section, we describe the implementation details first, then apply the image results and evaluation metrics to 
show the effectiveness of the proposed algorithm framework, without loss of generality, mainly including the generated schemes in  \cref{eq:algo_proposed1}, \cref{eq:algo_proposed2}, \cref{eq:Exact_NL_PDHGM} and \cref{eq:Linearized_NL_PDHGM}, 
for brevity, named EPD-Exact, EPD-Linerized, Exact-PDHGM, and Linearized-PDHGM respectively.  

We adopt the simulated datasets with/without noise that generated by \cref{eq:log_sum_exp} to 
evaluate the proposed schemes for DECT. The truth image is obtained from a 2D 
slice of the modified 3D forbild head phantom, which is shown in the right side 
of \cref{fig:normalized spectra}. The water and bone basis truth images take gray values 
over $[0, 1]$, which are digitized by using $128\!\times\!128$ pixels and defined on a 
fixed domain $[-5, 5]\!\times\![-5, 5]$, as shown in \cref{fig:180_noisefree_contrast} respectively.
\begin{figure}[htbp]
    \centering
    \includegraphics[scale=0.42]{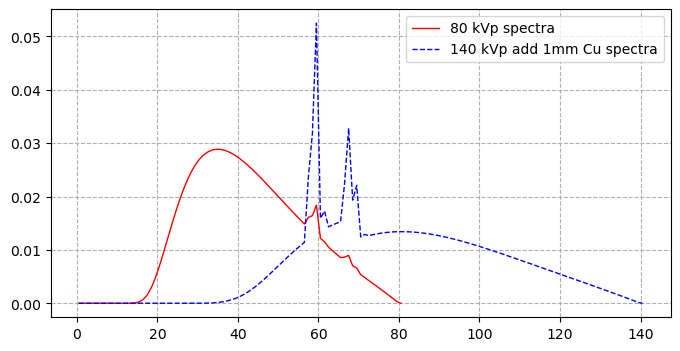}\hspace{4mm}
    \includegraphics[scale=0.36]{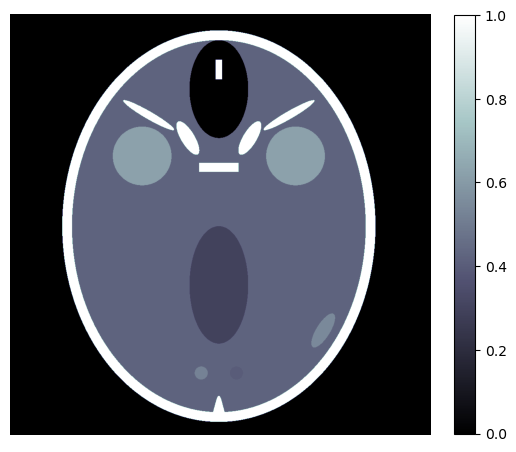}
    \caption{Two normalized X-ray spectra (left) and a 2D slice of the modified forbild phantom (right) used in tests.}
    \label{fig:normalized spectra}
\end{figure}

In this paper, we assume that the used energy spectra are known, which are simulated by the software 
SpectrumGUI, an open source X-ray spectra simulator. 
% (see \hyperlink{https://sourceforge.net/projects/spectrumgui/}{https://sourceforge.net/projects/spectrumgui/}). 
We used the discrete X-ray spectra of GE Maxiray 125 X-ray tube with tube voltages 80 kVp and 140 kVp,  
%which is split into energy intervals of 10 range. 
%In order to make the peaks of the energy spectra more separated, 
where the 140 kVp spectra is filtered by 1 mm copper. We normalize it and use the same X-ray spectra for different rays. 
The normalized spectra curve are shown in the left side of \cref{fig:normalized spectra}. The 
coefficients $\{b_{km}\}$ are obtained for water and bone basis materials from the database of 
National Institute of Standard Technology. 

We perform forward/backward projections by the tool in \cite{VaPa16}, or calculate the projection 
matrix by the algorithm in \cite{chen2020fast}. In all of the tests below, we use 2D parallel beam 
scanning geometry with 181 bins, which is evenly distributed on the range over $[-7.05, 7.05]$. 
Note that here the scanning geometry is unessential since it would only change the projection 
matrix. The routines were operated on ThinkStation with Xeon E5-2620 V4 2.10 GHz CPU, 
64GB ROM and TITAN Xp GPU. 

\subsubsection{Test 1: Simulated noise-free data}\label{test1_noisefree_data}

 For noise-free case, the measured data is simulated by $\bdg=\bdK(\bdf_{\mathrm{truth}})$, 
 where $\bdf_{\mathrm{truth}}$ denotes truth image. 
 Instead of using full-scan configuration, we consider a specific non-standard scanning configuration. 
The $180$ views for each of low (80 kVp) and high (140 kVp) energy spectra are uniformly distributed 
over $[0, \pi)$ and $[\theta_0, \theta_0 + \pi)$, respectively, where constant $\theta_0$ denotes 
the angular gap. Here we take it as $\pi/360$. We choose zero images as the initial point, 
and take $\lambda=0$ and $\tau=\sigma_{\!\!K}=0.2$. 

To validate numerical convergence, relative error, relative data fitting functional  
and relative \ac{TV} functional are acted as convergent conditions, which are defined as 
\begin{equation}\label{eq:relative_error}
RE_{\bdf}^{(n)} = \frac{\|\bdf^{n} - \bdf_{\mathrm{truth}}\|_2}{\|\bdf_{\mathrm{truth}}\|_2} , \quad RD_{\bdf}^{(n)} = \frac{\|\bdK(\bdf^{n}) - \bdg\|_2^2}{\|\bdg\|_2^2},\quad RT_{\bdf}^{(n)} = \frac{|\|\nabla\bdf^{n}\|_1 - \|\nabla\bdf_{\mathrm{truth}}\|_1|}{ \|\nabla\bdf_{\mathrm{truth}}\|_1}, 
\end{equation}
respectively. 
%   and the relative regularization function respectively
% \begin{equation}\label{eq:relative_fidelity}
% RDI_{\bdf}^n=\frac{\|\bdf^{n+1} - \bdf^n\|_2}{\|\bdf^{n+1}\|_2},~ RR_{\bdf}^{n} = \frac{\|W\bdf^n\|_1 - \|W\bdf_{\mathrm{truth}}\|_1}{\|W\bdf_{\mathrm{truth}}\|_1}.
% \end{equation} 
These metrics can be used to verify the numerical convergence as iteration number $n \rightarrow \infty$ (also see \cite{chpan21}). 
Since $\bdf_{\mathrm{truth}}$ is known in numerical simulation, the use of metrics $RE_{\bdf}^{(n)}$ and $RT_{\bdf}^{(n)}$ is reasonable. 
 %To achieve the convergent result, we perform sufficient iterations for the algorithm.  
 
 After $10^6$ iterations, the reconstructed water/bone basis images and 60/100 keV images are shown in the 
 second to fifth columns of \cref{fig:180_noisefree_contrast}, which are quite close to the truth images by visual comparison. 
\begin{figure}[htbp]
    \centering
    \vspace{-1mm}
    \includegraphics[trim={0 0 0 0},width=\textwidth]{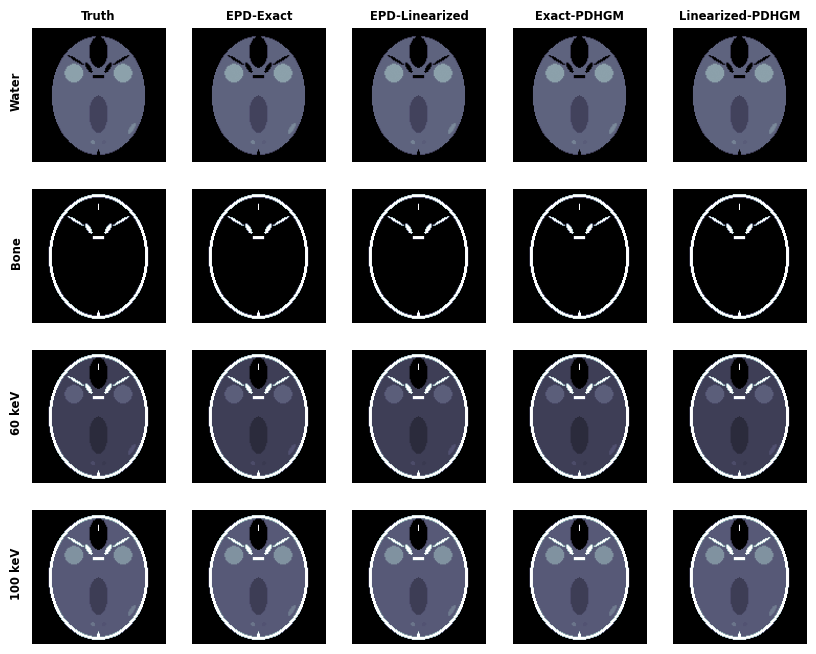}
    \vspace{-5mm}
    \caption{Reconstructed images after $10^6$ iterations using simulated noise-free data. From left to right: truth images, reconstructed results by schemes \cref{eq:algo_proposed1}, \cref{eq:algo_proposed2}, \cref{eq:Exact_NL_PDHGM} and \cref{eq:Linearized_NL_PDHGM}. From top to bottom: water basis image, bone basis image, 60 keV image and 100 keV image.}
    \vspace{-2mm}
\label{fig:180_noisefree_contrast}
\end{figure} 

\begin{figure}[htbp]
    \centering
    \includegraphics[width=0.33\textwidth]{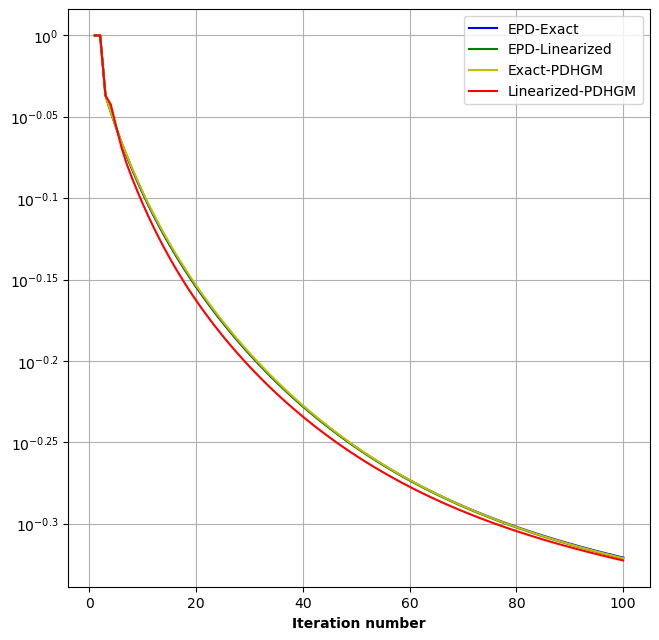}\hspace{-1mm}
    \includegraphics[width=0.33\textwidth]{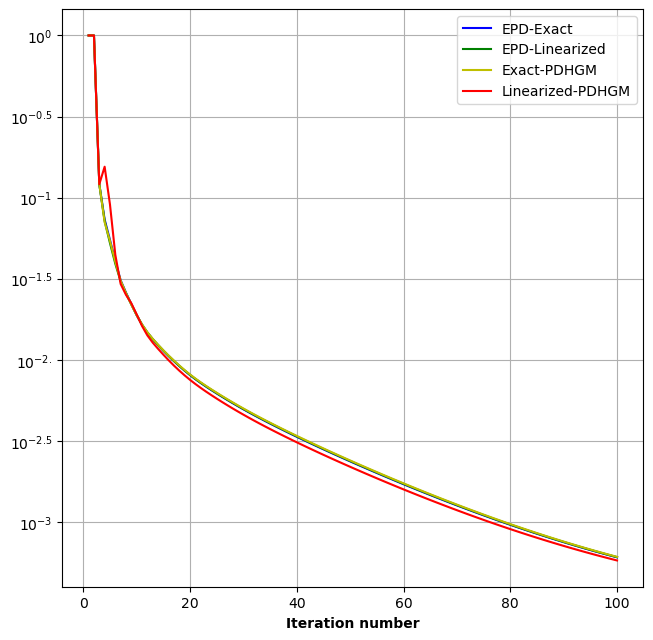}\hspace{-1mm}
    \includegraphics[width=0.33\textwidth]{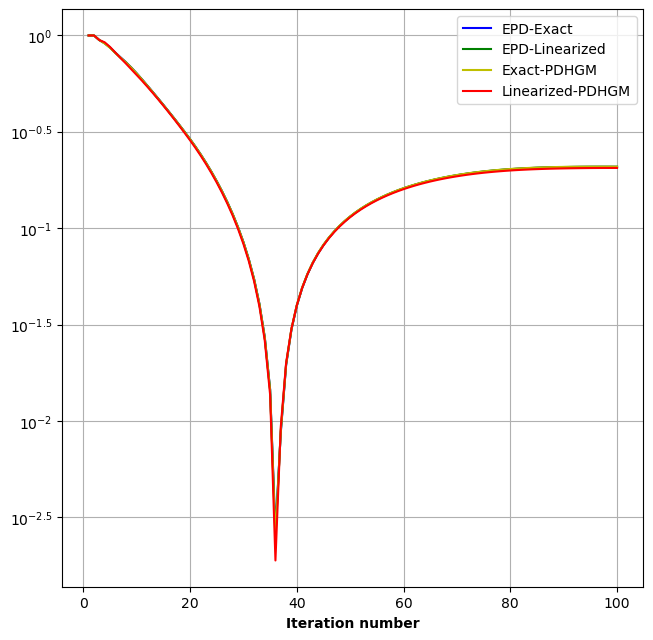}\\
    \includegraphics[width=0.33\textwidth]{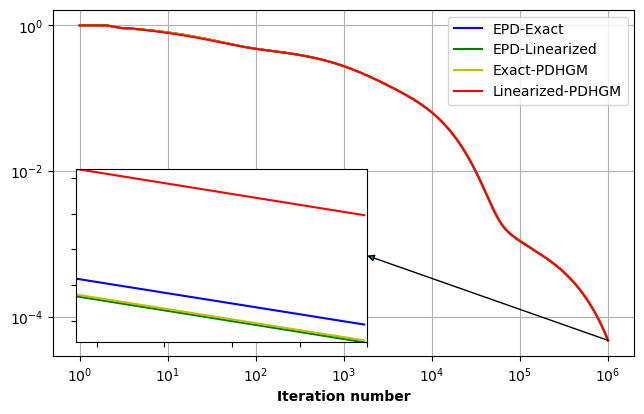}\hspace{-1mm}
    \includegraphics[width=0.33\textwidth]{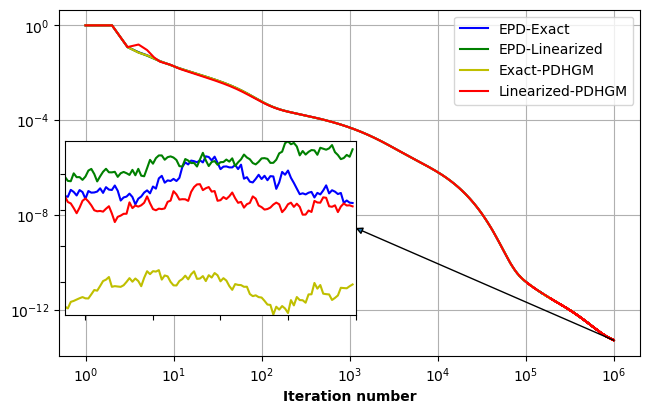}\hspace{-1mm}
    \includegraphics[width=0.33\textwidth]{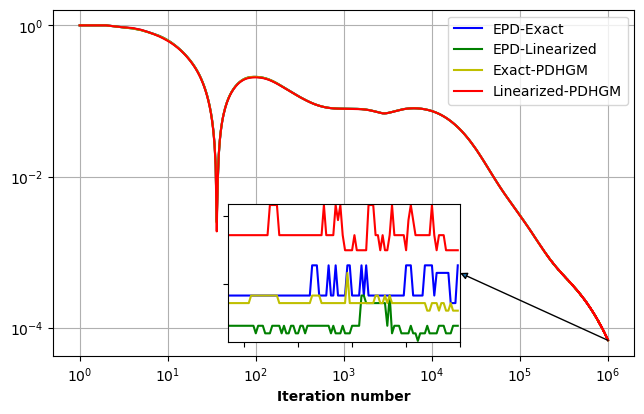}
    \caption{Convergent conditions $RE_{\bdf}^{(n)}$ (left), $RD_{\bdf}^{(n)}$ (middle) and $RT_{\bdf}^{(n)}$ (right) of schemes \cref{eq:algo_proposed1}, \cref{eq:algo_proposed2}, \cref{eq:Exact_NL_PDHGM} and \cref{eq:Linearized_NL_PDHGM} for simulated noise-free data,  as functions of iteration number $n$, where the first row shows the profiles of the first 100 iterations, and the embedded subgraphs of the second row show the last 100 iterations.}
    \label{fig:convergent_conds_0}
\end{figure} 
The convergent conditions are computed by \cref{eq:relative_error} for the iterative schemes and then 
plotted as functions of the iteration number in \cref{fig:convergent_conds_0}, which shows that the results reconstructed 
by the four schemes in \cref{eq:algo_proposed1}, \cref{eq:algo_proposed2}, \cref{eq:Exact_NL_PDHGM} and \cref{eq:Linearized_NL_PDHGM} are too close to distinguish by vision. We amplify the profiles for the fist 100 iterations shown in the first row of \cref{fig:convergent_conds_0} and the last 100 iterations shown in the embedded subgraphs in the second row of \cref{fig:convergent_conds_0}.  

We further use quantitative indexes to evaluate the reconstruction quility, such as the values of structural 
similarity (SSIM) \cite{wang2004image}, peak signal-to-noise ratio (PSNR), 
%\cite{huynh2008PSNR}, 
mean square error (MSE) and maximum pixel difference between the 
reconstructed image and corresponding truth (Max\_Diff). Remark that we use `1 - SSIM' instead of `SSIM' 
because the latter value is very close to but less than 1. Since the quantitative indexes by different schemes are quite similar, 
we only show those produced by scheme \cref{eq:algo_proposed1} in \cref{noisefree table}. 
\begin{table}[htbp]
  \centering
  \caption{Quantitative indexes for the result by scheme \cref{eq:algo_proposed1} using simulated noise-free data. }
    \begin{tabular}{crrrr}
    \hline
    \multicolumn{1}{l}{180 views, noise-free} & \multicolumn{1}{c}{1-SSIM} & \multicolumn{1}{c}{PSNR} & \multicolumn{1}{c}{MSE} & \multicolumn{1}{c}{Max\_Diff} \\
    \hline
    Water basis &  6.02$\times 10^{-8}$ &  {94.60} &  3.47$\times 10^{-10}$ &  7.63$\times 10^{-4}$ \\
    Bone basis &  7.51$\times 10^{-12}$ &  {131.86} &  6.41$\times 10^{-14}$ &  5.25$\times 10^{-6}$ \\
    60 keV &  {4.35$\times 10^{-9}$} &  {108.34} &  {1.47$\times 10^{-11}$} &  {1.57$\times 10^{-4}$} \\
    100 keV &  {3.04$\times 10^{-9}$} &  {109.96} &  {1.01$\times 10^{-11}$} &  {1.30$\times 10^{-4}$} \\
    \hline
    \end{tabular}%
  \label{noisefree table}%
\end{table}%

In addition, we plot the horizontal slices to compare the 60/100 keV images by scheme \cref{eq:algo_proposed1} 
with the truth counterparts, which are shown in \cref{horizontal_noisefree_slice}. 
The locations of the slices are shown at the red line on upper left of each picture in \cref{horizontal_noisefree_slice}. 
%We can see that the pretty well performance of these primal-dual schemes on simulated noise-free data.
\begin{figure}[ht]
    \centering\vspace{-2mm}
    \hspace{-3mm}
    \includegraphics[width=0.5\textwidth]{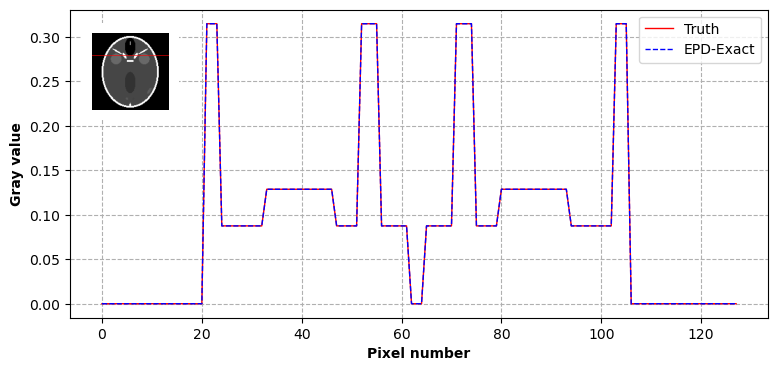}~
    \includegraphics[width=0.5\textwidth]{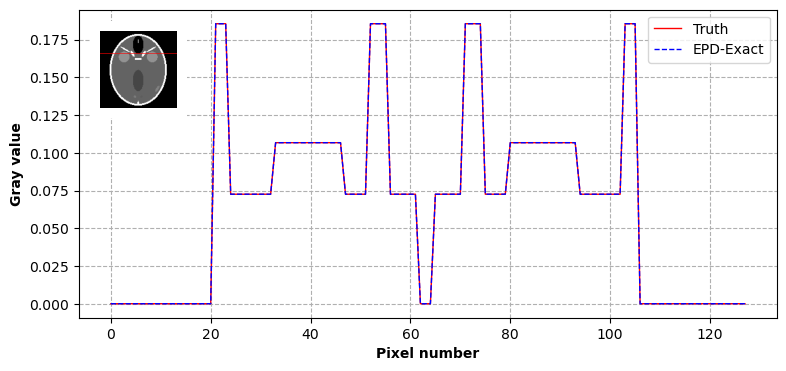}%\hspace{-2mm}
    \vspace{-1mm}
    \caption{Horizontal slices of the reconstructed 60/100 keV images by scheme \cref{eq:algo_proposed1} and the truths for simulated noise-free data.}
    \label{horizontal_noisefree_slice}
\end{figure} 

\subsubsection{Test 2: Simulated noise data}
\label{sec:test_suite_2}

We create simulated noise data by adding Gaussian noise onto the 
noise-free data as in \cref{eq:nl_op_eq}, which corresponds to the use 
of $\ell_2$ data fitting term. The scanning configuration is the same as the simulated noise-free case 
in \cref{test1_noisefree_data}. The signal-noise ratio (SNR) of the noise data is about 27.11 dB, and we take regularization 
parameter $\lambda=10^{-6}$. We choose zero images as the initial point, 
and take $\tau=\sigma_{\!\!K}=\sigma_{\!\!A}=0.2$. After $10^5$ iterations, the reconstructed 60/100 keV images by 
schemes \cref{eq:algo_proposed1}, \cref{eq:algo_proposed2}, \cref{eq:Exact_NL_PDHGM} and \cref{eq:Linearized_NL_PDHGM} 
are shown in \cref{fig:180angle_noise1e3_contrast}. 
\begin{figure}[htbp]
    \centering
    \includegraphics[trim={0 0 0 0},width=\textwidth]{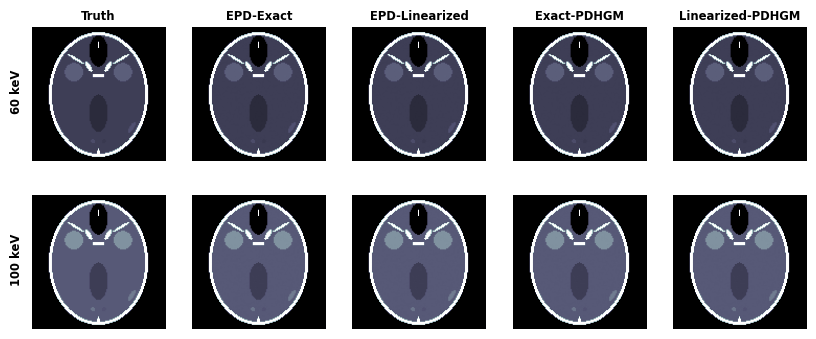}
     \vspace{-4mm}
    \caption{Reconstructed images after $10^5$ iterations using simulated noise data (about 27.11 dB). From left to right: truth images, reconstructed results by schemes \cref{eq:algo_proposed1}, \cref{eq:algo_proposed2}, \cref{eq:Exact_NL_PDHGM} and \cref{eq:Linearized_NL_PDHGM}. From top to bottom: 60 keV image and 100 keV image.}
\label{fig:180angle_noise1e3_contrast}
\end{figure}

The profiles of convergent conditions for schemes \cref{eq:algo_proposed1}, \cref{eq:algo_proposed2}, 
\cref{eq:Exact_NL_PDHGM} and \cref{eq:Linearized_NL_PDHGM} are plotted in \cref{fig:convergent_conds_noise1e3}, 
and the last 100 iterations are shown in the embedded subgraphs of \cref{fig:convergent_conds_noise1e3}. As we can see, 
these schemes have their own merits regarding convergence quality. 
The quantitative indexes of the results by scheme \cref{eq:algo_proposed1} are listed in \cref{tab:noise1e3}.
% comparing with the noise-free case \cref{noisefree table}, 
% the reconstructed image quality is reduced, which is accord with expectation. 
\begin{figure}[htbp]
    \centering
    \includegraphics[width=0.33\textwidth]{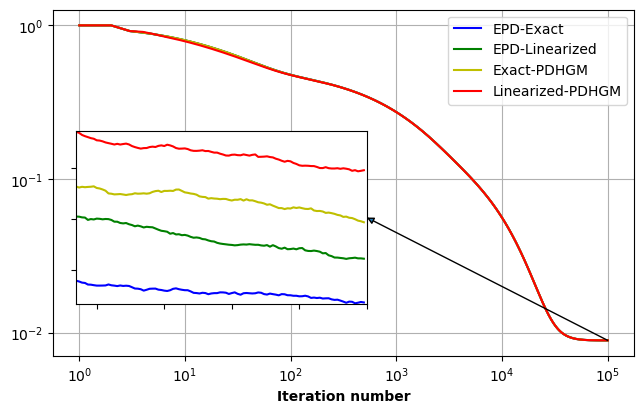}\hspace{-1mm}
    \includegraphics[width=0.33\textwidth]{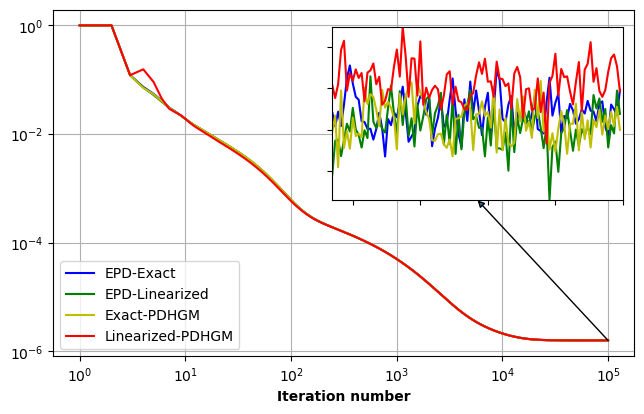}\hspace{-1mm}
    \includegraphics[width=0.33\textwidth]{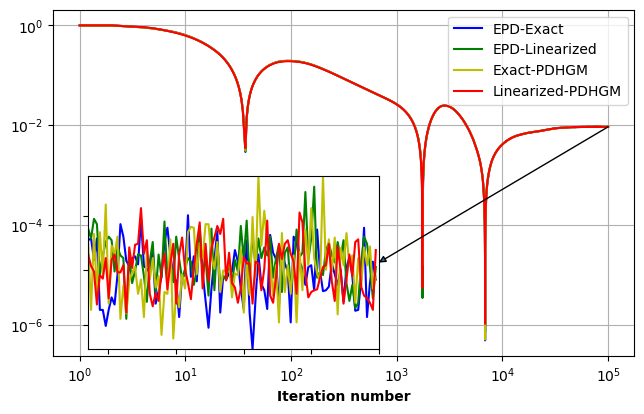}\vspace{-1mm}
    \caption{Convergent conditions $RE_{\bdf}^{(n)}$ (left), $RD_{\bdf}^{(n)}$ (middle) and $RT_{\bdf}^{(n)}$ (right) of schemes \cref{eq:algo_proposed1}, \cref{eq:algo_proposed2}, \cref{eq:Exact_NL_PDHGM} and \cref{eq:Linearized_NL_PDHGM} for simulated noise data (about 27.11 dB),  as functions of iteration number $n$, where the embedded subgraphs of the second row show the last 100 iterations.}
    \label{fig:convergent_conds_noise1e3}
\end{figure} 

\begin{table}[htbp]
  \centering
  \caption{Quantitative indexes for the results by scheme \cref{eq:algo_proposed1} using simulated noise data (180 views, 27.11 dB).}
    \begin{tabular}{crrrr}
    \hline
    \multicolumn{1}{l}{180 views, 27.11 dB} & \multicolumn{1}{c}{1-SSIM} & \multicolumn{1}{c}{PSNR} & \multicolumn{1}{c}{MSE} & \multicolumn{1}{c}{Max\_Diff} \\
    \hline
    60 keV &  {1.82$\times 10^{-5}$} &  {69.51} &  {1.12$\times 10^{-7}$} &  {3.22$\times 10^{-3}$} \\
    100 keV &  {1.78$\times 10^{-5}$} &  {70.01} &  {9.96$\times 10^{-8}$} &  {4.36$\times 10^{-3}$} \\
    \hline
    \end{tabular}
  \label{tab:noise1e3}
\end{table}

Similarly, we plot the horizontal slices of the reconstructed 60/100 keV images by scheme \cref{eq:algo_proposed1} 
and the truths in \cref{horizontal_noise1e3_slice}. We can see that the discrepancy between them is not obvious. 
\begin{figure}[htbp]
    \centering\hspace{-3mm}
    \includegraphics[width=0.5\textwidth]{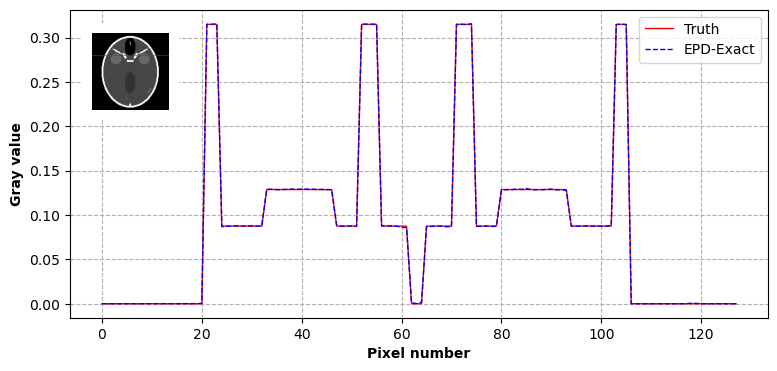}
    \includegraphics[width=0.5\textwidth]{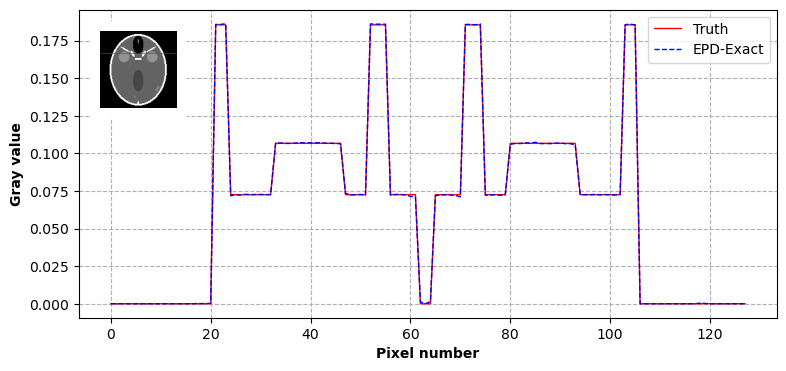}
    \caption{Horizontal slices of the reconstructed 60/100 keV results by scheme \cref{eq:algo_proposed1} and the truths for simulated noise data.}
    \label{horizontal_noise1e3_slice}
\end{figure}

\subsubsection{Test 3: sparse-view noise data}
\label{sec:test_suite_3} 

We test the proposed schemes with sparse-view noise data, which only contains 60 views for each low (80 kVp) 
and high (140 kVp) energy spectra respectively. The views of them are uniformly distributed over $[0,\pi)$ 
and $[\theta_0, \theta_0 +\pi)$ with $\theta_0 = \pi/120$. The noise data is generated as the way of the above test. 
The SNR of the data is about $27.14$ dB. We take regularization parameter $\lambda= 10^{-6}$, and 
choose zero images as the initial point, and take $\tau=\sigma_{\!\!K}=\sigma_{\!\!A}=0.2$. 
After $10^6$ iterations, the reconstructed result has achieved convergence,  
which is shown in \cref{fig:60angle_noise1e3_contrast}.  
\begin{figure}[htbp]
    \centering
    \includegraphics[trim={0 0 0 0},width=\textwidth]{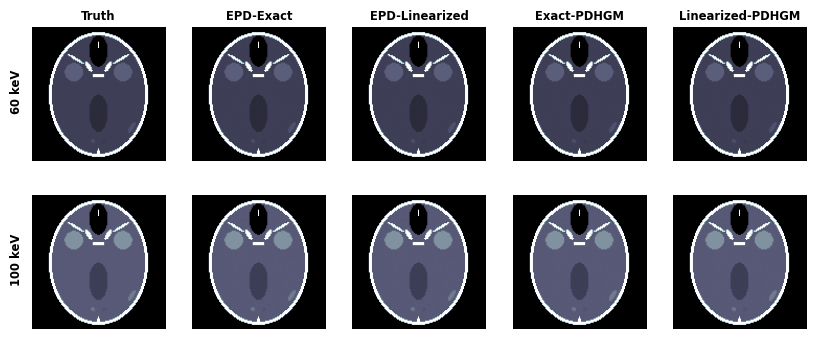}
     \vspace{-4mm}
    \caption{Reconstructed images after $10^6$ iterations using simulated sparse-view noise data (about 27.14 dB). 
    From left to right: truth images, reconstructed results by schemes \cref{eq:algo_proposed1}, \cref{eq:algo_proposed2}, 
     \cref{eq:Exact_NL_PDHGM} and \cref{eq:Linearized_NL_PDHGM}. 
     From top to bottom: 60 keV image and 100 keV image.} 
\label{fig:60angle_noise1e3_contrast}
\end{figure}

The convergent conditions are shown in \cref{fig:convergent_conds_sparse}, and the last 100 iterations are 
shown in the embedded subgraphs of \cref{fig:convergent_conds_sparse}. As we can see, the convergent 
conditions $RE_{\bdf}^{(n)}$, $RD_{\bdf}^{(n)}$ and $RT_{\bdf}^{(n)}$ have already converged. 
The performances of these schemes have only a little bit difference. 
\begin{figure}[htbp]
    \centering
    \includegraphics[width=0.33\textwidth]{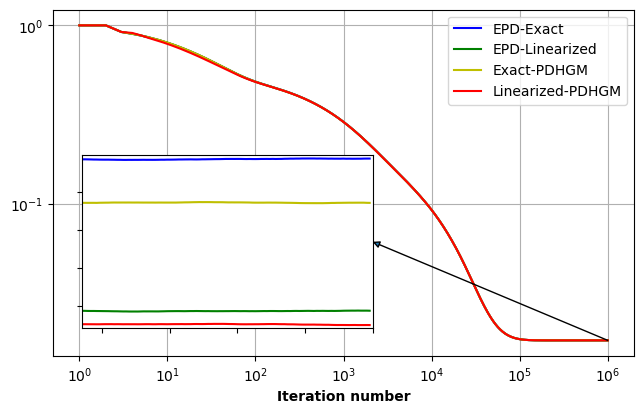}\hspace{-1mm}
    \includegraphics[width=0.33\textwidth]{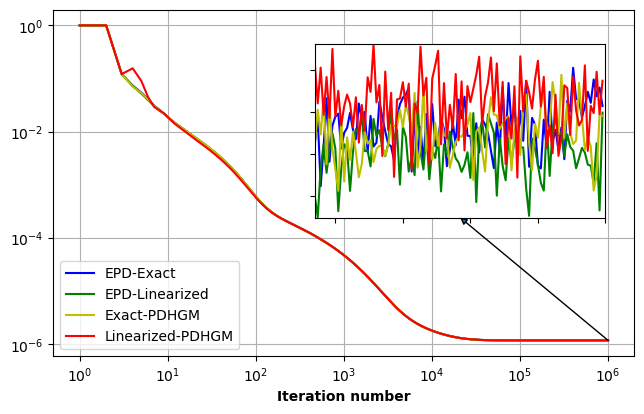}\hspace{-1mm}
    \includegraphics[width=0.33\textwidth]{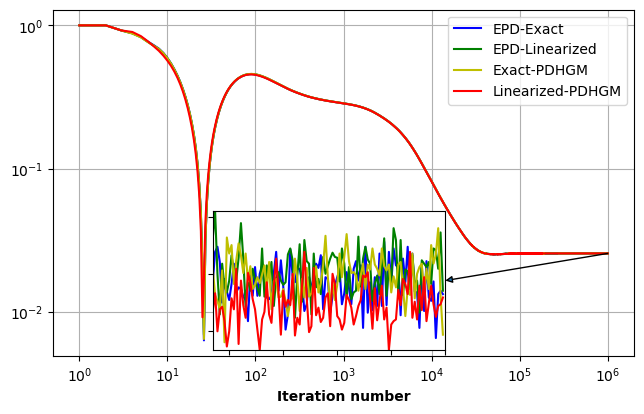}
    \caption{Convergent conditions $RE_{\bdf}^{(n)}$ (left), $RD_{\bdf}^{(n)}$ (middle) and $RT_{\bdf}^{(n)}$ (right) of schemes \cref{eq:algo_proposed1}, \cref{eq:algo_proposed2}, \cref{eq:Exact_NL_PDHGM} and \cref{eq:Linearized_NL_PDHGM} for simulated sparse-view noise data (about 27.14 dB),  as functions of iteration number $n$, where the subgraphs of the second row show the last 100 iterations.} 
    \label{fig:convergent_conds_sparse}
\end{figure} 

The quantitative indexes used to evaluate the image quality for scheme \cref{eq:algo_proposed1} are shown in \cref{tab:60angle1e3}, 
which illustrates good performance of the proposed scheme for sparse-view noise data. 
%The reconstructed effect is declined comparing with the case of more data \cref{tab:noise1e3}.
Furthermore, we plot the horizontal slices for the 60/100 keV images of scheme \cref{eq:algo_proposed1} and the truth images, 
which are shown in \cref{sparse_horizontal_noise1e3_slice}. We can see that the difference between them is tiny. 
\begin{table}[htbp]
  \centering 
  \caption{Quantitative indexes for the result by scheme \cref{eq:algo_proposed1} using simulated sparse-view noise data (60 views, 27.14 dB).} 
    \begin{tabular}{crrrr}
    \hline
    \multicolumn{1}{l}{60 views, 27.14 dB} & \multicolumn{1}{c}{1-SSIM} & \multicolumn{1}{c}{PSNR} & \multicolumn{1}{c}{MSE} & \multicolumn{1}{c}{Max\_Diff} \\
    \hline
    60 keV &  {7.75$\times 10^{-5}$} &  {63.35} &  {4.62$\times 10^{-7}$} &  {6.00$\times 10^{-3}$} \\
    100 keV &  {7.37$\times 10^{-5}$} &  {64.04} &  {3.94$\times 10^{-7}$} &  {1.03$\times 10^{-2}$} \\
    \hline
    \end{tabular}
  \label{tab:60angle1e3}%
\end{table}%

\begin{figure}[htbp]
    \centering\hspace{-3mm}
    \includegraphics[width=0.5\textwidth]{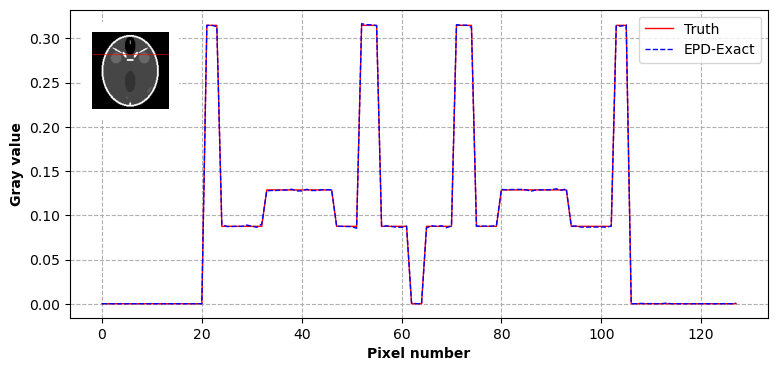}
    \includegraphics[width=0.5\textwidth]{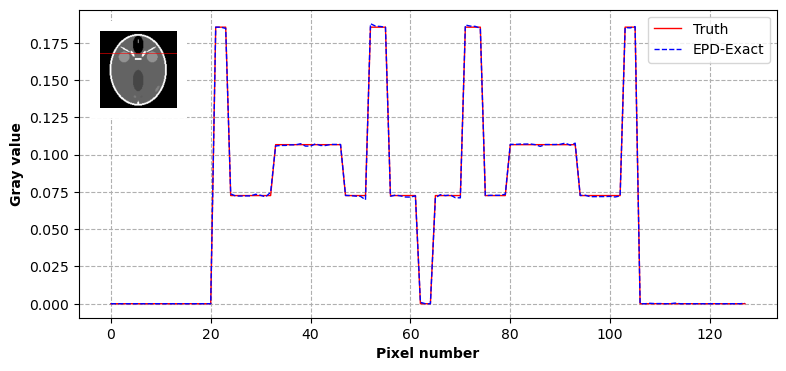}\vspace{-2mm}
    \caption{Horizontal slices of the reconstructed 60/100 keV results by scheme \cref{eq:algo_proposed1} and the truths for simulated sparse-view noise data.} 
    \label{sparse_horizontal_noise1e3_slice}
\end{figure}

\section{Conclusion}
\label{sec:conclusion}

In this work, we proposed an extended primal-dual algorithm framework 
in \cref{algo:proposed_NLPD_algorithm} for solving a class of general nonconvex problems 
in \cref{eq:opt_primal} with the operator $\bdK$ being nonlinear convex with respect to $\bdf$, 
which is motivated by the image reconstruction for several practical nonlinear imaging problems with such kind of forward 
operator. Note that model \cref{eq:opt_primal} also involves a kind of nonconvex sparse 
optimization model. Using the proposed framework, we put forward six different iterative  
schemes, where including the two that have been studied in the previous literatures \cite{va14,accle_Val19}. 
More importantly, we presented their complete mathematical explanation. 
We also established the relationship to existing algorithms, 
including primal-dual algorithm for convex problem, Exact/Linearized nonlinear \ac{PDHGM}, 
nonconvex primal-dual algorithm, and nonconvex \ac{ADMM}. With proper assumptions 
in \cref{subsec:Assumptions}, we proved the convergence of these schemes when the optimal dual 
variable regarding the nonlinear operator is non-vanishing, where the strong convexity assumption 
of $G$ is no longer required. 

To demonstrate the effectiveness of the proposed algorithm framework, we applied the generated iterative 
schemes to solve the image reconstruction problem in \cref{eq:specific_opt} for DECT. In particular, 
using special properties of the concrete problem, we proved the convergence of these customized 
schemes when the optimal dual variable regarding the nonlinear operator is vanishing. Moreover, 
the numerical experiments also showed that the proposed framework has good performance on the 
image reconstruction problem in DECT for various data with non-standard scanning configuration.

\appendix
\section{Theoretical proofs}\label{appendix:proofs} 

\subsection{Proof of \cref{lem:M_self_adjoint}}\label{proof:M_self_adjoint}
\begin{proof}
By \cref{eq:Mf}, we know $M(\bdf)$ is symmetric. For any $\bar{\bdbeta}=[\bar{\bdf}^{\tra}, \bar{\bdu}^{\tra}, \bar{\bdv}^{\tra}]^{\tra}\in \mathcal{X} \times \mathcal{Y}_K\times\mathcal{Y}_A$, we have
\begin{multline*}
    \bar{\bdbeta}^{\tra}\bigl( M(\bdf)-B_1(\bdf) \bigr)\bar{\bdbeta} 
    = \biggl\|\sqrt{\frac{s(1-\kappa)}{\tau}} \bar{\bdf} - \sqrt{\frac{\tau}{s(1-\kappa)}}[\nabla \bdK(\bdf)]^{\tra}\bar{\bdu}\biggr\|^2 \\ 
    + \biggl\|\sqrt{\frac{(1-s)(1-\kappa)}{\tau}} \bar{\bdf} - \sqrt{\frac{\tau}{(1-s)(1-\kappa)}}A^{\tra}\bar{\bdv}\biggr\|^2 \ge 0. 
\end{multline*}
Moreover, by Cauchy--Schwarz inequality, we obtain   
\[
 \bar{\bdbeta}^{\tra} B_1(\bdf) \bar{\bdbeta} 
    \ge  \frac{\delta}{\tau} \|\bar{\bdf}\|^2+ \frac{\|\bar{\bdu}\|^2}{\sigma_{\!\!K}}\biggl(1-\frac{\tau\sigma_{\!\!K}}{s(1-\kappa)}\|\nabla \bdK(\bdf)\|^2\biggr)+ \frac{\|\bar{\bdv}\|^2}{\sigma_{\!\!A}}\biggl(1-\frac{\tau\sigma_{\!\!A}}{(1-s)(1-\kappa)}\|A\|^2\biggr).
    \]
Using \cref{eq:step_assumption}, we have that the right side of the above inequality is nonnegative. Similarly, the following inequality is nonnegative
\begin{multline*}
    \bar{\bdbeta}^{\tra}\bigl( M(\bdf)- B_2(\bdf) \bigr)\bar{\bdbeta} 
    = \biggl\|\sqrt{\frac{\sigma_{\!\!K}}{1-\kappa}} \nabla \bdK(\bdf)\bar{\bdf} - \sqrt{\frac{1-\kappa}{\sigma_{\!\!K}}}\bar{\bdu}\biggr\|^2
    + \biggl\|\sqrt{\frac{\sigma_{\!\!A}}{1-\kappa}} A\bar{\bdf} - \sqrt{\frac{1-\kappa}{\sigma_{\!\!A}}}\bar{\bdv}\biggr\|^2. 
\end{multline*}
Using Cauchy--Schwarz inequality, we obtain   
\[
 \bar{\bdbeta}^{\tra} B_2(\bdf) \bar{\bdbeta} 
    \ge  \frac{\|\bar{\bdf}\|^2}{\tau}\biggl(1-\frac{\tau\sigma_{\!\!K}}{1-\kappa}\|\nabla \bdK(\bdf)\|^2-\frac{\tau\sigma_{\!\!A}}{1-\kappa}\|A\|^2\biggr) + \frac{\kappa}{\sigma_{\!\!K}}\|\bar{\bdu}\|^2+ \frac{\kappa}{\sigma_{\!\!A}}\|\bar{\bdv}\|^2 \ge 0.
    \]
Hence, we conclude the proof. 
\end{proof}

\subsection{Proof of \cref{lemma: alg_estimate} }\label{proof:alg_estimate}
\begin{proof}
 By the definition of $M_{n+1}^{\one}$,  
 the second term on the left side of \cref{eq:alg1_H_estimate} can be rewritten as 
\begin{equation}\label{eq:second_alg1_H_estimate}
    \frac{1}{2}\|\bdbeta^{n+1}-\widehat{\bdbeta}\|_{M_{n+2}^{\one}-M_{n+1}^{\one}}^{2}=\langle [\nabla \bdK(\bdf_{\theta}^{n})-\nabla \bdK(\bdf_{\theta}^{n+1})](\bdf^{n+1}-\whf),\bdu^{n+1}-\whu\rangle.
\end{equation}
%We consider the first term on the left side of \cref{eq:alg1_H_estimate}, by \cref{eq:acceleratexf_1} and \cref{eq:optimal_conditions}, 
%\begin{multline*}
%   \langle \wtH^{\one}_{n+1}(\bdbeta^{n+1}),\bdbeta^{n+1}-\widehat{\bdbeta}\rangle \\
%=  \langle\partial G(\bdf^{n+1}) +[\nabla \bdK(\bdf_{\theta}^n)]^{\tra}\bdu^{n+1}+A^{\tra}\bdv^{n+1} -\partial G(\whf)-[\nabla \bdK(\whf)]^{\tra} \whu-A^{\tra}\whv, \bdf^{n+1}-\whf \rangle \\
%+ \langle\partial F^{*}(\bdu^{n+1})+\nabla \bdK(\bdf_{\theta}^n)(\bdf^{n+1}-\bdf^n)-\bdK(\bdf_{\theta}^{n+1})-\partial F^{\ast}(\whu)+ \bdK(\whf), \bdu^{n+1}-\whu\rangle \\
%+\langle \partial E^{\ast}(\bdv^{n+1})-A\bdf^{n+1}-\partial E^{\ast}(\whv)+A\whf,\bdv^{n+1}-\whv \rangle .
%\end{multline*}
Using the convexity of $G$, $E^{\ast}$ and \cref{ass:monotone_of_F}, for the first term on the left 
side of \cref{eq:alg1_H_estimate}, we have  
\begin{multline}\label{eq:first_alg1_H_estimate}
   \langle \wtH^{\one}_{n+1}(\bdbeta^{n+1}),\bdbeta^{n+1}-\widehat{\bdbeta}\rangle
\geq  \gamma_{\!F^{\ast}}\|\bdu^{n+1}-\whu\|^2+\langle [\nabla \bdK(\bdf_{\theta}^n)]^{\tra}\bdu^{n+1}-[\nabla \bdK(\whf)]^{\tra} \whu, \bdf^{n+1}-\whf \rangle \\
+\langle \nabla \bdK(\bdf_{\theta}^n)(\bdf^{n+1}-\bdf^n)-\bdK(\bdf_{\theta}^{n+1})+\bdK(\whf), \bdu^{n+1}-\whu\rangle .
\end{multline}
Then, combining \cref{eq:first_alg1_H_estimate} with \cref{eq:second_alg1_H_estimate} leads to 
\begin{multline*}
     \langle \wtH^{\one}_{n+1}(\bdbeta^{n+1}),\bdbeta^{n+1}-\widehat{\bdbeta}\rangle-\frac{1}{2}\|\bdbeta^{n+1}-\widehat{\bdbeta}\|_{M_{n+2}^{\one}-M_{n+1}^{\one}}^{2}\\
\geq   \gamma_{F^{\ast}}\|\bdu^{n+1}-\whu\|^2
+\langle [\nabla \bdK(\bdf_{\theta}^n)-\nabla \bdK(\whf)](\bdf^{n+1}-\whf),\whu\rangle \\
 +\langle \bdK(\whf)-\bdK(\bdf_{\theta}^{n+1})-\nabla \bdK(\bdf_{\theta}^{n+1})(\whf-\bdf_{\theta}^{n+1}), \bdu^{n+1}-\whu\rangle \\
 +\langle [\nabla \bdK(\bdf_{\theta}^n)-\nabla \bdK(\bdf_{\theta}^{n+1})](\bdf^{n+1}-\bdf^n),\bdu^{n+1}-\whu\rangle .
\end{multline*}

Next, we estimate the above terms one by one. Using \cref{ass:nonlinear_restrict} and taking $\bx=\bdf_{\theta}^n$, 
$\bdf=\bdf^{n+1}$ and $\bdu=\bdu^{n+1}$, by Cauchy--Schwarz inequality, we have the following inequality
\begin{multline*}
\langle [\nabla \bdK(\bdf_{\theta}^n)-\nabla \bdK(\whf)](\bdf^{n+1}-\whf),\whu\rangle +\langle \bdK(\whf)-\bdK(\bdf^{n+1})-\nabla \bdK(\bdf^{n+1})(\whf-\bdf^{n+1}),\bdu^{n+1}-\whu\rangle \\
%\ge -\gamma_1\|\bdu^{n+1}-\whu\|^2 -\lambda_1\|\bdf^{n+1}-\bdf_{\theta}^n\|^{2}\\
\ge -\gamma_1\|\bdu^{n+1}-\whu\|^2 -2\lambda_1\|\bdf^{n+1}-\bdf^n\|^{2}-2\lambda_1\|\bdf^{n-1}-\bdf^n\|^{2}.
\end{multline*}
By mean value theorem, \cref{ass:locally_Lipschitz} and Cauchy--Schwarz inequality, we have
\begin{multline*}
    \langle [ \bdK(\whf)-\bdK(\bdf_{\theta}^{n+1})-\nabla \bdK(\bdf_{\theta}^{n+1})(\whf-\bdf_{\theta}^{n+1})]- [ \bdK(\whf)-\bdK(\bdf^{n+1})-\nabla \bdK(\bdf^{n+1})(\whf-\bdf^{n+1}) ], \bdu^{n+1}-\whu\rangle \\
%    =\langle \bdK(\bdf^{n+1})-\bdK(\bdf_{\theta}^{n+1})+[\nabla \bdK(\bdf^{n+1})-\nabla \bdK(\bdf_{\theta}^{n+1})](\whf-\bdf^{n+1})+\nabla \bdK(\bdf_{\theta}^{n+1})(\bdf_{\theta}^{n+1}-\bdf^{n+1}),\bdu^{n+1}-\whu \rangle \\
%    \ge -(2C_{\!K}+L\rho_{\bdf}) \|\bdf_{\theta}^{n+1}-\bdf^{n+1}\| \|\bdu^{n+1}-\whu\| \\
    \ge  -\tilde{C} \|\bdf^{n+1}-\bdf^{n}\|^2- \tilde{C}\|\bdu^{n+1}-\whu\|^2.
\end{multline*}
Using \cref{ass:locally_Lipschitz}, triangle inequality and Cauchy--Schwarz inequality, we obtain 
\begin{equation*}
   \langle [\nabla \bdK(\bdf_{\theta}^n)-\nabla \bdK(\bdf_{\theta}^{n+1})](\bdf^{n+1}-\bdf^n),\bdu^{n+1}-\whu\rangle \\
%    \ge - L\rho_{\bdu} \|\bdf_{\theta}^n-\bdf_{\theta}^{n+1}\| \|\bdf^{n+1}-\bdf^n\|\\
%    \ge  -L\rho_{\bdu} [2 \|\bdf^{n+1}-\bdf^n\|^2 + \frac12 (\|\bdf^{n+1}-\bdf^n\|^2+\|\bdf^{n-1}-\bdf^n\|^2)]\\
    \ge -\frac{5L\rho_{\bdu}}{2} \|\bdf^{n+1}-\bdf^n\|^2  -\frac{L\rho_{\bdu}}{2} \|\bdf^{n-1}-\bdf^n\|^2 .
\end{equation*} 
Considering the last term on the left side of \cref{eq:alg1_H_estimate}, and using \cref{lem:M_self_adjoint}, we obtain 
\begin{equation*}
     \frac{1}{2}\|\bdbeta^{n+1}-\bdbeta^{n}\|_{M_{n+1}^{\one}}^{2} 
\geq \frac{\kappa}{2\tau}\|\bdf^{n+1}-\bdf^{n}\|^2.
\end{equation*}
%By the above inequalities, adding the last term on the left of \cref{eq:alg1_H_estimate}, and using \cref{lem:M_self_adjoint}, we obtain 
%\begin{align*}
%     &\langle \wtH^{\one}_{n+1}(\bdbeta^{n+1}),\bdbeta^{n+1}-\widehat{\bdbeta}\rangle-\frac{1}{2}\|\bdbeta^{n+1}-\widehat{\bdbeta}\|_{M_{n+2}^{\one}-M_{n+1}^{\one}}^{2}+ \frac{1}{2}\|\bdbeta^{n+1}-\bdbeta^{n}\|_{M_{n+1}^{\one}}^{2}\\
%\geq&  (\gamma_{F^{\ast}}-\gamma_{1}-\tilde{C})\|\bdu^{n+1}-\whu\|^2
%-(\tilde{C}+2\lambda_1+\frac{5L\rho_{\bdu}}{2}) \|\bdf^{n+1}-\bdf^n\|^{2}-(2\lambda_1+\frac{L\rho_{\bdu}}{2})\|\bdf^{n-1}-\bdf^n\|^{2} \\
%&+ \frac{\kappa}{2\tau}\|\bdf^{n+1}-\bdf^{n}\|^2 +\Gamma_{n+1}(\bdf_{\theta}^n) \\
%= & (\gamma_{F^{\ast}}-\gamma_1-\tilde{C})\|\bdu^{n+1}-\whu\|^2 +\Lambda_{\one}(\tau) \|\bdf^{n+1}-\bdf^n\|^{2} -(2\lambda_1+\frac{L\rho_{\bdu}}{2}) \|\bdf^{n-1}-\bdf^n\|^{2}+ \Gamma_{n+1}(\bdf_{\theta}^n) .
%\end{align*}
By the above inequalities, we get \cref{eq:alg1_H_estimate} immediately. 

For algorithm \cref{eq:algo_proposed2}, similarly, we can estimate 
\begin{multline*}
     \langle \wtH^{\two}_{n+1}(\bdbeta^{n+1}),\bdbeta^{n+1}-\widehat{\bdbeta}\rangle-\frac{1}{2}\|\bdbeta^{n+1}-\widehat{\bdbeta}\|_{M_{n+2}^{\two}-M_{n+1}^{\two}}^{2} \\
     \geq    \gamma_{F^{\ast}} \|\bdu^{n+1}-\whu\|^2 
    + \langle [\nabla \bdK({\bdf^n})-\nabla \bdK(\whf)](\bdf^{n+1}-\whf),\whu\rangle \\
    +\langle \bdK(\whf)- \bdK(\bdf^{n+1})-\nabla \bdK(\bdf^{n+1})(\whf-\bdf^{n+1}), \bdu^{n+1}-\whu \rangle \\
+\langle [\nabla \bdK(\bdf^n) - \nabla \bdK(\bdf^{n+1})](\bdf^{n+1}-\bdf^n), \bdu^{n+1}-\whu \rangle.
\end{multline*}
Using \cref{ass:nonlinear_restrict}, and taking $\bx=\bdf^n,\bdf=\bdf^{n+1}$, $\bdu=\bdu^{n+1}$, we obtain 
\begin{multline*}
    \langle [\nabla \bdK({\bdf^n})-\nabla \bdK(\whf)](\bdf^{n+1}-\whf),\whu\rangle 
    +\langle \bdK(\whf)- \bdK(\bdf^{n+1})-\nabla \bdK(\bdf^{n+1})(\whf-\bdf^{n+1}), \bdu^{n+1}-\whu \rangle \\
    \ge -\gamma_1\|\bdu^{n+1}-\whu\|^2 
 -\lambda_1\|\bdf^{n+1}-\bdf^n\|^{2}.
\end{multline*}
By \cref{ass:locally_Lipschitz}, it is easy to obtain that  
\begin{equation*}
    \langle [\nabla \bdK(\bdf^n) - \nabla \bdK(\bdf^{n+1})](\bdf^{n+1}-\bdf^n), \bdu^{n+1}-\whu \rangle\ge - L\rho_{\bdu} \|\bdf^{n+1}-\bdf^n\|^{2}.
\end{equation*} 
Moreover,  
\[
\frac{1}{2}\|\bdbeta^{n+1}-\bdbeta^{n}\|_{M_{n+1}^{\two}}^{2} \geq \frac{\kappa}{2\tau}\|\bdf^{n+1}-\bdf^{n}\|^2 .
\]
Using the above results, we have \cref{eq:alg2_H_estimate}. 
%\begin{align*}
%    &\langle \wtH^{\two}_{n+1}(\bdbeta^{n+1}),\bdbeta^{n+1}-\widehat{\bdbeta}\rangle-\frac{1}{2}\|\bdbeta^{n+1}-\widehat{\bdbeta}\|_{M_{n+2}^{\two}-M_{n+1}^{\two}}^{2}+ \frac{1}{2}\|\bdbeta^{n+1}-\bdbeta^{n}\|_{M_{n+1}^{\two}}^{2}\\
%    \geq & (\gamma_{F^{\ast}}-\gamma_1)\|\bdu^{n+1}-\whu\|^2 
%-(\lambda_1+L\rho_{\bdu})\|\bdf^{n+1}-\bdf^n\|^{2}  + \frac{1}{2}\|\bdbeta^{n+1}-\bdbeta^{n}\|_{M_{n+1}^{\two}}^{2}\\
%\ge &  (\gamma_{F^{\ast}}-\gamma_1) \|\bdu^{n+1}-\whu\|^2 + (\frac{\kappa}{2\tau}-(\lambda_1+L\rho_{\bdu})) \|\bdf^{n+1}-\bdf^n\|^{2}+\Gamma_{n+1}(\bdf^n)\\ 
%= &  (\gamma_{F^{\ast}}-\gamma_1) \|\bdu^{n+1}-\whu\|^2 +\Lambda_{\two}(\tau)\|\bdf^{n+1}-\bdf^n\|^{2}+\Gamma_{n+1}(\bdf^n),
%\end{align*}
%which is our desired inequality.
\end{proof}

\subsection{Proof of \cref{feifei}}\label{proof:roughly_estimate}
\begin{proof} 
% By proof in \cref{lem:M_self_adjoint}, we know $\Gamma_{n+1}(\bdf)\ge0$.
Considering the first step of 
algorithm \cref{eq:algo_proposed1}, due to $\bar{\bdf}^0=\bdf^0$ and $\gamma_{F^{\ast}}-\gamma_1-\tilde{C} > 0$, by \cref{eq:alg1_H_estimate} and \cref{thm:basis convergence}, we get
\begin{align*}
    \frac{1}{2}\|\bdbeta^{1}-\widehat{\bdbeta}\|_{M_{2}^{\one}}^{2} &\leq \frac{1}{2}\|\bdbeta^{0}-\widehat{\bdbeta}\|_{M_{1}^{\one}}^{2} - \Lambda_{\one}(\tau) \|\bdf^{1}-\bdf^0\|^2.
\end{align*}
The assumption \cref{eq:step_restrict1} implies that $\Lambda_{\one}(\tau) > 2\lambda_1 + L\rho_{\bdu}/2$.
Then, when $N\ge 1$, by \cref{inq: N steps descent estimate}, we have 
\begin{multline*}
    \frac{1}{2}\|\bdbeta^{N+1}-\widehat{\bdbeta}\|_{M_{N+2}^{\one}}^{2}-\frac{1}{2}\|\bdbeta^{0}-\widehat{\bdbeta}\|_{M_{1}^{\one}}^{2} \\
%    \leq   -\Lambda_{\one}(\tau) \|\bdf^{1}-\bdf^0\|^2 - \sum_{n=1}^{N}  \left[\Lambda_{\one}(\tau) \|\bdf^{n+1}-\bdf^n\|^{2} - \Bigl(2\lambda_1+\frac{L\rho_{\bdu}}{2}\Bigr) \|\bdf^n-\bdf^{n-1}\|^{2}\right]\\
    \le   - \sum_{n=0}^{N-1} \biggl(\Lambda_{\one}(\tau)-\Bigl(2\lambda_1+\frac{L\rho_{\bdu}}{2}\Bigr)\biggr) \|\bdf^{n+1}-\bdf^n\|^{2} - \Lambda_{\one}(\tau) \|\bdf^{N+1}-\bdf^N\|^2 \le 0, 
\end{multline*}
which implies \cref{eq:alg1_iterates_estimate}.

Next, for algorithm \cref{eq:algo_proposed2}, by \cref{eq:alg2_H_estimate} and \cref{thm:basis convergence}, we obtain
\begin{equation*}
    \frac{1}{2}\|\bdbeta^{n+1}-\widehat{\bdbeta}\|_{M_{n+2}^{\two}}^{2} 
    \leq  \frac{1}{2}\|\bdbeta^{n}-\widehat{\bdbeta}\|_{M_{n+1}^{\two}}^{2} -\Lambda_{\two}(\tau) \|\bdf^{n+1}-\bdf^n\|^{2}. 
\end{equation*}
Then the condition \cref{eq:step_restrict2} implies that $\Lambda_{\two}(\tau) > 0$. Hence, 
%\begin{equation*}
%    \frac{1}{2}\|\bdbeta^{n+1}-\widehat{\bdbeta}\|_{M_{n+2}^{\two}}^{2} \leq \frac{1}{2} \|\bdbeta^{n}-\widehat{\bdbeta}\|_{M_{n+1}^{\two}}^2 \le \cdots\le \frac{1}{2}\|\bdbeta^{0}-\widehat{\bdbeta}\|_{M_{1}^{\two}}^2,
%\end{equation*}
\cref{eq:alg2_iterates_estimate} is proved. 
\end{proof}

\subsection{Proof of \cref{lemma:ensure_beta_n_neigborhood}}\label{proof:ensure_beta_O}
\begin{proof}
We begin the proof by estimating the distances from $\bdf^{n+1},\bdf^n$ to $\whf$.  
Using \cref{eq:optimal_conditions} yields that 
%$-[\nabla \bdK(\whf)]^{\tra}\whu - A^{\tra}\whv \in \partial G(\whf)$, 
%$\bdK(\whf)\in \partial F^{\ast}(\whu)$ and $A\whf \in \partial E^{\ast}(\whv)$. Then we have 
\begin{align}
    \label{eq:f_update_alg1}&\whf = (\Id+ \tau \partial G)^{-1}(\whf-\tau [\nabla \bdK(\whf)]^{\tra}\whu-\tau A^{\tra}\whv),\\
     \label{eq:u_update_alg1}&\whu = (\Id+ \sigma_{\!\!K} \partial F^{\ast})^{-1}\bigl(\whu+\sigma_{\!\!K} \bdK(\whf)\bigr),\\
     \label{eq:v_update_alg1}&\whv = (\Id+ \sigma_{\!\!A} \partial E^{\ast})^{-1}(\whv+\sigma_{\!\!A} A\whf ).
\end{align}
Using \cref{eq:updatef_1} and \cref{eq:f_update_alg1}, by \cref{nonexpan of prox}, we obtain
\begin{equation}\label{ineq: fn distance estimate}
    \|\bdf^{n+1}-\whf\|  
%    =& \|(\Id + \tau\partial G)^{-1}\bigl(\bdf^n - \tau[\nabla \bdK(\bdf_{\theta}^n)]^{\tra}\bdu^n-\tau A^{\tra}\bdv^n \bigl) - (\Id+ \tau \partial G)^{-1}(\whf-\tau [\nabla \bdK(\whf)]^{\tra}\whu-\tau A^{\tra}\whv) \| \notag \\
   % \le \|\bdf^n-\whf - \tau[\nabla \bdK(\bdf_{\theta}^n)]^{\tra}\bdu^{n}+\tau [\nabla \bdK(\whf)]^{\tra}\whu- \tau A^{\tra}(\bdv^{n}-\whv)\| \\
     \le \|\bdf^n-\whf\| + C_{\bdf}^n, 
\end{equation}
%\label{eq:updatef_1} \bdf^{n+1} &= (\Id + \tau\partial G)^{-1}\bigl(\bdf^n - \tau[\nabla \bdK(\bdf_{\theta}^n)]^{\tra}\bdu^n-\tau A^{\tra}\bdv^n \bigl)
where $C_{\bdf}^n:=\tau\|[\nabla \bdK(\bdf_{\theta}^n)]^{\tra}\bdu^{n}-[\nabla \bdK(\whf)]^{\tra}\whu+A^{\tra}(\bdv^{n}-\whv)\|$.
Similarly, we obtain the following results
\begin{equation*}
    \|\bdu^{n+1}-\whu\| \le \|\bdu^n-\whu\| + C_{\bdu}^n \quad\text{and}\quad \|\bdv^{n+1}-\whv\| \le \|\bdv^n-\whv\| + C_{\bdv}^n, 
\end{equation*}
where $C_{\bdu}^n:=\sigma_{\!\!K}\| \bdK(\bdf_{\theta}^{n+1}) - \bdK(\whf)\|$ and 
$C_{\bdv}^n:=\sigma_{\!\!A}\| A\bdf_{\theta}^{n+1} - A\whf\|$, respectively. 
Hence we prove a stronger conclusion $\{\bdbeta^{n}\}_{n\in\mathbb{N}}\subset \mathcal{O}(r_{\bdf}, r_{\bdu}, r_{\bdv})$ than $\{\bdbeta^{n}\}_{n\in\mathbb{N}}\subseteq \mathcal{O}(\rho_{\bdf},\rho_{\bdu},\rho_{\bdv})$ by induction. 

Firstly we consider the initial step, by $\bdbeta^0 \in \mathcal{O}(r_{\bdf}, r_{\bdu}, r_{\bdv})$,   
%\begin{align*}
%    C^0_{\bdf}&=\tau\|[\nabla \bdK({\bdf}^0)]^{\tra}\bdu^{0}-[\nabla \bdK(\whf)]^{\tra}\whu+A^{\tra}(\bdv^{0}-\whv)\| \\
%    &\le \tau(\|[\nabla \bdK({\bdf}^0)]^{\tra}(\bdu^{0}-\whu)+[\nabla \bdK ({\bdf}^0)-\nabla \bdK (\whf)]^{\tra}\whu\|+\|A^{\tra}(\bdv^{0}-\whv)\|) \\
%    &\le \tau \left( \|\nabla \bdK({\bdf}^0)\|  r_{\bdu} + (\|\nabla \bdK(\bdf^0)\|+\|\nabla \bdK(\whf)\|) \|\whu\|+\|A\| r_{\bdv} \right) \\
%    & \le C_{\bdf}, 
%\end{align*}
\[
    C^0_{\bdf} \le \tau(\|[\nabla \bdK({\bdf}^0)]^{\tra}(\bdu^{0}-\whu)+[\nabla \bdK ({\bdf}^0)-\nabla \bdK (\whf)]^{\tra}\whu\|+\|A^{\tra}(\bdv^{0}-\whv)\|) \le C_{\bdf}, 
\]
where $C_{\bdf} := \tau\left((r_{\bdu}+2\|\whu\|)C_{\!K}+\|A\|r_{\bdv}\right)$. 
By the convexity of $G$ and \cref{nonexpan of prox}, we have
\begin{align*}
    \langle (\bdf^n - \tau[\nabla \bdK(\bdf_{\theta}^n)]^{\tra}\bdu^{n}-\tau A^{\tra}\bdv^n)-(\whf-\tau [\nabla \bdK(\whf)]^{\tra}\whu-\tau A^{\tra}\whv),\bdf^{n+1}-\whf \rangle \ge \|\bdf^{n+1}-\whf\|^2 
\end{align*}
for all $n$.
% By the auxiliary identity,
%\begin{align*}
%    \|\bdf^{n+1}-\whf\|^2 - \langle \bdf^n -\whf,\bdf^{n+1}-\whf \rangle=\frac{1}{2}(\|\bdf^{n+1}-\whf\|^{2} + \|\bdf^{n+1}-\bdf^{n}\|^{2}-\|\bdf^{n}-\whf\|^{2}), 
%\end{align*}
By the above inequality, 
\begin{align*}
\|\bdf^{n+1}-\whf\|^{2} &+ \|\bdf^{n+1}-\bdf^{n}\|^{2}- \|\bdf^{n}-\whf\|^{2} \\
 &= 2( \|\bdf^{n+1}-\whf\|^2 - \langle \bdf^n -\whf,\bdf^{n+1}-\whf \rangle)  \\
& \le 2\tau \langle -[\nabla \bdK(\bdf_{\theta}^n)]^{\tra}\bdu^{n}+[\nabla \bdK(\whf)]^{\tra}\whu-A^{\tra}\bdv^n+A^{\tra}\whv,\bdf^{n+1}-\whf \rangle \\
& \le 2 C_{\bdf}^n \|\bdf^{n+1}-\whf\|, 
\end{align*}
which holds for all $n$. 

By the above estimate and \cref{ineq: fn distance estimate}, for $\bdf_{\theta}^1$, we have
\begin{align*}
    \|\bdf_{\theta}^1-\whf\|^2 &=2(\|\bdf^{1}-\whf\|^2 + \|\bdf^{1}-\bdf^{0}\|^2)-\|\bdf^0-\whf\|^2 \\
    %&\le 2(\|\bdf^0 -\whf\|^2 + 2 C_{\bdf}^0 \|\bdf^{1}-\whf\|)-\|\bdf^0-\whf\|^2\\
    & \le 4 C_{\bdf}^0 \|\bdf^{1}-\whf\| +\|\bdf^0-\whf\|^2\\
    &\le 4 C_{\bdf}^0 (\|\bdf^{0}-\whf\|+C_{\bdf}^0) +\|\bdf^0-\whf\|^2\\
    &=(\|\bdf^{0}-\whf\|+2 C_{\bdf}^0)^2\\
    &\le (r_{\bdf}+ 2 C_{\bdf})^2,
\end{align*}
which illustrates $\bdf_{\theta}^1\in \mathcal{B}_{\whf}(r_{\bdf}+2C_{\bdf})$. Therefore,
\[
     C_{\bdu}^0 =\sigma_{\!\!K}\| \bdK(\bdf_{\theta}^1) - \bdK(\whf)\| 
     \le \sigma_{\!\!K} (r_{\bdf}+2C_{\bdf}) C_{\!K} = C_{\bdu}.
\]
Similarly, $C_{\bdv}^0\le C_{\bdv}$. Thus, when $n=1$, the induction proposition holds. Now suppose that induction assumption holds with $1\le n\le N$, which is 
\begin{align}\label{induction assumption}
    \bdbeta^{n}\in \mathcal{O}(r_{\bdf}, r_{\bdu}, r_{\bdv}),\quad C_{\bdf}^{n-1}\le C_{\bdf},\quad C_{\bdu}^{n-1}\le C_{\bdu}, \quad C_{\bdv}^{n-1} \le C_{\bdv}.
\end{align}
Next we proof the case for $n=N+1$. One needs to pay attention to the order of proof here. 
Firstly, 
\begin{align*}
    C^N_{\bdf}&=\tau\|[\nabla \bdK({\bdf}^N)]^{\tra}\bdu^{N}-[\nabla \bdK(\whf)]^{\tra}\whu+A^{\tra}(\bdv^{N}-\whv)\| \\
    &\le \tau(\|[\nabla \bdK({\bdf}^N)]^{\tra}(\bdu^{N}-\whu)+[\nabla \bdK ({\bdf}^N)-\nabla \bdK (\whf)]^{\tra}\whu\|+\|A^{\tra}(\bdv^{N}-\whv)\|) \\
    %&\le \tau \left( \|\nabla \bdK({\bdf}^N)\|  r_{\bdu} + (\|\nabla \bdK(\bdf^N)\|+\|\nabla \bdK(\whf)\|) \|\whu\|+\|A\| r_{\bdv} \right) \\
    & \le \tau\left((r_{\bdu}+2\|\whu\|)C_{\!K}+\|A\|r_{\bdv}\right) = C_{\bdf}.
\end{align*}
Then, 
%\begin{align*}
%    \|\bdf_{\theta}^{N+1}-\whf\|^2 &=2(\|\bdf^{N+1}-\whf\|^2 + \|\bdf^{N+1}-\bdf^{N}\|^2)-\|\bdf^N-\whf\|^2 \\
%    %&\le 2(\|\bdf^N -\whf\|^2 + 2 C_{\bdf}^N \|\bdf^{1}-\whf\|)-\|\bdf^N-\whf\|^2\\
%    & \le 4 C_{\bdf}^N \|\bdf^{N+1}-\whf\| +\|\bdf^N-\whf\|^2\\
%    &\le 4 C_{\bdf}^N (\|\bdf^{N}-\whf\|+C_{\bdf}^N) +\|\bdf^N-\whf\|^2\\
%    &=(\|\bdf^{N}-\whf\|+2 C_{\bdf}^N)^2\\
%    &\le (r_{\bdf}+ 2 C_{\bdf})^2. 
%\end{align*}
\[
    \|\bdf_{\theta}^{N+1}-\whf\|^2 =2(\|\bdf^{N+1}-\whf\|^2 + \|\bdf^{N+1}-\bdf^{N}\|^2)-\|\bdf^N-\whf\|^2 
    %&\le 2(\|\bdf^N -\whf\|^2 + 2 C_{\bdf}^N \|\bdf^{1}-\whf\|)-\|\bdf^N-\whf\|^2\\
%    & \le 4 C_{\bdf}^N \|\bdf^{N+1}-\whf\| +\|\bdf^N-\whf\|^2\\
%    &\le 4 C_{\bdf}^N (\|\bdf^{N}-\whf\|+C_{\bdf}^N) +\|\bdf^N-\whf\|^2\\
%    &=(\|\bdf^{N}-\whf\|+2 C_{\bdf}^N)^2\\
    \le (r_{\bdf}+ 2 C_{\bdf})^2. 
\]
Similarly, we obtain 
\[
     C_{\bdu}^N =\sigma_{\!\!K}\| \bdK(\bdf_{\theta}^{N+1}) - \bdK(\whf)\| 
     \le \sigma_{\!\!K} (r_{\bdf}+2C_{\bdf}) C_{\!K} = C_{\bdu}, 
\]
and $C_{\bdv}^N\le C_{\bdv}$. 
Together with the estimate \cref{ineq: fn distance estimate}, we have
\begin{align*}
    \|\bdf^{N+1}-\whf\| &\le \|\bdf^N-\whf\| + C_{\bdf}^N 
    \le r_{\bdf}+ C_{\bdf}^N \le r_{\bdf}+ C_{\bdf},
\end{align*}
which shows that $\bdf^{N+1}\in \mathcal{B}_{\whf}(r_{\bdf}+C_{\bdf})$. Similarly, we have
$\bdu^{N+1}\in \mathcal{B}_{\whu}(r_{\bdu}+ C_{\bdu})$ and $\bdv^{N+1}\in \mathcal{B}_{\whv}(r_{\bdv}+ C_{\bdv})$. These give $\bdbeta^{N+1}\in \mathcal{O}(\rho_{\bdf},\rho_{\bdu},\rho_{\bdv})$. Then by \cref{feifei}, we obtain $\|\bdbeta^{N+1}-\widehat{\bdbeta}\|_{M_{N+2}^{\one}} \le \|\bdbeta^{0}-\widehat{\bdbeta}\|_{M_{1}^{\one}}.$ Using \cref{lem:M_self_adjoint}, we have
\begin{align*}
    \|\bdbeta^{N+1}-\widehat{\bdbeta}\|_{M_{N+2}^{\one}}^2 
    \ge&
    \max\biggl\{ \|\bdbeta^{N+1}-\widehat{\bdbeta}\|_{B_1(\bdf^{N+1}_{\theta})}^2,\|\bdbeta^{N+1}-\widehat{\bdbeta}\|_{B_2(\bdf^{N+1}_{\theta})}^2\biggr\} \\
    \ge& \max\biggl\{ \frac{\kappa}{\tau}\|\bdf^{N+1}-\whf\|^2, \frac{\kappa}{\sigma_{\!\!K}} \|\bdu^{N+1}-\whu\|^2,\frac{\kappa}{\sigma_{\!\!A}}\|\bdv^{N+1}-\whv\|^2 \biggr\}.
\end{align*}
Thus, together with \cref{eq:initial_point_assumption}, we obtain 
\begin{align*}
    \|\bdf^{N+1}-\whf\|\le \frac{\|\bdbeta^{N+1}-\widehat{\bdbeta}\|_{M_{N+2}^{\one}}}{\sqrt{\kappa /\tau}}
    \le\frac{\|\bdbeta^0-\widehat{\bdbeta}\|_{M_1^{\one}}}{\sqrt{\kappa/\tau}}\le r_{\bdf}.
\end{align*}
Similarly, we have
\begin{align*}
    \|\bdu^{N+1}-\whu\|\le \frac{\|\bdbeta^{N+1}-\widehat{\bdbeta}\|_{M_{N+2}^{\one}}}{\sqrt{\kappa /\sigma_{\!\!K}}}
    \le\frac{\|\bdbeta^0-\widehat{\bdbeta}\|_{M_1^{\one}}}{\sqrt{\kappa /\sigma_{\!\!K}}}\le r_{\bdu},
\end{align*}
and $\|\bdv^{N+1}-\whv\|\le r_{\bdv}$. Therefore $\bdbeta^{N+1}\in \mathcal{O}(r_{\bdf}, r_{\bdu}, r_{\bdv})$, and the assumption of induction holds. Hence we complete the proof.
\end{proof}

\subsection{Proof of \cref{thm:PD alg converge}}\label{proof:convergence_thm}
\begin{proof}
By \cref{lemma:ensure_beta_n_neigborhood}, 
iterate points $\bdbeta^n, \bdbeta^{n+1} \in \mathcal{O}(\rho_{\bdf},\rho_{\bdu},\rho_{\bdv})$ 
and $\bdf_{\theta}^{n+1} \in \mathcal{B}_{\whf}(\rho_{\bdf})$. Using \cref{lemma: alg_estimate}, 
we obtain \cref{eq:alg1_H_estimate}. By \cref{thm:basis convergence}, for all $n\ge 1,$
\begin{multline*}
    \frac{1}{2}\|\bdbeta^{n}-\widehat{\bdbeta}\|_{M_{n+1}^{\one}}^{2}- \frac{1}{2}\|\bdbeta^{n+1}-\widehat{\bdbeta}\|_{M_{n+2}^{\one}}^{2} 
   \ge  (\gamma_{F^{\ast}}-\gamma_1-\tilde{C}) \|\bdu^{n+1}-\whu\|^2 \\
   +\Lambda_{\one}(\tau)\|\bdf^{n+1}-\bdf^n\|^{2} -(2\lambda_1+\frac{L\rho_{\bdu}}{2}) \|\bdf^{n-1}-\bdf^n\|^{2}.
\end{multline*}
Specially for the case of $n=0$, since $\bar{\bdf}^0=\bdf^0$, we have
\begin{equation*}
   \frac{1}{2}\|\bdbeta^{0}-\widehat{\bdbeta}\|_{M_{1}^{\one}}^{2}- \frac{1}{2}\|\bdbeta^{1}-\widehat{\bdbeta}\|_{M_{2}^{\one}}^{2} 
  \ge  (\gamma_{F^{\ast}}-\gamma_1) \|\bdu^{1}-\whu\|^2 + \Lambda_{\one}(\tau) \|\bdf^{1}-\bdf^0\|^2.
\end{equation*}
Summarizing the above estimate inequalities over $n=0,\ldots,N$, we obtain
\begin{align*}
    \frac{1}{2}\|\bdbeta^{0}&-\widehat{\bdbeta}\|_{M_{1}^{\one}}^{2} -\frac{1}{2}\|\bdbeta^{N+1}-\widehat{\bdbeta}\|_{M_{N+2}^{\one}}^{2}\\
    \ge &  (\gamma_{F^{\ast}}-\gamma_1-\tilde{C})\sum_{n=0}^N   \|\bdu^{n+1}-\whu\|^2+\Lambda_{\one}(\tau) \|\bdf^{1}-\bdf^0\|^2\\
    &+\sum_{n=1}^{N}  \left[\Lambda_{\one}(\tau) \|\bdf^{n+1}-\bdf^n\|^{2} -(2\lambda_1+\frac{L\rho_{\bdu}}{2}) \|\bdf^n-\bdf^{n-1}\|^{2}\right]\\
    = & (\gamma_{F^{\ast}}-\gamma_1-\tilde{C})\sum_{n=0}^N   \|\bdu^{n+1}-\whu\|^2+ \Lambda_{\one}(\tau) \|\bdf^{N+1}-\bdf^N\|^2\\
    &+\biggl(\Lambda_{\one}(\tau)-(2\lambda_1+\frac{L\rho_{\bdu}}{2})\biggr)\sum_{n=0}^{N-1}   \|\bdf^{n+1}-\bdf^n\|^{2}.
\end{align*}
% By \cref{eq:step_assumption}, $\Gamma_{n+1}(\bdf^n_{\theta})\ge c(\|\bdu^{n+1}-\bdu^{n}\|^2+\|\bdv^{n+1}-\bdv^{n}\|^2)$ for some positive constant $c$. 
Recall that $\bdbeta^{N+1} \in \mathcal{O}(r_{\bdf}, r_{\bdu}, r_{\bdv})$ for all $N\in \mathbb{N}$ in \cref{lemma:ensure_beta_n_neigborhood} and \cref{eq:initial_point_assumption}, we obtain that the left side of above inequality is bounded. Thus
\begin{align*}
   \biggl(\Lambda_{\one}(\tau)-(2\lambda_1+\frac{L\rho_{\bdu}}{2})\biggr) \sum_{n=0}^{N-1}   \|\bdf^{n+1}-\bdf^n\|^{2} < \infty.
\end{align*}
The condition \cref{eq:step_restrict1} implies that $\Lambda_{\one}(\tau) > (2\lambda_1+\frac{L\rho_{\bdu}}{2})$, consequently,
\begin{align*}
      \sum_{n=0}^{N-1}  \|\bdf^{n+1}-\bdf^n\|^{2} < \infty,
\end{align*}
%Let $N\rightarrow \infty$, we obtain that for any $\epsilon>0$, there exists $N_0$, for any $m\ge 0$ 
%\begin{align*}
%      \sum_{n=N_0}^{N_0+m}  \|\bdf^{n+1}-\bdf^n\|^{2} < \infty.
%\end{align*}
%Thus for any $n_1>n_2\ge N_0$, by triangle inequality,
%\begin{align*}
%    \|\bdf^{n_1}-\bdf^{n_2}\|^2\le \sum_{n=n_2}^{n_1 -1}  \|\bdf^{n+1}-\bdf^n\|^{2} < \epsilon,
%\end{align*}
which illustrates that $\{\bdf^n\}_{n\in\mathbb{N}}$ is a Cauchy sequence. 
Since $\mathcal{X}$ is a finite dimensional Euclidean space, it follows 
that $\{\bdf^n\}_{n\in\mathbb{N}}$ converges to some $\whf^{\prime}\in\mathcal{B}_{\whf}(r_{\bdf})$. 
Similarly, $\{\bdv^n\}_{n\in\mathbb{N}}$ converges to $\whv^{\prime}\in\mathcal{B}_{\whv}(r_{\bdv})$. 
Moreover, by strongly monotone factor $\gamma_{F^{\ast}} >\gamma_1+\tilde{C}$, we obtain
$\{\bdu^n\}_{n\in\mathbb{N}}$ converges to $\whu$. Recall the iteration formula of \cref{eq:algo_proposed1}, let $N\rightarrow\infty$, we obtain
\begin{equation*}
0\in \begin{bmatrix}
 \partial G(\whf^{\prime})+[\nabla \bdK(\whf^{\prime})]^{\tra} \whu +A^{\tra}\whv^{\prime} \\
\partial F^{*}(\whu)-\bdK(\whf^{\prime})\\
\partial E^{*}(\whv^{\prime})-A(\whf^{\prime})
\end{bmatrix},
\end{equation*}
 which is the desired result.
\end{proof} 

\subsection{Proof of \cref{thm:convex_lips_g}}\label{sec:convex_lips_g_spec_CT}
\begin{proof}
Obviously, $K_j$ is Fr\'echet differentiable. Let $\bdb_{\!d} = \left[b_{d1}, \ldots, b_{dM}\right]^{\tra} \in \Real^{M}$ for $d = 1,\ldots, D$. 
We further define 
\begin{equation*}
z_{jm}(\bdf):=-\sum_{d=1}^D b_{dm}\bda_j^{\tra} \bdf_{\!\!d}, \quad   
\bz_j(\bdf):= [z_{j1}(\bdf), \ldots, z_{jM}(\bdf)]^{\tra} 
    %\bz_j(\bdf):=-\sum_{d=1}^D \bdb_{d}\bda_j^{\tra}\bdf_{\!\!d}, 
\end{equation*}
and sometimes use $z_{jm}$ and $\bz_j$ without $\bdf$ for simplicity.

 Let 
 \[
 \exp(\bz_j):=[\exp(z_{j1}),\ldots,\exp(z_{jM})]^{\tra}.
 \] 
 Then we can reformulate \cref{eq:log_sum_exp} as $K_j(\bdf) = \ln\sum_{m=1}^{M} s_{jm} \exp(z_{jm})$. Direct calculation yields
\begin{equation}\label{eq:bar_bdomega}
 \frac{\partial K_j(\bdf)}{\partial \bz_j} = \frac{\bds_j\odot\exp(\bz_j)}{\bds_j^{\tra}\exp(\bz_j)} =: \bar{\bdomega}_{j}(\bdf),
\end{equation}
where $\odot$ denotes the element-wise product between two same-dimensional vectors. By the chain rule, 
\begin{equation}\label{eq:gradient_Rj}
    \nabla_{\bdf_{\!\!d}}K_j(\bdf) = \sum_{m=1}^{M}\frac{\partial K_j(\bdf)}{\partial z_{jm}} \frac{\partial z_{jm}}{\partial \bdf_{\!\!d}}
    = -\bigl(\bar{\bdomega}_{j}(\bdf)^{\tra}\bdb_{\!d}\bigr) \bda_j.
\end{equation}
To compute the Hessian matrix of $K_j(\bdf)$, we further let 
\begin{equation}\label{eq:zeta_jm}
\zeta_{jm} := s_{jm}\exp(z_{jm}), 
\end{equation} 
and derive  
\[
    \nabla_{\bdf_{\!\!d} \bdf_{\!\!e}}^2K_j(\bdf) = \frac{\bda_j\bda_j^{\tra}}{\bigl(\bds_j^{\tra}\exp(\bz_j)\bigr)^2}
    \sum_{m<n}^M\zeta_{jm}\zeta_{jn}(b_{dm}-b_{dn})(b_{em}-b_{en})
    =:\alpha_{de}(\bdf)\bda_j\bda_j^{\tra}.
\]
% By the definition of $\zeta_{jm}$ and $\sum_{m=1}^M s_{jm}=1$ and $z_{jm}<0$., we have $0<\zeta_{jm}<1$, together with \cref{assumption: bkm}, we can conclude that $$\alpha_{de}(\bdf)>0.$$
% Then the Hessian matrix of $K_j(\bdf)$ has the following special form  
% \begin{align*}%\Large
% \nabla^2 K_j(\bdf)=
%     \begin{bmatrix}
%      \alpha_{11}(\bdf)\bda_{j} \bda_{j}^{\tra} &\alpha_{12}(\bdf)\bda_{j} \bda_{j}^{\tra}\\
%     \alpha_{21}(\bdf)\bda_{j} \bda_{j}^{\tra}&\alpha_{22}(\bdf)\bda_{j} \bda_{j}^{\tra}\\
%   \end{bmatrix}.
% \end{align*}
For any $\bdu=[\bdu_1^{\tra},\ldots,\bdu_D^{\tra}]^{\tra}$ with $\bdu_d\in \Real^N$ for $d=1,\ldots,D$,
\begin{equation*}%\Large
 \bdu^{\tra}\nabla^2 K_j(\bdf)\bdu \\
%= \sum_{d,e}^D \alpha_{de}(\bdf) \bda_j^{\tra}u_d \bda_j^{\tra}u_e \\
  =\frac{1}{\bigl(\bds_j^{\tra}\exp(\bz_j)\bigr)^2}\sum_{m<n}^M \sum_{d,e}^D \zeta_{jm}\zeta_{jn} \Bigl[\bda_j^{\tra}\bdu_d (b_{dm}-b_{dn})+\bda_j^{\tra}\bdu_e (b_{em}-b_{en})\Bigl]^2 \geq 0, 
\end{equation*}
which illustrates that $\nabla^2 K_j(\bdf)$ is positive semidefinite. Hence, $K_j(\bdf)$ is convex.
% Moreover, by the special form of Hessian $\nabla^2 K_j(\bdf)$, assume $\bdu =[u_1,\cdots ,u_K]^{\tra}$ be its eigenvector associated with eigenvalue $\mu$, which is
% \begin{align*}
%     \begin{bmatrix}
%      \alpha_{11}(\bdf)\bda_{j} \bda_{j}^{\tra} &\cdots &\alpha_{1D}(\bdf)\bda_{j} \bda_{j}^{\tra}\\
%      \vdots & \ddots & \vdots\\
%     \alpha_{D1}(\bdf)\bda_{j} \bda_{j}^{\tra}& \cdots &\alpha_{DD}(\bdf)\bda_{j} \bda_{j}^{\tra}\\
%   \end{bmatrix}
%   \begin{bmatrix}
%   u_1\\
%   \vdots\\
%   u_D
%   \end{bmatrix}=\mu \begin{bmatrix}
%   u_1\\
%   \vdots\\
%   u_D
%   \end{bmatrix}.
% \end{align*}
% Directly calculation indicates that  all positive eigenvalue of $\nabla^2 K_j(\bdf)$ are $\|\bda_j\|^2$ times of all eigenvalue of the following real symmetric matrix
% \begin{align*}
%  \begin{bmatrix}
%      \alpha_{11}(\bdf) &\cdots &\alpha_{1K}(\bdf)\\
%      \vdots &\ddots & \vdots\\
%     \alpha_{K1}(\bdf)& \cdots &\alpha_{KK}(\bdf) \\
%   \end{bmatrix}.
% \end{align*}
Subsequently, for any $\bdf\in \mathbb{R}^{DN}$, the maximum eigenvalue of $\nabla K_j(\bdf)$ satisfies 
\[
    \mu_{\mathrm{max}} \le \|\bda_j\|^2\sum_{d=1}^D \alpha_{dd}(\bdf) <\infty, 
\] 
which shows that $\nabla K_j$ is global Lipschitz continuous. 
\end{proof}

\subsection{Proof of \cref{lem:R_f}}\label{sec:proof_R_f}
\begin{proof}
By \cref{eq:gradient_Rj}, we know  
\begin{align*}
    \nabla K_j(\bdf) = \begin{bmatrix}
     -\bigl(\bar{\bdomega}_{j}(\bdf)^{\tra}\bdb_1\bigr) \bda_j\\
     \vdots \\
     -\bigl(\bar{\bdomega}_{j}(\bdf)^{\tra}\bdb_{\!D}\bigr) \bda_j
    \end{bmatrix}. 
%    =R_j(\bdf,\tilde{\bdf})\begin{bmatrix}
%     -\bar{\bdomega}_j(\tilde{\bdf})^{\top}\bdb_1 \bda_j\\
%     \vdots\\
%     -\bar{\bdomega}_j(\tilde{\bdf})^{\top}\bdb_{\!D} \bda_j
%    \end{bmatrix}. 
\end{align*}
Hence, there is a solution to \cref{eq:Kj_Rj} as 
\begin{equation*}
   R_j(\bdf,\tilde{\bdf}) = \diag\biggl( \frac{\bar{\bdomega}_j(\bdf)^{\top}\bdb_1}{\bar{\bdomega}_j(\tilde{\bdf})^{\top}\bdb_1} \Id, \ldots, \frac{\bar{\bdomega}_j(\bdf)^{\top}\bdb_{\!d}}{\bar{\bdomega}_j(\tilde{\bdf})^{\top}\bdb_{\!D}} \Id \biggr). 
%   = \begin{bmatrix}
%     \frac{\bar{\bdomega}_j(\bdf)^{\top}\bdb_1}{\bar{\bdomega}_j(\tilde{\bdf})^{\top}\bdb_1} I_N && \\
%     & \ddots& \\
%       & &  \frac{\bar{\bdomega}_j(\bdf)^{\top}\bdb_{\!d}}{\bar{\bdomega}_j(\tilde{\bdf})^{\top}\bdb_{\!d}} I_N \\
 %   \end{bmatrix}
\end{equation*}
 
%where $I_N$ denotes $N\times N$ identity matrix, and $R(\bdf,\tilde{\bdf})=\diag(R_j(\bdf,\tilde{\bdf}))$. 

Due to the soft-max structure of $\bar{\bdomega}_j(\bdf)$ in \cref{eq:bar_bdomega} and positivity of $\bdb_{\!d}$, 
\[
    \frac{\min_m b_{dm}}{\max_m b_{dm}}\le \frac{\bar{\bdomega}_j(\bdf)^{\top}\bdb_{\!d}}{\bar{\bdomega}_j(\tilde{\bdf})^{\top}\bdb_{\!d}} \le \frac{\max_m b_{dm}}{\min_m b_{dm}}
\] 
for any $\bdf,\tilde{\bdf}\in \mathbb{R}^{DN}_{+}$,  
which implies that the eigenvalues of $R_j(\bdf,\tilde{\bdf})$ have uniform bound. 
By mean value theorem and Cauchy--Schwarz inequality, we have
\begin{align*}
    \|R_j(\bdf,\tilde{\bdf}) - \Id\|_2 
%  =  \max_d\{|\frac{\bar{\bdomega}_j(\bdf)^{\top}\bdb_{\!d}}{\bar{\bdomega}_j(\tilde{\bdf})^{\top}\bdb_{\!d}}-1|\} \\
%    &\le \|\bar{\bdomega}_j(\bdf)-\bar{\bdomega}_j(\tilde{\bdf}) \|  \max_{d} \frac{ \|\bdb_{\!d}\|}{\min_m b_{dm}}\\
    &\le \max_{d} \frac{ \|\bdb_{\!d}\|}{\min_m b_{dm}} \|\nabla \bar{\bdomega}_j(\xi\bdf+(1-\xi)\tilde{\bdf})\| \|\bdf-\tilde{\bdf}\|.
\end{align*}
Furthermore, by the definitions in \cref{eq:bar_bdomega} and \cref{eq:zeta_jm}, for any $\bdf\in \Real^{DN}_+$, we have 
\begin{align*}
    \|\nabla \bar{\bdomega}_j(\bdf)\|
    %&=\sum_{m=1}^M \sum_{d=1}^D\|(\nabla_{\bdf_{\!\!d}}\bar{\bdomega}_{j}(\bdf))_m\|^2 \\
    & \le \sum_{d=1}^D\sum_{m=1}^M \sum_{n=1}^M \frac{\zeta_{jm}\zeta_{jn}|b_{dn}-b_{dm}|\|\bda_j\|}{\Bigl(\bds_j^{\tra}\exp(\bz_j)\Bigr)^2} \\
    &\le \sum_{d=1}^D\max_{m,n} |b_{dm}-b_{dn}|\|\bda_j\|. 
\end{align*}
By the definition of $R(\bdf,\tilde{\bdf})$, and using the above estimates, we have 
\begin{align*}
    \|R(\bdf,\tilde{\bdf})-I\|_2&= \max_{j}\|R_j(\bdf)-I\|_2 \\
    &\le \sum_{d=1}^D\max_{m,n} |b_{dm}-b_{dn}| \max_{d}\frac{\|\bdb_{\!d}\|}{\min_m b_{dm}}\max_j \|\bda_j\| \|\bdf-\tilde{\bdf}\|, 
\end{align*}
which concludes the proof. 
\end{proof}

\subsection{Proof of \cref{lem:local_property_g_hat}}\label{sec:proof_local_property_g_hat}
\begin{proof}
By \cref{lem:R_f}, we have
\begin{align*}
\|\bdK(\bdf)-\bdK(\tilde{\bdf})-\nabla \bdK(\tilde{\bdf})(\bdf-\tilde{\bdf})\| &= \biggl\|\int_{0}^{1} (\nabla \bdK(\bdf_{\!\!t})-\nabla \bdK(\tilde{\bdf}))(\bdf-\tilde{\bdf}) ~\mathrm{d} t \biggr\| \\
& \leq \int_{0}^{1} \|\bigl(R(\bdf_{\!\!t},\tilde{\bdf})-I\bigr)\nabla \bdK(\tilde{\bdf})(\bdf-\tilde{\bdf})\| ~\mathrm{d} t\\
& \leq \biggl(\int_{0}^{1} C_R \|\bdf_{\!\!t}-\tilde{\bdf}\|~\mathrm{d}t\biggr) \|\nabla\bdK(\tilde{\bdf})(\bdf-\tilde{\bdf})\| \\
& = \frac12 C_R\|\bdf-\tilde{\bdf}\|\|\nabla \bdK(\tilde{\bdf})(\bdf-\tilde{\bdf})\|, 
\end{align*}
where $\bdf_{\!\!t}:=t\bdf+(1-t)\tilde{\bdf}$. Using $\bdf,\tilde{\bdf}\in\mathcal{B}_{\whf}(\rho)$ yields that 
\begin{equation*}
    \|\bdK(\bdf)-\bdK(\tilde{\bdf})-\nabla \bdK(\tilde{\bdf})(\bdf-\tilde{\bdf})\|\le \rho C_{R}\|\nabla \bdK(\tilde{\bdf})(\bdf-\tilde{\bdf})\|. 
\end{equation*}
Then we obtain 
\[
\|\nabla \bdK(\tilde{\bdf})(\bdf-\tilde{\bdf})\| \le \frac{1}{1-\rho C_{R}}\|\bdK(\bdf)-\bdK(\tilde{\bdf})\|,
\]
if $\rho < 1 / C_{R}$. Hence, \cref{local_property_g} is proved. 
\end{proof}

\subsection{Proof of \cref{thm:hat_u_eqs_0}}\label{proof:case_spectral_CT}

\begin{proof}
By the definition of $\wtH^{\one}_{n+1}$ and $M(\bdf)$, \cref{eq:Mf}, and using \cref{eq:specific_F_star}, $\whu=0$ and $G$ convex, we obtain
\begin{multline*}
     \langle \wtH^{\one}_{n+1}(\bdbeta^{n+1}),\bdbeta^{n+1}-\widehat{\bdbeta}\rangle-\frac{1}{2}\|\bdbeta^{n+1}-\widehat{\bdbeta}\|_{M_{n+2}^{\one}-M_{n+1}^{\one}}^{2}\\
\geq  \|\bdu^{n+1}\|^2 +\langle \bdK(\whf)-\bdK(\bdf_{\theta}^{n+1})-\nabla \bdK(\bdf_{\theta}^{n+1})(\whf-\bdf_{\theta}^{n+1}), \bdu^{n+1}\rangle \\
 +\langle [\nabla \bdK(\bdf_{\theta}^n)-\nabla \bdK(\bdf_{\theta}^{n+1})](\bdf^{n+1}-\bdf^n),\bdu^{n+1}\rangle. 
\end{multline*}
Now we assume that $\bdf_{\theta}^{n+1} \in \mathcal{B}_{\whf}(\rho_{\bdf})$ and $\bdu^{n+1}\in \mathcal{B}_{\whu}(\rho_{\bdu})$. 
Then by \cref{lem:local_property_g_hat} and \cref{eq:updated_1_explicit_u}, we have 
\begin{align*}
    \langle \bdK(\whf)&-\bdK(\bdf_{\theta}^{n+1})-\nabla \bdK(\bdf_{\theta}^{n+1})(\whf-\bdf_{\theta}^{n+1}), \bdu^{n+1}\rangle \\
   % &\ge -\|\bdK(\whf)-\bdK(\bdf_{\theta}^{n+1})-\nabla \bdK(\bdf_{\theta}^{n+1})(\whf-\bdf_{\theta}^{n+1}) \| \|\bdu^{n+1}\| \\
    &\ge  -\eta \| \bdK(\whf)-\bdK(\bdf_{\theta}^{n+1})\|  \|\bdu^{n+1}\| \\
    &\ge -\frac{\eta}{\sigma_{\!\!K}} \|(\sigma_{\!\!K}+1)\bdu^{n+1}-\bdu^n \|  \|\bdu^{n+1}\| \\
    %\ge & -\frac{\eta}{\sigma_{\!\!K}} \|\bdu^{n+1}-\bdu^n \|  \|\bdu^{n+1}\| -\eta \|\bdu^{n+1}\|^2 \\
    &\ge  -\Bigl(\frac{\eta}{2\sigma_{\!\!K}}+\eta\Bigr) \|\bdu^{n+1}\|^2 - \frac{\eta}{2\sigma_{\!\!K}}\|\bdu^{n+1}-\bdu^n \|^2.
\end{align*}
By the global Lipschitz continuity in \cref{thm:convex_lips_g},
\begin{align*}
&\langle [\nabla \bdK(\bdf_{\theta}^n)-\nabla \bdK(\bdf_{\theta}^{n+1})](\bdf^{n+1}-\bdf^n),\bdu^{n+1}\rangle 
\ge  -\frac{5L\rho_{\bdu}}{2} \|\bdf^{n+1}-\bdf^n\|^2  -\frac{L\rho_{\bdu}}{2} \|\bdf^{n-1}-\bdf^n\|^2.
\end{align*}
By \cref{lem:M_self_adjoint} and \cref{special_step_restrict}, we obtain
\begin{align*}
    \frac{1}{2}\|\bdbeta^{n+1}-\bdbeta^{n}\|_{M_{n+1}^{\one}}^{2}&  \ge \frac{\kappa}{2\tau}\|\bdf^{n+1}-\bdf^{n}\|^2 +\frac{\|\bdu^{n+1}-\bdu^{n}\|^2}{2\sigma_{\!\!K}}\biggl( 1- \frac{\tau\sigma_{\!\!K}}{s(1-\kappa)} \|\nabla \bdK(\bdf_{\theta}^n)\|^2  \biggr) \\
    & \ge \frac{\kappa}{2\tau}\|\bdf^{n+1}-\bdf^{n}\|^2  +\frac{\eta}{2\sigma_{\!\!K}}\|\bdu^{n+1}-\bdu^{n}\|^2.
\end{align*}
Hence, arranging these terms yields that
    \begin{align*}
     \langle \wtH^{\one}_{n+1}(\bdbeta^{n+1}), &\bdbeta^{n+1}-\widehat{\bdbeta}\rangle-\frac{1}{2}\|\bdbeta^{n+1}-\widehat{\bdbeta}\|_{M_{n+2}^{\one}-M_{n+1}^{\one}}^{2}+ \frac{1}{2}\|\bdbeta^{n+1}-\bdbeta^{n}\|_{M_{n+1}^{\one}}^{2}\\
\ge & \Bigl(1-\frac{\eta}{2\sigma_{\!\!K}}-\eta \Bigr)\|\bdu^{n+1}\|^2
+\frac{\kappa-5\tau L \rho_{\bdu}}{2\tau} \|\bdf^{n+1}-\bdf^n\|^{2} -\frac{L \rho_{\bdu}}{2}\|\bdf^{n-1}-\bdf^n\|^{2}.
\end{align*}
 Then by \cref{thm:basis convergence}, we obtain
\begin{align}\label{special_beta_estimate}
 &\frac{1}{2}\|\bdbeta^{0}-\widehat{\bdbeta}\|_{M_{1}^{\one}}^{2}- \frac{1}{2}\|\bdbeta^{N+1}-\widehat{\bdbeta}\|_{M_{N+2}^{\one}}^{2}\\ \notag
    \geq &   \Bigl(1-\frac{\eta}{2\sigma_{\!\!K}}-\eta\Bigr)\sum_{n=0}^N  \|\bdu^{n+1}\|^2 
    + \frac{\kappa-5\tau L\rho_{\bdu}}{2\tau} \|\bdf^{N+1}-\bdf^N\|^{2}
    + \frac{\kappa-6\tau L\rho_{\bdu}}{2\tau}\sum_{n=0}^{N-1}  \|\bdf^{n+1}-\bdf^n\|^{2} .
\end{align}
Because of the assumptions on $\tau$ and $\sigma_{\!\!K}$ in \cref{special_step_restrict}, we have
\begin{align*}
   \|\bdbeta^{0}-\widehat{\bdbeta}\|_{M_{1}^{\one}}^{2}\ge \|\bdbeta^{N+1}-\widehat{\bdbeta}\|_{M_{N+2}^{\one}}^{2}. 
\end{align*}
%Now we use mathematical induction, similar with \cref{proof:ensure_beta_O}, when $n=1$, we have
%\begin{align*}
%    \bdbeta^{1}\in \mathcal{O}(r_{\bdf}, r_{\bdu},r_{\bdv}),\quad C_{\bdf}^{0}\le C_{\bdf},\quad C_{\bdu}^{0}\le C_{\bdu},\quad C_{\bdv}^{0}\le C_{\bdv}.
%\end{align*}
%Next, we suppose that induction assumption holds with $1\le n\le N$, which is 
%\begin{align*}
%    \bdbeta^{n}\in \mathcal{O}(r_{\bdf}, r_{\bdu}, r_{\bdv}),\quad C_{\bdf}^{n-1}\le C_{\bdf},\quad C_{\bdu}^{n-1}\le C_{\bdu}, \quad C_{\bdv}^{n-1}\le C_{\bdv}, 
%\end{align*}
Similar with the proof of \cref{lemma:ensure_beta_n_neigborhood}, 
% we have 
% \begin{align*}
%     \|\bdf^{N+1}-\whf\|\le \frac{\|\bdbeta^{N+1}-\widehat{\bdbeta}\|_{M_{N+2}^{\one}}}{\sqrt{\kappa /\tau}}
%     \le\frac{\|\bdbeta^0-\widehat{\bdbeta}\|_{M_1^{\one}}}{\sqrt{\kappa/\tau}}\le r_{\bdf},
% \end{align*}
% and 
% \begin{align*}
%     \|\bdu^{N+1}-\whu\|\le \frac{\|\bdbeta^{N+1}-\widehat{\bdbeta}\|_{M_{N+2}^{\one}}}{\sqrt{\kappa /\sigma_{\!\!K}}}
%     \le\frac{\|\bdbeta^0-\widehat{\bdbeta}\|_{M_1^{\one}}}{\sqrt{\kappa /\sigma_{\!\!K}}}\le r_{\bdu},
% \end{align*}
% and $C_{\bdf}^N\le C_{\bdf}, ~C_{\bdu}^N\le C_{\bdu}$, 
we claim sequence $\{\bdbeta^n\}_{n\in\mathbb{N}}$ is always in $ \mathcal{O}(\rho_{\bdf},\rho_{\bdu},\rho_{\bdv})$ 
and $\{\bdf^n_{\theta}\}_{n\in\mathbb{N}}$ is always in $\mathcal{B}_{\whf}(\rho_{\bdf})$. 
Since \cref{special_beta_estimate} holds for all $N$, we have that $\{\bdf^n\}$ is a Cauchy sequence 
and $\bdu^{n}\rightarrow 0$. So $\bdf^n\rightarrow\whf^{\prime}$ when $n\rightarrow \infty$. Consequently, $\bdK(\whf^{\prime})=\bdg$, which is the desired result.
\end{proof}

\section*{Acknowledgments} 
The authors would like to thank Buxing Chen, Dan Xia, Zhiqiang Xu, Zheng Zhang and Yunsong Zhao for their helpful discussions.

\bibliographystyle{plain}
\bibliography{MCTreferences}

\begin{thebibliography}{10}

\bibitem{alma76}
R.~E. Alvarez and A.~Macovski.
\newblock Energy-selective reconstructions in {X}-ray computerized tomography.
\newblock {\em Phys. Med. Biol.}, 21(5):733--744, 1976.

\bibitem{BaSi21}
R.~Barber and E.~Sidky.
\newblock {Convergence for nonconvex ADMM, with applications to CT imaging}.
\newblock Technical Report 2006.07278, ArXiv e-prints, 2021.

\bibitem{be17}
A.~Beck.
\newblock {\em {First-Order Methods in Optimization}}.
\newblock Society for Industrial and Applied Mathematics, Philadelphia, PA,
  2017.

\bibitem{chpo11}
A.~Chambolle and T.~Pock.
\newblock A first-order primal-dual algorithm for convex problems with
  applications to imaging.
\newblock {\em J. Math. Imaging Vison}, 40:120--145, 2011.

\bibitem{chpan20}
B.~Chen, L.~Xin, Z.~Zhang, D.~Xia, E.~Sidky, and X.~Pan.
\newblock Optimization-based algorithm for solving the discrete x-ray transform
  with nonlinear partial volume effect.
\newblock {\em Journal of medical imaging}, 7(5):053502, 2020.

\bibitem{chpa18}
B.~Chen, Z.~Zhang, D.~Xia, E.~Sidky, and X.~Pan.
\newblock {Non-Convex Chambolle--Pock Algorithm for Multispectral CT}.
\newblock In {\em Proceedings of The Fifth International Conference on Image
  Formation in X-Ray Computed Tomography}, pages 377--381, 2018.

\bibitem{chpan21}
B.~Chen, Z.~Zhang, D.~Xia, E.~Sidky, and X.~Pan.
\newblock Non-convex primal-dual algorithm for image reconstruction in spectral
  {CT}.
\newblock {\em Computerized Medical Imaging and Graphics}, 87:101821, 2021.

\bibitem{chen2020fast}
C.~Chen, R.~Wang, C.~Bajaj, and O.~{\"O}ktem.
\newblock {An Efficient Algorithm to Compute the X-ray Transform}.
\newblock {\em International Journal of Computer Mathematics},
  DOI:10.1080/00207160.2021.1969017, 2021.

\bibitem{accle_Val19}
C.~Clason, S.~Mazurenko, and T.~Valkonen.
\newblock Acceleration and global convergence of a first-order primal-dual
  method for nonconvex problems.
\newblock {\em SIAM Journal on Optimization}, 29(1):933--963, 2019.

\bibitem{El08}
M.~Elad.
\newblock {\em {Sparse and Redundant Representations: From Theory to
  Applications in Signal and Image Processing}}.
\newblock Springer, 2010.

\bibitem{eszhch10}
E.~Esser, X.~Zhang, and T.~F. Chan.
\newblock {A General Framework for a Class of First Order Primal-Dual
  Algorithms for Convex Optimization in Imaging Science}.
\newblock {\em SIAM J. Imaging Sci.}, 3(4):1015--1046, 2010.

\bibitem{FaYu05}
J.~Fan and Y.~Yuan.
\newblock {On the Quadratic Convergence of the Levenberg--Marquardt Method
  without Nonsingularity Assumption}.
\newblock {\em Computing}, 74(1):23--39, 2005.

\bibitem{heyu2012}
B.~He and X.~Yuan.
\newblock {Convergence Analysis of Primal-Dual Algorithms for a Saddle-Point
  Problem: From Contraction Perspective}.
\newblock {\em SIAM J. Imaging Sci.}, 5(1):119--149, 2012.

\bibitem{sgd_for_ill_posed}
B.~Jin, Z.~Zhou, and J.~Zou.
\newblock On the convergence of stochastic gradient descent for nonlinear
  ill-posed problems.
\newblock {\em SIAM Journal on Optimization}, 30(2):1421--1450, 2020.

\bibitem{kanesch08}
B.~Kaltenbacher, A.~Neubauer, and O.~Scherzer.
\newblock {\em Iterative regularization methods for nonlinear ill-posed
  problems}.
\newblock Radon Series on Computational and Applied Mathematics. Walter de
  Gruyter, 2008.

\bibitem{lpna17}
G.~Landi, E.~Loli~Piccolomini, and J.~G. Nagy.
\newblock {A limited memory BFGS method for a nonlinear inverse problem in
  digital breast tomosynthesis}.
\newblock {\em Inverse Problems}, 33:095005 (21pp), 2017.

\bibitem{maal76}
A.~Macovski, R.~E. Alvarez, J.~L.-H. Chan, J.~P. Stonestrom, and L.~M. Zatz.
\newblock Energy dependent reconstruction in {X}-ray computerized tomography.
\newblock {\em Comput. Biol. Med.}, 6:325--336, 1976.

\bibitem{rowe04}
R.~Rockafellar and R.~Wets.
\newblock {\em Variational Analysis}, volume 317.
\newblock Springer-Verlag, 2004.

\bibitem{RuOsFa92}
L.~Rudin, S.~J. Osher, and E.~Fatemi.
\newblock {Nonlinear total variation based noise removal algorithms}.
\newblock {\em Physica D: Nonlinear Phenomena}, 60(1):259--268, 1992.

\bibitem{va14}
T.~Valkonen.
\newblock {A primal-dual hybrid gradient method for nonlinear operators with
  applications to MRI}.
\newblock {\em Inverse Problems}, 30(5):055012 (45pp), 2014.

\bibitem{VaPa16}
W.~van Aarle, W.~J. Palenstijn, J.~Cant, E.~Janssens, F.~Bleichrodt,
  A.~Dabravolski, J.~De~Beenhouwer, K.~J. Batenburg, and J.~Sijbers.
\newblock {Fast and Flexible X-ray Tomography Using the ASTRA Toolbox}.
\newblock {\em Optics Express}, 24(22):25129--25147, 2016.

\bibitem{WaXu19}
Y.~Wang and Z.~Xu.
\newblock {Generalized phase retrieval: measurement number, matrix recovery and
  beyond}.
\newblock {\em Appl. Comput. Harmon. Anal.}, 47(2):423--446, 2019.

\bibitem{WaYiZe19}
Y.~Wang, W.~Yin, and J.~Zeng.
\newblock {Global Convergence of ADMM in Nonconvex Nonsmooth Optimization}.
\newblock {\em Journal of Scientific Computing}, 78:29--63, 2019.

\bibitem{wang2004image}
Z.~Wang, A.~C. Bovik, H.~R. Sheikh, E.~P. Simoncelli, et~al.
\newblock Image quality assessment: from error visibility to structural
  similarity.
\newblock {\em IEEE transactions on image processing}, 13(4):600--612, 2004.

\bibitem{zhaozz14}
Y.~Zhao, X.~Zhao, and P.~Zhang.
\newblock {An extended algebraic reconstruction technique (E-ART) for dual
  spectral CT}.
\newblock {\em IEEE Transactions on Medical Imaging}, 34(3):761--768, 2015.

\end{thebibliography}

\end{document}